\documentclass[a4paper,12pt]{article}

\usepackage{amsmath,amsthm,amssymb}
\usepackage{pb-diagram}
\usepackage{ifthen}

\newcommand{\lowercasegreek}[1]{
\ifthenelse{\equal{#1}{a}}{\alpha}{
\ifthenelse{\equal{#1}{b}}{\beta}{
\ifthenelse{\equal{#1}{g}}{\gamma}{
\ifthenelse{\equal{#1}{d}}{\delta}{
\ifthenelse{\equal{#1}{e}}{\varepsilon}{
\ifthenelse{\equal{#1}{z}}{\zeta}{
\ifthenelse{\equal{#1}{l}}{\lambda}{
\ifthenelse{\equal{#1}{m}}{\mu}{
\ifthenelse{\equal{#1}{n}}{\nu}{
\ifthenelse{\equal{#1}{x}}{\xi}{
\ifthenelse{\equal{#1}{o}}{\omega}{
\ifthenelse{\equal{#1}{r}}{\rho}{
\ifthenelse{\equal{#1}{s}}{\sigma}{
\ifthenelse{\equal{#1}{u}}{\upsilon}{
\ifthenelse{\equal{#1}{f}}{\varphi}{
\ifthenelse{\equal{#1}{h}}{\chi}{
\ifthenelse{\equal{#1}{p}}{\psi}{
\ifthenelse{\equal{#1}{t}}{\theta}{
\ifthenelse{\equal{#1}{k}}{\kappa}{
\ifthenelse{\equal{#1}{i}}{\iota}{
\ifthenelse{\equal{#1}{t1}}{\tau}{
\ifthenelse{\equal{#1}{e1}}{\eta}{error\_greek
}}}}}}}}}}}}}}}}}}}}}}}

\newcommand{\thrf}[1]{\textbf{#1}}
\newcommand{\ordf}[1]{\lowercasegreek{#1}}
\newcommand{\onf}[1]{\lowercasegreek{#1}}
\newcommand{\natf}[1]{\textit{\textsf{#1}}}
\newcommand{\nattrf}[1]{\lowercasegreek{#1}}
\newcommand{\elf}[1]{\textit{#1}}
\newcommand{\prefixf}[1]{\ifthenelse{\equal{#1}{P}}{\mathit{\Gamma}}{\ifthenelse{\equal{#1}{W}}{\mathit{\Delta}}{prefixf\_error}}}
\newcommand{\fovarf}[1]{\textbf{#1}}
\newcommand{\prvarf}[1]{\textbf{#1}}
\newcommand{\objf}[1]{\textit{#1}}
\newcommand{\algf}[1]{\mathbf{#1}}
\newcommand{\setf}[1]{\textit{#1}}
\newcommand{\csetf}[1]{\textit{#1}}
\newcommand{\isetf}[1]{\textit{#1}}
\newcommand{\foflf}[1]{\lowercasegreek{#1}}
\newcommand{\prflf}[1]{\lowercasegreek{#1}}
\newcommand{\typef}[1]{\textsf{\textmd{#1}}}
\newcommand{\tembf}[1]{\mathfrak{#1}}
\newcommand{\funf}[1]{\textit{#1}}
\newcommand{\frf}[1]{\mathcal{#1}}
\newcommand{\termf}[1]{\textbf{#1}}
\newcommand{\constf}[1]{\textbf{#1}}
\newcommand{\symf}[1]{\textbf{#1}}

\newcommand{\cmpf}[1]{\texttt{#1}}
\newcommand{\seqf}[1]{\textbf{#1}}
\newcommand{\seqelf}[1]{\textit{#1}}
\newcommand{\extsqtf}[1]{\textbf{#1}}

\newcommand{\typeplus}{+}

\newcommand{\baz}{\mathbf{0}}
\newcommand{\bau}{\mathbf{1}}
\newcommand{\baand}{\boldsymbol{\cdot}}
\newcommand{\baor}{\boldsymbol{+}}
\newcommand{\bacmp}[1]{\mbox{\textbf{-}}#1}
\newcommand{\bacmpi}[2]{\mbox{\textbf{-}}^{#1}#2}
\newcommand{\bacmps}{\mbox{\textbf{-}}}

\newcommand{\boxo}[2][error]{\boldsymbol{\tau}_{#1}(#2)}

\newcommand{\boxos}[1][error]{\boldsymbol{\tau}_{#1}}

\newcommand{\diamo}[2][error]{\boldsymbol{d}_{#1}(#2)}
\newcommand{\diamoi}[3][error]{\boldsymbol{d}_{#1}^{#2}(#3)}
\newcommand{\diamos}[1][error]{\boldsymbol{d}_{#1}}
\newcommand{\diamosi}[2][error]{\boldsymbol{d}_{#1}^{#2}}
\newcommand{\algl}{\triangleleft}
\newcommand{\algs}{\bowtie}
\newcommand{\ione}{0}
\newcommand{\itwo}{1}

\newcommand{\boxm}[1][0]{[#1]}
\newcommand{\diamm}[1][0]{\langle#1\rangle}
\newcommand{\sppvar}[1]{v^{#1}}
\newcommand{\sppconst}[1]{u^{#1}}
\newcommand{\trtfl}[1]{{#1}^{\star}}
\newcommand{\foshift}{\mathrm{SF}}
\newcommand{\prshift}{\mathrm{SP}}
\newcommand{\sqtr}[3]{\mathrm{SQ}_{#1}^{#2}(#3)}
\newcommand{\sqtrs}[2]{\mathrm{SQ}_{#1}^{#2}}
\newcommand{\sqftr}[3]{\mathrm{SQF}_{#1}^{#2}(#3)}
\newcommand{\sqftrs}[2]{\mathrm{SQF}_{#1}^{#2}}
\newcommand{\sqfatr}[2]{\mathrm{SQF}_{#1}(#2)}
\newcommand{\sqfatrs}[1]{\mathrm{SQF}_{#1}}
\newcommand{\doubleprefix}[1]{#1^{+}}
\newcommand{\negprefix}[1]{\overline{#1}}
\newcommand{\quant}{\mathbf{Q}}
\newcommand{\formeq}{\leftrightharpoons}
\newcommand{\foor}{\lor}
\newcommand{\foand}{\land}
\newcommand{\foimp}{\to}
\newcommand{\fonot}{\lnot}
\newcommand{\primp}{\to}
\newcommand{\prand}{\land}
\newcommand{\pror}{\lor}
\newcommand{\prnot}{\lnot}
\newcommand{\prequiv}{\;\leftrightarrow\;}
\newcommand{\freevariables}[1]{\mathrm{FV}(#1)}

\newcommand{\elthr}[2][]{\thrf{Th}_{#1}(#2)}

\newcommand{\dom}{\textbf{dom}}
\newcommand{\ran}{\textbf{ran}}
\newcommand{\codom}{\textbf{codom}}
\newcommand{\comp}{\circ}

\newcommand{\prodfr}{\times}
\newcommand{\linprodfr}{\otimes}
\newcommand{\purfr}{\frf{P}}
\newcommand{\fextfr}{\frf{E}}
\newcommand{\fquotalg}{\frf{Q}}
\newcommand{\squotalg}{\frf{R}}
\newcommand{\fextfquotfr}{\frf{S}}

\newcommand{\sqext}{\algf{Y}}
\newcommand{\sqqlp}{\algf{U}}
\newcommand{\sqqlpm}{\nattrf{z}}
\newcommand{\sqqlpam}{\nattrf{h}}

\newcommand{\idf}{\mathrm{id}}
\newcommand{\fqprsqiso}{\nattrf{e1}}
\newcommand{\fqlprsqiso}{\nattrf{l}}
\newcommand{\fextemb}{\nattrf{e}}

\newcommand{\fquotalgm}{\nattrf{e1}}
\newcommand{\fextfquotm}{\nattrf{i}}

\newcommand{\slangat}[1]{\textit{LA}_{#1}}
\newcommand{\slangate}[1]{\textit{LAE}_{#1}}
\newcommand{\slangprn}[1]{\textit{LP}_{#1}}
\newcommand{\slangn}[2]{\textit{L}_{#1}^{#2}}
\newcommand{\slangatc}[1]{\textit{LAC}_{#1}}

\newcommand{\slangprnc}[1]{\textit{LPC}_{#1}}
\newcommand{\slangnc}[2]{\textit{LC}_{#1}^{#2}}

\newcommand{\Implication}{\;\Rightarrow\;}
\newcommand{\defiff}{\stackrel{\mathrm{def}}{\iff}}

\newcommand{\lang}[1]{\mathcal{L}(#1)}
\newcommand{\sptype}[1]{\typef{Y}^{#1}}
\newcommand{\spemb}[2]{\tembf{h}^{#1}_{#2}}

\newtheorem{proposition}{Proposition}
\newtheorem{lemma}{Lemma}
\newtheorem{theorem}{Theorem}
\newtheorem{corollary}{Corollary}

\begin{document}
\title{On Elementary Theories of \thrf{GLP}-Algebras}
\author{Fedor~Pakhomov\thanks{This work was partially supported by RFFI grant 12-01-00888\_a and Dynasty foundation.}\\Steklov Mathematical Institute,\\Moscow\\ \texttt{pakhfn@mi.ras.ru}}
\date{December 2014}
\maketitle

\begin{abstract} There is a polymodal provability logic $\thrf{GLP}$. We consider generalizations of this logic:  the logics $\thrf{GLP}_{\ordf{a}}$, where $\ordf{a}$ ranges over linear ordered sets and play the role of the set of indexes of modalities. We consider the varieties of modal algebras that corresponds to the polymodal logics. We prove that the elementary theories of the free $\emptyset$-generated $\thrf{GLP}_{\natf{n}}$-algebras are decidable for all finite ordinals $\natf{n}$. 
\end{abstract}

\section{Introduction}
There is a classical modal logic $\thrf{GL}$, it can be axiomatized over $\thrf{K}$ by the axiom scheme $\Box (\Box \prflf{f} \primp \prflf{f})\primp \Box \prflf{f}$. R.M.~Solovay have proved \cite{Sol76} that the logic $\thrf{GL}$ proves a formula iff formal arithmetics $\thrf{PA}$ proves every arithmetical interpretation of the formula.An arithmetical interpretation of modal formulas interprets variables by arbitrary  arithmetical sentences, commute with propositional connectives, and  interprets $\Box\prflf{f}$ by arithmetical sentence that means ``$\thrf{PA}$ prove the interpretation of $\prflf{f}$''. 

G.K.~Japaridze  have introduced polymodal provability logic $\thrf{GLP}$\cite{Jap86}. The modalities of the logic $\thrf{GLP}$ are $\boxm[0],\boxm[1],\ldots$. There is an analogue of Solovay theorem for the logic $\thrf{GLP}$ \cite{Jap86} (there is a more modern variant of the result in \cite{Bek11}). 

There were several research on closed fragment of $\thrf{GLP}$, i.e. the fragment consists of all formulas without variables \cite{Ign93}\cite{BJV05}. There were simple representation of an universal model for the closed fragment of $\thrf{GLP}$. 

There are generalization of the logics $\thrf{GLP}$ --- the logics $\thrf{GLP}_{\ordf{a}}$, where $\ordf{a}$  are linear ordered sets that are sets of index of modalities \cite{BekFerJoo13}; the standard logic $\thrf{GLP}$ is the same as $\thrf{GLP}_{\omega}$. In \cite{FerJoo13} it were shown that the construction of universal model for the logic $\thrf{GLP}$ can be generalized to the case of the logics $\thrf{GLP}_{\ordf{a}}$, when $\ordf{a}$ is an ordinal.  

For every modal logic there is the corresponding variety of modal algebras. The free algebra of the variety with the set of generators$\setf{A}$ is the same as the Lindenbaum-Tarski algebra for the fragment of the logic with variables restricted to some set of variables indexed by elements of $\setf{A}$. In particular $\emptyset$-generated algebra is the same as the Lindenbaum-Tarski algebra for the closed fragment.

The decidability of elementary problem is classical for model theory. S.N.~Artemov and L.D.~Beklemishev have proved that for finite $\csetf{C}$ the elementary theory of free $\csetf{C}$-generated $\thrf{GL}$-algebra is decidable iff $\csetf{C}=\emptyset$ \cite{ArtBek93}. L.D.~Beklemishev have asked the question about the decidability of free $\emptyset$-generated $\thrf{GLP}$-algebra \cite[Problem 33]{BekVis06}. We prove that the free $\thrf{GLP}_{\natf{n}}$-algebra have decidable elementary theory for every $\natf{n}$. 

In the paper we introduce the notion of linear $\thrf{GLP}$-algebra that generalize  the notion of free $\emptyset$-generated $\thrf{GLP}$-algebra. We prove that every   free $\emptyset$-generated $\thrf{GLP}_{\ordf{a}}$-algebra is linear. We introduce operation of linear product of $\thrf{GLP}_{\ordf{a}}$-algebras. We consider some decompositions of the free $\emptyset$-generated $\thrf{GLP}_{\ordf{a}}$-algebras with respect to the operation of linear product. We use this decompositions in our proof of the decidability of elementary theories of $\thrf{GLP}_{\natf{n}}$-algebras.
\section{\thrf{GLP}-Algebras}
In this section we give the notion of a \thrf{GLP}-algebra with a given set of modalities and constants. The only algebras we consider are $\thrf{GLP}$-algebras; thus we omit \thrf{GLP} in  ``\thrf{GLP}-algebras'' and write ``algebras''. 

Underlying formalism of our work is set-theoretic. We assume that there is the proper class of constant symbols. We have a unique unary functional symbol $\boxos[\objf{a}]$ for every  set $\objf{a}$. 

Suppose we have a pair $\typef{A}=(\ordf{a},\csetf{A})$, where $\ordf{a}$ is a strict linear order $(\setf{D}_{\ordf{a}},<_{\ordf{a}})$ and  $\csetf{A}$ is a set of constant symbols such that symbols $\baz,\bau\not\in\csetf{A}$. We call such a pair an {\it algebra type} (or shorter {\it type}). $\algf{A}$ is a {\it $\thrf{GLP}$-algebra of the type $\typef{A}$} (or shorter {\it $\typef{A}$-algebra}) if $\algf{A}$ is a model of the signature $$\{\baz,\bau,\baand,\baor,\bacmps\}\sqcup\{\diamos[\objf{i}]\mid\objf{i}\in\ordf{a}\}\sqcup \{\constf{c}\mid\constf{c}\in\csetf{A}\}$$ 
such that $\algf{A}$ is a Boolean algebra and satisfies the following axioms
\begin{enumerate}
\item \label{GLPA_Ax1} $\diamo[\objf{i}]{\baz}=\baz$, for $\objf{i}\in \ordf{a}$; 
\item \label{GLPA_Ax2} $\diamo[\objf{i}]{\elf{x}}\baor\diamo[\objf{i}]{\elf{y}}=\diamo[\objf{i}]{\elf{x}\baor\elf{y}}$, for $\objf{i}\in \ordf{a}$;
\item \label{GLPA_Ax3} $\diamo[\objf{i}]{\bacmp\diamo[\objf{i}]{\elf{x}}\baand\elf{x}}=\diamo[\objf{i}]{\elf{x}}$, где $\objf{i}\in \ordf{a}$;
\item \label{GLPA_Ax4} $\diamo[\objf{j}]{\elf{x}}\le\diamo[\objf{i}]{\elf{x}}$, for $\objf{i},\objf{j}\in\ordf{a}$, $\objf{i}<_{\ordf{a}}\objf{j}$;
\item \label{GLPA_Ax5} $\diamo[\objf{i}]{\elf{x}}\le\boxo[\objf{j}]{\diamo[\objf{i}]{\elf{x}}}$, for $\objf{i},\objf{j}\in\ordf{a}$, $\objf{i}<_{\ordf{a}}\objf{j}$.
\end{enumerate} 
Note that $\elf{x}\le\elf{y}$ is an abbreviation for $\elf{x}=\elf{x}\baand \elf{y}$ and unary functions $\boxos[\objf{i}]$ are given by $$\boxo[\objf{i}]{\elf{x}}=\bacmp{\diamo[\objf{i}]{\bacmp{\elf{x}}}}.$$ We denote the first-order theory of $\typef{A}$-algebras by $\thrf{GLPA}_{\typef{A}}$.

A simple check shows that
\begin{lemma} \label{GLPA_Eq1}
Suppose $\typef{A}=(\ordf{a},\csetf{A})$ is a type. Then the following equations holds in all $\typef{A}$-algebras:
\begin{enumerate}
\item \label{GLPA_Eq1_1}$ \diamo[\objf{i}]{\diamo[\objf{i}]{\elf{x}}}\le\diamo[\objf{i}]{\elf{x}}$, for $\objf{i}\in\ordf{a}$;
\item \label{GLPA_Eq1_2}$\diamo[\objf{j}]{\elf{x} \baand \diamo[\objf{i}]{\elf{y}}}=\diamo[\objf{j}]{\elf{x}} \baor \diamo[\objf{i}]{\elf{y}}$, for $\objf{i},\objf{j}\in\ordf{a}$, $\objf{i}<_{\ordf{a}}\objf{j}$;
\item \label{GLPA_Eq1_3}$\diamo[\objf{j}]{\elf{x} \baand  \boxo[\objf{i}]{\elf{y}}}=\diamo[\objf{j}]{\elf{x}} \baand \boxo[\objf{i}]{\elf{y}}$, for $\objf{i},\objf{j}\in\ordf{a}$, $\objf{i}<_{\ordf{a}}\objf{j}$;
\item \label{GLPA_Eq1_4}$\diamo[\objf{j}]{\elf{x} \baor \diamo[\objf{i}]{\elf{y}}}\baor \diamo[\objf{i}]{\elf{y}}=\diamo[\objf{j}]{\elf{x}} \baor \diamo[\objf{i}]{\elf{y}}$, for $\objf{i},\objf{j}\in\ordf{a}$, $\objf{i}<_{\ordf{a}}\objf{j}$;
\item \label{GLPA_Eq1_5}$\diamo[\objf{j}]{\elf{x} \baor \boxo[\objf{i}]{\elf{y}}}\baor \boxo[\objf{i}]{\elf{y}}=\diamo[\objf{j}]{\elf{x}} \baor \boxo[\objf{i}]{\elf{y}}$, for $\objf{i},\objf{j}\in\ordf{a}$, $\objf{i}<_{\ordf{a}}\objf{j}$;
\end{enumerate}
\end{lemma}

$\thrf{GLP}$-algebras are related to the logic $\thrf{GLP}$. The axioms of $\thrf{GLP}$-algebras are axioms of the logic $\thrf{GLP}$ ``translated'' to the language of Boolean algebras with additional operators. Classically logic $\thrf{GLP}$  is defined as a polymodal logic with modalities indexed by natural numbers. We index modalities by elements of an arbitrary linear ordered set and we have unique propositional variable $\sppvar{\constf{c}}$ for every constant symbol $\constf{c}$ and $\sppvar{\fovarf{x}}$ for every first-order variable $\fovarf{x}$. Suppose $\ordf{a}$ is a strict linear order. The set $\lang{\thrf{GLP}_{\ordf{a}}}$ of well-formed formulas of the logic $\thrf{GLP}_{\ordf{a}}$ is given inductively by
$$\begin{aligned}\mbox{$\thrf{GLP}_{\ordf{a}}$-Form}::= & \mbox{Propositional variable}\mid \mbox{Propositional constant}\mid \\ &  \top\mid \bot\mid   \mbox{$\thrf{GLP}_{\ordf{a}}$-Form}\prand \mbox{$\thrf{GLP}_{\ordf{a}}$-Form}\mid  \mbox{$\thrf{GLP}_{\ordf{a}}$-Form}\pror \mbox{$\thrf{GLP}_{\ordf{a}}$-Form}  \mid \\ &  \mbox{$\thrf{GLP}_{\ordf{a}}$-Form} \primp \mbox{$\thrf{GLP}_{\ordf{a}}$-Form} \mid  \prnot \mbox{$\thrf{GLP}_{\ordf{a}}$-Form} \mid  \boxm[\objf{i}] \mbox{$\thrf{GLP}_{\ordf{a}}$-Form},\\ & \mbox{where }\objf{i}\in\ordf{a}.\end{aligned}$$
For an index $\objf{i}$ and a formula $\prflf{p}$ we write $\diamm[\objf{i}]\prflf{p}$ for $\prnot \boxm[\objf{i}]\prnot\prflf{p}$. The axioms and inference rules of the logic $\thrf{GLP}_{\ordf{a}}$ are 
\begin{enumerate}
\item the axiom schemes of $\thrf{PC}$;
\item $\boxm[\objf{i}](\prflf{f}\primp \prflf{p})\primp (\boxm[\objf{i}] \prflf{f} \primp \boxm[\objf{i}] \prflf{p})$;
\item $\boxm[\objf{i}](\boxm[\objf{i}]\prflf{p}\primp \prflf{p})\primp \boxm[\objf{i}]\prflf{p}$;
\item $\boxm[\objf{i}]\prflf{p}\primp \boxm[\objf{j}]\prflf{p}$, for $\objf{i}<_{\ordf{a}}\objf{j}$;
\item $\diamm[\objf{i}]\prflf{p}\primp \boxm[\objf{j}]\diamm[\objf{i}]\prflf{p}$, for $\objf{i}<_{\ordf{a}}\objf{j}$;
\item $\frac{\prflf{f}\;\prflf{f}\primp\prflf{p}}{\prflf{p}}$ (Modus Ponens);
\item $\frac{\prflf{f}}{\boxm[\objf{i}]\prflf{f}}$ (Generalization);
\item $\frac{\prflf{f}}{\prflf{f}[\prflf{p}/\prvarf{x}]}$, where $\prvarf{x}$ is a propositional variable.
\end{enumerate}

There are correspondence between the logic $\thrf{GLP}_{\ordf{a}}$ and the theory of $(\ordf{a},\csetf{A})$-algebras. 

We use propositional constants $\sppconst{\constf{c}}$ for all constant symbols $\constf{c}$ and propositional variables $\sppvar{\fovarf{x}}$ for all propositional variables $\fovarf{x}$. We consider the class of all terms that all functional symbols in them are either from the signature of boolean algebras or of the form $\diamos[\objf{x}]$. We give a translation $\termf{t}\longmapsto\trtfl{\termf{t}}$ of the terms of the class to modal formulas:
\begin{enumerate}
\item $\trtfl{\baz}=\bot$;
\item $\trtfl{\bau}=\top$;
\item $\trtfl{\fovarf{x}}=\sppvar{\fovarf{x}}$, for a first order variable $\fovarf{x}$;
\item $\trtfl{\constf{c}}=\sppconst{\constf{c}}$, где $\constf{c}\ne\baz$, $\constf{c}\ne \bau$, $\constf{c}$ is a constant symbol;
\item $\trtfl{(\termf{t}_1\baand \termf{t}_2)}=\trtfl{\termf{t}_1}\prand \trtfl{\termf{t}_2}$;
\item $\trtfl{(\termf{t}_1\baor \termf{t}_2)}=\trtfl{\termf{t}_1}\pror \trtfl{\termf{t}_2}$;
\item $\trtfl{(\bacmp{\termf{t}})}=\prnot(\trtfl{\termf{t}})$;
\item $\trtfl{(\diamo[\objf{x}]{\termf{t}})}=\diamm[\objf{x}](\trtfl{\termf{t}})$.
\end{enumerate}

\begin{lemma}\label{fo-mod_corr} Suppose $(\ordf{a},\csetf{A})$ is a type, $\termf{t}$, $\termf{u}$  are $\thrf{GLPA}_{(\ordf{a},\csetf{A})}$-terms and  $\{\termf{v}_{\objf{i}}\mid \objf{i}\in\isetf{I}\}$, $\{\termf{w}_{\objf{i}}\mid \objf{i}\in\isetf{I}\}$ are families of $\thrf{GLPA}_{(\ordf{a},\csetf{A})}$-terms. Then 
$$\thrf{GLPA}_{(\ordf{a},\csetf{A})}+\{\termf{v}_{\objf{i}}=\termf{u}_{\objf{i}}\mid \objf{i}\in\setf{I}\}\vdash \termf{t}=\termf{u}$$
iff
$$\thrf{GLP}_{\ordf{a}}+\{\trtfl{\termf{v}_{\objf{i}}}\prequiv\trtfl{\termf{w}_{\objf{i}}}\mid \objf{i}\in \isetf{I}\}\vdash  \trtfl{\termf{t}}\prequiv\trtfl{\termf{u}}.$$\end{lemma}
\begin{proof}\textit{(Sketch)} All axioms of the theory $\thrf{GLPA}_{(\ordf{a},\csetf{A})}+\{\termf{v}_{\objf{i}}=\termf{u}_{\objf{i}}\mid \objf{i}\in\isetf{I}\}$ are equations. It is well-known (some form of the following fact is due to Birkhoff \cite{Bir35}, also it can be found in the textbook \cite[II,§§14]{BurSan81})  that for a theory axiomatizable by equations all first-order theorems that are equations can be deduced from the axioms by the following rules:
\begin{itemize}
\item $\frac{}{\termf{t}=\termf{t}}$ (Reflexivity);
\item $\frac{\termf{t}_1=\termf{t}_2\;\;\termf{t}_2=\termf{t}_3}{\termf{t}_1=\termf{t}_3}$ (Transitivity);
\item $\frac{\termf{t}_1=\termf{t}_2}{\termf{t}_2=\termf{t}_1}$ (Symmetricity);
\item $\frac{\termf{t}_1=\termf{t}_2}{\termf{t}_1=\termf{t}_2[\termf{t}_3/\fovarf{x}]}$ (Substitution);
\item $\frac{\termf{t}_1=\termf{t}_2[\termf{t}_3/\fovarf{x}]\;\;\termf{t}_3=\termf{t}_4}{\termf{t}_1=\termf{t}_2[\termf{t}_4/\fovarf{x}]}$ (Replacement).
\end{itemize} In this lemma we consider that type of derivations for $\thrf{GLPA}_{(\ordf{a},\csetf{A})}+\{\termf{v}_{\objf{i}}=\termf{u}_{\objf{i}}\mid \objf{i}\in\isetf{I}\}$.

Both 'if' and 'only if' parts of the lemma can be proved by straightforward induction on the length of derivations.\end{proof}

The following form of deduction theorem holds for the logic $\thrf{GLP}$
\begin{lemma}\label{glp_deduction} Suppose $\typef{A}=(\ordf{a},\csetf{A})$ is a type. Then for $\thrf{GLP}_{\ordf{a}}$-formulas $\prflf{f},\prflf{p}_1,\ldots,\prflf{p}_{\natf{n}}$   without free variables there exists $\objf{x}_0\in\ordf{a}$ such that for all $\objf{x}\le_{\ordf{a}}\objf{x}_0$ the following conditions are equivalent:
\begin{enumerate}
\item \label{glp_deduction_c1}$\thrf{GLP}_{\ordf{a}} + \prflf{p}_1+\ldots+\prflf{p}_{\natf{n}}\vdash \prflf{f}$,
\item \label{glp_deduction_c2}$\thrf{GLP}_{\ordf{a}} \vdash (\prflf{p}_1\prand\ldots\prand\prflf{p}_{\natf{n}})\prand \boxm[\objf{x}] (\prflf{p}_1\prand\ldots\prand\prflf{p}_{\natf{n}}) \primp \prflf{f}$.
\end{enumerate} 
\end{lemma} 
\begin{proof} Obviously, from the condition \ref{glp_deduction_c2}  it follows  the condition \ref{glp_deduction_c1}.
 
By induction on a length of a proof we prove that for any $\thrf{GLP}_{\ordf{a}}$-formulas $\prflf{f},\prflf{p}_1,\ldots,\prflf{p}_{\natf{n}}$  if
$$\thrf{GLP}_{\ordf{a}} + \prflf{p}_1+\ldots+\prflf{p}_{\natf{n}}\vdash \prflf{f}$$
 then there exists   $\objf{x}_0\in\ordf{a}$ such that for all $\objf{x}\le_{\ordf{a}}\objf{x}_0$
$$\thrf{GLP}_{\ordf{a}} \vdash (\prflf{p}_1\prand\ldots\prand\prflf{p}_{\natf{n}})\prand \boxm[\objf{x}] (\prflf{p}_1\prand\ldots\prand\prflf{p}_{\natf{n}}) \primp \prflf{f}.$$ The induction is almost the same as the induction in the classical proof of the deduction theorem for propositional calculus. The only essential difference is the case when  $\prflf{f}$ is $\boxm[\objf{y}]\prflf{x}$ and the last rule in the proof of  $\prflf{f}$ in $\thrf{GLP}_{\ordf{a}} + \prflf{p}_1+\ldots+\prflf{p}_{\natf{n}}$ is $\frac{\prflf{x}}{\boxm[\objf{y}]\prflf{x}}$. From induction hypothesis it follows that  there exists   $\objf{x}_0\le_{\ordf{a}}\objf{y}$ such that for all $\objf{x}\le_{\ordf{a}}\objf{x}_0$ $$\thrf{GLP}_{\ordf{a}} \vdash (\prflf{p}_1\prand\ldots\prand\prflf{p}_{\natf{n}})\prand \boxm[\objf{x}] (\prflf{p}_1\prand\ldots\prand\prflf{p}_{\natf{n}}) \primp \prflf{x}.$$ Hence for all $\objf{x}\le_{\ordf{a}}\objf{x}_0$ $$\thrf{GLP}_{\ordf{a}} \vdash \boxm[\objf{y}]((\prflf{p}_1\prand\ldots\prand\prflf{p}_{\natf{n}})\prand \boxm[\objf{x}] (\prflf{p}_1\prand\ldots\prand\prflf{p}_{\natf{n}}) \primp \prflf{x}).$$ Because $\objf{x}_0\le_{\ordf{a}}\objf{y}$ for all $\objf{x}\le_{\ordf{a}}\objf{x}_0$  $$\thrf{GLP}_{\ordf{a}} \vdash \boxm[\objf{y}]\boxm[\objf{x}] (\prflf{p}_1\prand\ldots\prand\prflf{p}_{\natf{n}}) \primp \boxm[\objf{y}]((\prflf{p}_1\prand\ldots\prand\prflf{p}_{\natf{n}})\foand \boxm[\objf{x}] (\prflf{p}_1\prand\ldots\prand\prflf{p}_{\natf{n}}))$$. Hence for all $\objf{x}\le_{\ordf{a}}\objf{x}_0$ $$\thrf{GLP}_{\ordf{a}} \vdash (\prflf{p}_1\prand\ldots\prand\prflf{p}_{\natf{n}})\prand \boxm[\objf{x}] (\prflf{p}_1\prand\ldots\prand\prflf{p}_{\natf{n}}) \primp \boxm[\objf{y}]\prflf{x}.$$
\end{proof}

 A tuple $\tembf{l}=(\typef{A},\funf{f},\funf{g},\typef{B})$ is a {\it type embedding} if 
\begin{enumerate}
\item $\typef{A}=(\ordf{a},\csetf{A})$ is a type; 
\item $\typef{B}=(\ordf{b},\csetf{B})$ is a type;
\item $\funf{f}\colon \ordf{a}\to\ordf{b}$ is a strictly monotone function, i.e.
$$\forall \objf{x},\objf{y}\in \ordf{a}( \objf{x}<_{\ordf{a}}\objf{y} \Implication \funf{f}(\objf{x})<_{\ordf{b}}\funf{g}(\objf{y});$$
\item $\ordf{g}\colon \csetf{A}\to\csetf{B}$ is an injection.
\end{enumerate}
We say that $\tembf{l}$ is an embedding of $\typef{A}$ into $\typef{B}$, $\typef{A}$ is the {\it domain} of $\tembf{l}$ and that $\typef{B}$ is the {\it codomain} of $\tembf{l}$. For a $\tembf{l}\colon\typef{A}\to\typef{B}$, $\dom(\tembf{l})=\typef{A}$ and $\codom(\tembf{l})=\typef{B}$. We call a type  embedding $(\typef{A},\funf{f},\funf{g},\typef{B})$ a {\it trivial type embedding} if $\funf{f}$ and $\funf{g}$ maps every $\objf{x}$ to itself.

Suppose $\tembf{l}_1=(\typef{A},\funf{f}_1,\funf{g}_1,\typef{B})$ and $\tembf{l}_2=(\typef{B},\funf{f}_2,\funf{g}_2,\typef{C})$ are type embedding. We denote by $\tembf{l}_1\comp\tembf{l}_2$ a type embedding $(\typef{A},\funf{f}_1\comp\funf{f}_2,\funf{g}_1\comp\funf{g}_2,\typef{C})$.

Suppose $\typef{A}=(\ordf{a},\csetf{A})$ is a type. We frequently consider a type $\typef{A}$ as a set of symbols. A symbol lies in $\typef{A}$ if it is either $\diamos[\objf{x}]$ and $\objf{x}\in \ordf{a}$ or $\constf{c}$ and $\constf{c}\in \csetf{A}$. A type embedding $\tembf{l}\colon \typef{A}\to\typef{B}$, $(\typef{A},\funf{f},\funf{g},\typef{B})$ can be considered as the mapping of symbols. In this sense the domain is $\typef{A}$ as the set of symbols, $$\tembf{l}\colon\diamos[\objf{x}]\longmapsto\diamos[\funf{f}(\objf{x})]$$ and $$\tembf{l}\colon \constf{c}\longmapsto\funf{g}(\constf{c}).$$

Suppose $\typef{A},\typef{B}$ are types. Obviously, there is at most one trivial type embedding $\tembf{l}$ of $\typef{A}$ into $\typef{B}$. If such an $\tembf{l}$ exists then we call $\typef{B}$  an extension of $\typef{A}$. Suppose $\typef{B}$ is an extension of $\typef{A}$.  If unary operators of $\typef{A}$ and $\typef{B}$ are the same then we call an extension $\typef{B}$ of $\typef{A}$ a {\it constant extension} of $\typef{A}$. If a type $\typef{A}$ and set of constants $\csetf{C}$ are such that $\csetf{C}$ and $\typef{A}$ don't intersects then we denote by $\typef{A}\typeplus \csetf{C}$ the only type $\typef{B}$ such that $\typef{B}$ is a constant extension of $\typef{A}$ and constant symbols of $\typef{B}$ are exactly constant symbols from $\typef{A}$ and symbols from $\csetf{C}$. For a type embedding $\tembf{l}\colon \typef{A}\to \typef{B}$ and set of constants $\csetf{C}$ such that $\typef{A}\typeplus \csetf{C}$ and $\typef{B}\typeplus \csetf{C}$ are defined we denote by $\tembf{l}\typeplus\csetf{C}$ the type embedding  $\tembf{r}\colon \typef{A}\typeplus \csetf{C}\to \typef{B}\typeplus \csetf{C}$ such that $\tembf{r}$ maps symbols from $\typef{A}$ as $\tembf{l}$ and $\tembf{r}$ maps symbols from $\csetf{C}$ to themselves.  For a type $\typef{A}$ and a constant symbol $\constf{c}\not\in \typef{A}$ we denote by $\typef{A}\typeplus \constf{c}$ the type $\typef{A}\typeplus \{\constf{c}\}$. For a type embedding $\tembf{l}\colon\typef{A}\to \typef{B}$ and a constant symbol $\constf{c}$, $\constf{c}\not\in\typef{A}$,$\constf{c}\not\in\typef{B}$ we denote by $\tembf{l}\typeplus \constf{c}$ the type embedding $\tembf{l}\typeplus \{\constf{c}\}$.

Suppose $\tembf{l}\colon\typef{A}\to\typef{B}$ is a type embedding, and $\algf{B}$ is a $\typef{B}$-algebra. We say that an $\typef{A}$-algebra $\algf{A}$ is the {\it $\tembf{l}$-puration} of $\algf{B}$ if the domains of $\algf{A}$ and $\algf{B}$ are the same, Boolean algebra structure of $\algf{A}$ and $\algf{B}$ are the same, and for every symbol $\symf{s}$ from $\typef{A}$ it's interpretation in $\algf{A}$ is the same as the interpretation of $\tembf{l}(\symf{s})$ in $\algf{B}$.  Obviously, $\tembf{l}$-puration of every $\typef{B}$-algebra exists and unique. We denote $\tembf{l}$-puration of $\algf{B}$ by $\purfr^{\tembf{l}}(\algf{B})$.  For a homomorphism $\funf{f}\colon \algf{A}\to \algf{B}$ of $\typef{B}$-algebras, we denote by $\purfr^{\tembf{l}}(\funf{f})$ the homomorphism $\funf{g}\colon \purfr^{\tembf{l}}(\algf{A})\to \purfr^{\tembf{l}}(\algf{B})$ such that $\funf{g}$ is given by the same function from domain of $\algf{A}$ to domain of $\algf{B}$ as $\funf{f}$. Note that $\purfr^{\tembf{l}}$ is a functor from the category of $\typef{B}$-algebras to the category of $\typef{A}$-algebras.

Suppose  $\algf{B}$ is a $\typef{B}$-algebra and $\setf{C}$ is a set of symbols such that every symbol from $\setf{C}$ lies in $\typef{B}$. Then there exists the unique $\typef{A}$ such that $\typef{B}=\typef{A}\typeplus\csetf{C}$. The {\it $\setf{C}$-puration} of $\algf{B}$ is the $\tembf{l}$-puration of $\algf{B}$, where $\tembf{l}$ is the trivial embedding of $\typef{A}$ into $\typef{B}$.

We call a $\typef{B}$-algebra $\algf{B}$ a {\it strong extension} of an $\typef{A}$-algebra $\algf{A}$ if there is a trivial type embedding $\tembf{l}\colon\typef{A}\to\typef{B}$ such that $\algf{A}$ is a $\tembf{l}$-puration of $\algf{B}$.  We call a $\typef{B}$-algebra $\algf{B}$ a {\it strong constant extension}  of an $\typef{A}$-algebra $\algf{A}$  if $\algf{B}$ is a strong extension of $\algf{A}$ and  $\typef{B}$ is a constant extension of $\typef{A}$. We call a $\typef{B}$-algebra $\algf{B}$ a {\it strong  extension by a set of constants $\csetf{C}$} of an $\typef{A}$-algebra $\algf{A}$ if $\algf{B}$ is strong extension of $\algf{A}$ and $\typef{A}\typeplus\csetf{C}=\typef{B}$.

Below we will define the notion of free $\tembf{l}$-extension. 

 Suppose $\tembf{l}\colon \typef{A}\to \typef{B}$ is a type embedding and $\algf{A}$ is a $\typef{A}$-algebra. We define (up to isomorphism) a $\typef{B}$-algebra $\fextfr^{\tembf{l}}(\algf{A})$  and homomorphism $\fextemb^{\tembf{l}}_{\algf{A}}\colon \algf{A}\to \purfr^{\tembf{l}}(\fextfr^{\tembf{l}}(\algf{A}))$. We call $\fextfr^{\tembf{l}}(\algf{A})$  a  {\it free $\tembf{l}$-extension} of $\algf{A}$. $\fextfr^{\tembf{l}}(\algf{A})$ is a $\typef{B}$-algebra such that for every $\typef{B}$-algebra $\algf{C}$ and homomorphism $\funf{g}\colon \algf{A} \to \purfr^{\tembf{l}}(\algf{C})$ there exists the unique homomorphism $\funf{h}\colon\algf{B}\to \algf{C}$ such that $\fextemb^{\tembf{l}}_{\algf{A}}\comp \purfr^{\tembf{l}}(\funf{h})=\funf{g}$. 
\[
\begin{diagram}
\node[3]{\purfr^{\tembf{l}}(\algf{C})}
\node[3]{\algf{C}}\\
\node{\algf{A}}
  \arrow{ene,t}{\funf{g}}
  \arrow[2]{e,b}{\fextemb^{\tembf{l}}_{\algf{A}}}
\node[2]{\purfr^{\tembf{l}}(\fextfr^{\tembf{l}}(\algf{A}))}
  \arrow{n,r,..}{\purfr^{\tembf{l}}(\funf{h})} 
\node[3]{\fextfr^{\tembf{l}}(\algf{A})}
  \arrow{n,r,..}{\funf{h}}\\
\node[2]{\mbox{\it $\typef{A}$-algebras}}
\node[4]{\mbox{\it $\typef{B}$-algebras}}
\end{diagram}
\]
Simple check shows that  every two algebras that satisfies the definition of $\fextfr^{\tembf{l}}(\algf{A})$ are isomorphic. Obviously, if $(\algf{B},\funf{f})$ and $(\algf{B}',\funf{f}')$ satisfies the definition of $(\fextfr^{\tembf{l}}(\algf{A}),\fextemb^{\tembf{l}}_{\algf{A}})$  then there exists the unique isomorphism $\funf{g}\colon \algf{B}\to \algf{B}'$ such that $\funf{f}\comp \purfr^{\tembf{l}}(\funf{g})=\funf{f}'$.\[
\begin{diagram}
\node[3]{\purfr^{\tembf{l}}(\algf{B}')}
\node[3]{\algf{B}'}\\
\node{\algf{A}}
  \arrow{ene,t}{\funf{f}'}
  \arrow[2]{e,b}{\funf{f}}
\node[2]{\purfr^{\tembf{l}}(\algf{B})}
  \arrow{n,r,..}{\purfr^{\tembf{l}}(\funf{g})} 
\node[3]{\algf{B}}
  \arrow{n,r,..}{\funf{g}}\\
\node[2]{\mbox{\it $\typef{A}$-algebras}}
\node[4]{\mbox{\it $\typef{B}$-algebras}}
\end{diagram}
\] Further, we will prove that $\fextfr^{\tembf{l}}(\algf{A})$  and $\fextemb^{\tembf{l}}_{\algf{A}}$ exists; we will assume that we work with some fixed choice of  $\fextfr^{\tembf{l}}(\algf{A})$  and $\fextemb^{\tembf{l}}_{\algf{A}}$.

Suppose   $\tembf{l}\colon \typef{A}\to \typef{B}$ is a type embedding. The {\it $\tembf{l}$-shift} of a $\thrf{GLPA}_{\typef{A}}$-term $\termf{t}$ is the result of replacing every operator symbol $\boxos[\objf{x}]$ and constant symbol $\constf{c}$ with their $\tembf{l}$-image. For a first-order formula $\foflf{f}\in\lang{\thrf{GLPA}_{\typef{A}}}$ we denote by $\foshift^{\tembf{l}}(\foflf{f})$  the result of replacing every term $\termf{t}$ from $\foflf{f}$ with it's $\tembf{l}$-shift; we call the formula $\foshift^{\tembf{l}}(\foflf{f})$ the  $\tembf{l}$-shift of $\foflf{f}$. 
  
 We call an $\typef{A}$-algebra $\algf{A}$ {\it constant complete} if for every $\elf{x}\in\algf{A}$ there exists a constant $\constf{c}\in \typef{A}$ such that $\constf{c}^{\algf{A}}=\elf{x}$.  Note that for a given constant complete $\typef{A}$-algebra $\algf{A}$ and $\typef{A}$-algebra $\algf{B}$ there is at most one homomorphism from  $\algf{A}$ to $\algf{B}$. Clearly for every algebra there exists a strong constant extension which is constant complete. Suppose  $\tembf{l}\colon \typef{A}\to \typef{B}$ is a type embedding and $\algf{A}$ is constant complete $\typef{A}$-algebra. We consider a $\typef{B}$-algebra $\algf{B}$ built of equivalence classes of closed $\thrf{GLPA}_{\typef{B}}$-terms where the equivalence relation is given by
$$\begin{aligned}\termf{t}_1\sim \termf{t}_2 \defiff \thrf{GLPA}_{\typef{B}} &+ \foshift^{\tembf{l}}(\foflf{f}) \vdash \termf{t}_1=\termf{t}_2\mbox{, for some conjunction $\foflf{f}$ of}\\ &\mbox{closed $\thrf{GLPA}_{\typef{A}}$-equations such that $\algf{A}\models \foflf{f}$}. \end{aligned}$$
Interpretations of functions and constants are given for $\algf{B}$ in a natural way: 
\begin{enumerate}
\item $[\termf{t}_1]\baand^{\algf{B}}[\termf{t}_2]=[\termf{t}_1\baand\termf{t}_2]$;
\item $[\termf{t}_1]\baor^{\algf{B}}[\termf{t}_2]=[\termf{t}_1\baor\termf{t}_2]$;
\item $\bacmpi{\algf{B}}{[\termf{t}]}=[\bacmp{\termf{t}}]$;
\item $\diamoi[\objf{x}]{\algf{B}}{[\termf{t}]}=[\diamo[\objf{x}]{\termf{t}}]$, for $\diamos[\objf{x}]\in \typef{B}$;
\item $\constf{c}^{\algf{B}}=[\constf{c}]$, for $\constf{c}\in\typef{A}$ or $\constf{c}\in\{\baz,\bau\}$.
\end{enumerate} 

Suppose $\algf{A}$ is a constant complete $\typef{A}$-algebra. We consider the function $\funf{f}\colon \algf{A}\to \purfr^{\tembf{l}}(\algf{B})$ that maps a $\constf{c}^{\algf{A}}$ to $[\tembf{l}(\constf{c})]$. Clearly, if $\constf{c}_1^{\algf{A}}=\constf{c}_2^{\algf{A}}$ then $\algf{A}\models \constf{c}_1=\constf{c}_2$ and hence $\tembf{l}(\constf{c}_1)\sim\tembf{l}(\constf{c}_2)$. Thus function $\funf{f}$ is well-defined.  Simple check shows that $\funf{f}$ is a homomorphism.  Let us check that $(\algf{B},\funf{f})$ satisfies the definition of $(\fextfr^{\tembf{l}}(\algf{A}),\fextemb^{\tembf{l}}_{\algf{A}})$. Suppose we have a $\typef{B}$-algebra $\algf{C}$ and homomorphism $\funf{g}\colon \algf{A}\to \purfr^{\tembf{l}}(\algf{C})$. We claim that there exists the unique homomorphism $\funf{h}\colon \algf{B}\to \algf{C}$  such that $\funf{f}\comp \purfr^{\tembf{l}}(\funf{h})=\funf{g}$. We put $\funf{h}([\termf{t}])=\termf{t}^{\algf{C}}$, for every $\thrf{GLPA}_{\typef{B}}$-term $\termf{t}$.  For every quantifier-less closed $\foflf{f}\in\lang{\thrf{GLPA}_{\typef{A}}}$ that is true in $\algf{A}$ the algebra $\algf{C}$ satisfies $\foshift^{\tembf{l}}(\foflf{f})$ and hence for closed $\thrf{GLPA}_{\typef{B}}$-terms $\termf{t}_1$ and $\termf{t}_2$ such that lie in a one equivalence class in $\algf{B}$ we have $\termf{t}_1^{\algf{C}}=\termf{t}_2^{\algf{C}}$. Thus $\funf{h}$ is a well-defined function. Clearly, $\funf{h}$  is a homomorphism. Clearly, the homomorphism $\funf{f}\comp \purfr^{\tembf{l}}(\funf{h})=\funf{g}$. Obviously, for every homomorphism $\funf{h}'\colon \algf{B}\to \algf{C}$ and close $\thrf{GLPA}_{\typef{B}}$-term $\termf{t}$ we have $\funf{h}'([\termf{t}])=\funf{h}'(\termf{t}^{\algf{B}})=\termf{t}^{\algf{C}}$. Hence our claim holds.

For a type embedding $\tembf{l}\colon \typef{A}\to \typef{B}$ we denote by $\ran(\tembf{l})$ the set of all $\tembf{l}(\funf{a})$ for symbols $\funf{a}\in\typef{A}$.

Suppose $\tembf{l}\colon \typef{A} \to \typef{B}$, $\tembf{l}'\colon \typef{A}'\to\typef{B}'$, $\tembf{r}_1\colon \typef{A}\to \typef{A}'$, and $\tembf{r}_2\colon \typef{B}\to \typef{B}'$ such  that $\tembf{l}\comp \tembf{r}_2=\tembf{r}_1\comp\tembf{l}'$, $\tembf{r}_1$ and $\tembf{r}_2$ are constant extensions, $\ran(\tembf{l}')\cap \ran(\tembf{r}_2)=\ran(\tembf{l}\comp \tembf{r}_2)$, and $\ran(\tembf{l}')\cup \ran(\tembf{r}_2)=\typef{B}'$. 
\[
\begin{diagram}
\node{\typef{B}'}
\node{\typef{A}'}
  \arrow{w,t}{\tembf{l}'}\\
\node{\typef{B}}
  \arrow{n,l}{\tembf{r}_2}
\node{\typef{A}}
  \arrow{n,r}{\tembf{r}_1}
  \arrow{w,b}{\tembf{l}}
\end{diagram}
\]
We claim that for an $\typef{A}$-algebra $\algf{A}$ and its strong constant extension $\algf{A}'$ that is constant complete $\typef{A}'$-algebra the pair $(\purfr^{\tembf{r}_2}(\fextfr^{\tembf{l}'}(\algf{A})),\purfr^{\tembf{r}_1}(\fextemb^{\tembf{l}'}_{\algf{A}}))$ satisfies the definition of $(\fextfr^{\tembf{l}}(\algf{A}),\fextemb^{\tembf{l}}_{\algf{A}})$. For a $\typef{B}$-algebra $\algf{C}$ and homorphism $\funf{g}\colon \algf{A}\to \purfr^{\tembf{l}}(\algf{C})$ we can in the unique way find a $\typef{B}'$-algebra $\algf{C}'$ and homorphism $\funf{g}'\colon \algf{A}'\to \purfr^{\tembf{l}'}(\algf{C})$ such that $\algf{C}'$ is strong constant extension of $\algf{C}$ and $\purfr^{\tembf{l}}(\funf{g}')=\funf{g}$.  Clearly, every morphism $\funf{h}\colon \algf{A} \to \purfr^{\tembf{r}_2}(\fextfr^{\tembf{l}}(\algf{A}))$  such that $\funf{g}=\purfr^{\tembf{r}_1}(\fextemb^{\tembf{l}}_{\algf{A}})\comp\funf{h}$  is the $\purfr^{\tembf{r}_1}$-image of the unique $\funf{h}'\colon \algf{A}' \to \fextfr^{\tembf{l}}(\algf{A})$. Hence our claim holds.

From the claim it follows that for every $\tembf{l}\colon \typef{A}\to\typef{B}$ and $\typef{A}$-algebra $\algf{A}$  there exists some $\fextfr^{\tembf{l}}(\algf{A})$ and corresponding $\fextemb^{\tembf{l}}_{\algf{A}}$. We fix a choice of $(\fextfr^{\tembf{l}}(\algf{A}),\fextemb^{\tembf{l}}_{\algf{A}})$ for all type embeddings $\tembf{l}\colon \typef{A}\to\typef{B}$ and $\typef{A}$-algebras $\algf{A}$. Also from the claim it follows that
\begin{lemma} \label{fextfr_const_shift}Suppose $\tembf{l}\colon \typef{A} \to \typef{B}$, $\tembf{l}'\colon \typef{A}'\to\typef{B}'$, $\tembf{r}_1\colon \typef{A}\to \typef{A}'$, and $\tembf{r}_2\colon \typef{B}\to \typef{B}'$ are type embeddings,  $\algf{A}$  is an $\typef{A}$-algebra, and $\algf{A}'$ is an $\typef{A}'$-algebra such  that $\tembf{l}\comp \tembf{r}_2=\tembf{r}_1\comp\tembf{l}'$, $\tembf{r}_1$ and $\tembf{r}_2$ are constant extensions, $\ran(\tembf{l}')\cap \ran(\tembf{r}_2)=\ran(\tembf{l}\comp \tembf{r}_2)$, $\ran(\tembf{l}')\cup \ran(\tembf{r}_2)=\typef{B}'$, and $\purfr^{\tembf{r}_1}(\algf{A}')=\algf{A}$. Then the pair $(\purfr^{\tembf{r}_2}(\fextfr^{\tembf{l}'}(\algf{A}')),\purfr^{\tembf{r}_1}(\fextemb^{\tembf{l}'}_{\algf{A}'}))$ satisfies the definition of $(\fextfr^{\tembf{l}}(\algf{A}),\fextemb^{\tembf{l}}_{\algf{A}})$.
\end{lemma}

Suppose $\tembf{l}\colon \typef{A}\to \typef{B}$ is a trivial embedding, $\setf{C}$ is the set of all symbols that lie in $\typef{B}$ but not in $\typef{A}$. Then 
\begin{itemize}
\item $\purfr^{\typef{A},\typef{B}}(\algf{A})$ denotes the algebra $\purfr^{\tembf{l}}(\algf{A})$, for an $\typef{A}$-algebra $\algf{A}$;
\item$\purfr^{\typef{A},\typef{B}}(\funf{f})$ denotes the homomorphism $\purfr^{\tembf{l}}(\funf{f})$, for a homomorphism of $\typef{A}$-algebras $\funf{f}$;
\item $\purfr^{\csetf{C}}(\algf{A})$ denotes the algebra $\purfr^{\tembf{l}}(\algf{A})$, for an $\typef{A}$-algebra $\algf{A}$;
\item$\purfr^{\csetf{C}}(\funf{f})$ denotes the homomorphism $\purfr^{\tembf{l}}(\funf{f})$, for a homomorphism of $\typef{A}$-algebras $\funf{f}$;
\item $\purfr^{\termf{c}}(\algf{A})$ denotes the algebra $\purfr^{\tembf{l}}(\algf{A})$, for an $\typef{A}$-algebra $\algf{A}$, if $\setf{C}=\{\termf{c}\}$;
\item$\purfr^{\termf{c}}(\funf{f})$ denotes the homomorphism $\purfr^{\tembf{l}}(\funf{f})$, for a homomorphism of $\typef{A}$-algebras $\funf{f}$, if $\setf{C}=\{\termf{c}\}$.
\end{itemize}

The following corollary is frequently used form of Lemma \ref{fextfr_const_shift}
\begin{lemma} \label{fextfr_const_shift_c}Suppose $\tembf{l}\colon \typef{A} \to \typef{B}$ is a type embedding and $\csetf{C}$ is a set of constants, $\algf{A}$ is $\typef{A}$-algebra, and $\algf{A}'$ is a strong constant extension of $\algf{A}$ by $\csetf{C}$ such that $\tembf{l}\typeplus \csetf{C}$ is well-defined. Then the pair $(\purfr^{\csetf{C}}(\fextfr^{\tembf{l}\typeplus \csetf{C}}(\algf{A}')),\purfr^{\csetf{C}}(\fextemb^{\tembf{l}\typeplus\csetf{C}}_{\algf{A}'}))$ satisfies the definition of $(\fextfr^{\tembf{l}}(\algf{A}),\fextemb^{\tembf{l}}_{\algf{A}})$.
\end{lemma}

The following lemmma is a corollary of arithmetic completeness theorem for the logic $\thrf{GLP}$:
\begin{lemma}\label{mod_add_cons} Suppose $\tembf{l}\colon(\ordf{a},\csetf{A})\to(\ordf{b},\csetf{B})$ is a type embedding  and $\prflf{f}$ is a $\thrf{GLP}_{\ordf{a}}$-formula. Then $$\thrf{GLP}_{\ordf{a}}\vdash\prflf{f} \iff \thrf{GLP}_{\ordf{b}}\vdash\prshift^{\tembf{l}}(\prflf{f}).$$
\end{lemma}

We will prove the following lemma in Section \ref{syntactical_facts}
\begin{lemma}\label{box_switch_imp}Suppose $\ordf{a}$ is an order type, $\objf{x}_1\le_{\ordf{a}}\objf{x}_2$ are indexes from $\ordf{a}$, and  $\prflf{f},\prflf{p}$ are formulas from $\lang{\thrf{GLP}_{\ordf{a}}}$ such that for any $\boxm[\objf{y}]$ from $\prflf{f}$ or $\prflf{p}$ we have $\objf{x}_2\le_{\ordf{a}}\objf{y}$. Then $\thrf{GLP}_{\ordf{a}}\vdash \boxm[\objf{x}_1]\prflf{p} \primp \prflf{f}$ iff $\thrf{GLP}_{\ordf{a}}\vdash \boxm[\objf{x}_2]\prflf{p} \primp \prflf{f}$.
\end{lemma}

From Lemmas \ref{box_switch_imp},  \ref{glp_deduction}, \ref{fo-mod_corr}, \ref{mod_add_cons} it follows that for a type embedding $\tembf{l} \colon \typef{A} \to \typef{B}$, a constant complete $\typef{A}$-algebra $\algf{A}$, and $[\termf{t}_1],[\termf{t}_2]\in\fextfr^{\tembf{l}}(\algf{A})$
$$\termf{t}_1^{\algf{A}}\ne\termf{t}_1^{\algf{A}}\Implication [\termf{t}_1]\ne[\termf{t}_2].$$ 
Hence we have
\begin{lemma} For a type embedding $\tembf{l} \colon \typef{A} \to \typef{B}$ and an $\typef{A}$-algebra $\algf{A}$ the homorphism $\fextemb^{\tembf{l}}_{\algf{A}}$ is an embedding.
\end{lemma}

\begin{lemma}\label{fextfr_comp}Suppose $\tembf{l}\colon \typef{A}\to \typef{B}$ and $\tembf{r}\colon \typef{B} \to \typef{C}$ are type embeddings. Then for an $\typef{A}$-algebra $\algf{A}$ there exists an isomorphism $\funf{f}\colon \fextfr^{\tembf{l}\comp\tembf{r}}(\algf{A}) \to \fextfr^{\tembf{r}}(\fextfr^{\tembf{l}}(\algf{A})$ such that $\fextemb^{\tembf{l}\comp\tembf{r}}_{\algf{A}}\comp\purfr^{\tembf{l}\comp\tembf{r}}(\funf{f})=\fextemb^{\tembf{l}}_{\algf{A}}\comp\purfr^{\tembf{l}}(\fextemb^{\tembf{r}}_{\fextfr^{\tembf{l}}(\algf{A})})$.
\end{lemma}
\begin{proof} Clearly, the lemma holds if the pair $(\fextfr^{\tembf{r}}(\fextfr^{\tembf{l}}(\algf{A})),\fextemb^{\tembf{l}}_{\algf{A}}\comp\purfr^{\tembf{l}}(\fextemb^{\tembf{r}}_{\fextfr^{\tembf{l}}(\algf{A})})$ satisfies the definition of $(\fextfr^{\tembf{l}\comp\tembf{r}}(\algf{A}),\fextemb^{\tembf{l}\comp\tembf{r}}_{\algf{A}})$. We will prove the late. We denote $\fextfr^{\tembf{l}}(\algf{A})$ by $\algf{B}$, $\fextfr^{\tembf{r}}(\fextfr^{\tembf{l}}(\algf{A}))$ by $\algf{C}$, $\fextemb^{\tembf{l}}_{\algf{A}}$ by $\funf{f}$, and $\fextemb^{\tembf{r}}_{\fextfr^{\tembf{l}}(\algf{A})}$ by $\funf{g}$. Suppose we have a $\typef{C}$-algebra $\algf{D}$ and a homomorphism $\funf{h}\colon \algf{A}\to \purfr^{\tembf{l}\comp\tembf{r}}(\algf{D})$. We claim that there exists a homomorphism $\funf{u}\colon \algf{C}\to \algf{D}$ such that $\funf{u}=\funf{f}\comp \purfr^{\tembf{l}}(\funf{g})\comp\purfr^{\tembf{l}\comp\tembf{r}}(\funf{u})$. We have the unique $\funf{e}\colon \algf{B}\to \purfr^{\tembf{r}}(\algf{D})$ such that $\funf{f}\comp\purfr^{\tembf{l}}(\funf{e})=\funf{h}$.
{
\[
\small
\begin{diagram}
\node[3]{\purfr^{\tembf{l}\comp\tembf{r}}(\algf{D})}
\node[2]{\purfr^{\tembf{r}}(\algf{D})}
\node{\algf{D}}\\
\node{\algf{A}}
  \arrow{ene,t}{\funf{h}}
  \arrow{e,b}{\funf{f}}
\node{\purfr^{\tembf{l}}(\algf{B})}
  \arrow{ne,r,..}{\purfr^{\tembf{l}}(\funf{e})}
  \arrow{e,b}{\purfr^{\tembf{l}}(\funf{g})}
\node{\purfr^{\tembf{l}\comp\tembf{r}}(\algf{C})}
  \arrow{n,r,!}{}
\node{\algf{B}}
  \arrow{ne,b,..}{\funf{e}}
  \arrow{e,b}{\funf{g}}
\node{\purfr^{\tembf{r}}(\algf{C})}
  \arrow{n,r,!}{}
\node{\algf{C}}
  \arrow{n,r,!}{}
\end{diagram}
\]
}
 We obtain the unique $\funf{q}\colon \algf{C}\to \algf{D}$ such that $\funf{g}\comp\purfr^{\tembf{r}}(\funf{q})=\funf{e}$.
{
\[
\small
\begin{diagram}
\node[3]{\purfr^{\tembf{l}\comp\tembf{r}}(\algf{D})}
\node[2]{\purfr^{\tembf{r}}(\algf{D})}
\node{\algf{D}}\\
\node{\algf{A}}
  \arrow{ene,t}{\funf{h}}
  \arrow{e,b}{\funf{f}}
\node{\purfr^{\tembf{l}}(\algf{B})}
  \arrow{ne,r}{\purfr^{\tembf{l}}(\funf{e})}
  \arrow{e,b}{\purfr^{\tembf{l}}(\funf{g})}
\node{\purfr^{\tembf{l}\comp\tembf{r}}(\algf{C})}
  \arrow{n,r,..}{\purfr^{\tembf{l}\comp\tembf{r}}(\funf{q})}
\node{\algf{B}}
  \arrow{ne,b}{\funf{e}}
  \arrow{e,b}{\funf{g}}
\node{\purfr^{\tembf{r}}(\algf{C})}
   \arrow{n,r,..}{\purfr^{\tembf{r}}(\funf{q})}
\node{\algf{C}}
   \arrow{n,lr,..}{\funf{u}}{\funf{q}}
\end{diagram}
\]
}
 We put $\funf{u}=\funf{q}$. Obviously, $\funf{h}=\funf{f}\comp \purfr^{\tembf{l}}(\funf{g})\comp\purfr^{\tembf{l}\comp\tembf{r}}(\funf{u})$. Let us prove that for an $\funf{u}'\colon \algf{C}\to \algf{D}$ such that $\funf{h}=\funf{f}\comp \purfr^{\tembf{l}}(\funf{g})\comp\purfr^{\tembf{l}\comp\tembf{r}}(\funf{u}')$ we have $\funf{u}=\funf{u}'$. Clearly, we have
$$\funf{f}\comp\purfr^{\tembf{l}}(\funf{g}\comp\purfr^{\tembf{r}}(\funf{u}'))=\funf{h},$$
hence from the uniqueness of $\funf{e}$ it follows that $\funf{g}\comp\purfr^{\tembf{r}}(\funf{u}')=\funf{e}$.   Further,  from the uniqueness of $\funf{q}$ it follows that $\funf{u}'=\funf{q}=\funf{u}$.\end{proof}

Suppose $\typef{A}$ is a type, $\typef{B}$ is an extension of $\typef{A}$, and $\tembf{l}$ is the trivial embedding from $\typef{A}$ to $\typef{B}$. We use alias $\fextfr^{\typef{A},\typef{B}}=\fextfr^{\tembf{l}}$.

We call a type $\typef{A}=(\ordf{a},\csetf{A})$ a {\it normal type} if $\ordf{a}$ has the minimal element $\objf{m}$; if $\typef{A}$ is a normal type we call $\objf{m}$ the {\it minimal operator index} of $\typef{A}$ and $\boxos[\objf{m}]$ the {\it minimal operator} of $\typef{A}$.

We define the product of a pair of $\typef{A}$-algebras in a standard fashion. Suppose $\algf{A}$ and $\algf{B}$ are $\typef{A}$-algebras. Then the product $\typef{A}$-algebra $\algf{A}\prodfr\algf{B}$ has the domain $\{(\elf{x},\elf{y})\mid \elf{x}\in \algf{A}, \elf{y}\in \algf{B}\}$ and for any symbol $\funf{f}(\fovarf{x}_{1},\ldots,\fovarf{x}_{\natf{n}})$ from  the signature of $\thrf{GLPA}_{\typef{A}}$ we interpret $\funf{f}(\fovarf{x}_{1},\ldots,\fovarf{x}_{\natf{n}})$  as following:
$$\funf{f}^{\algf{A}\prodfr \algf{B}}((\elf{x}_1,\elf{y}_1),\ldots,(\elf{x}_{\natf{n}},\elf{y}_{\natf{n}}))=(\funf{f}^{\algf{A}}(\elf{x}_1,\ldots,\elf{x}_{\natf{n}}),\funf{f}^{\algf{B}}(\elf{y}_1,\ldots,\elf{y}_{\natf{n}})).$$
Obviously, it gives us an $\typef{A}$-algebra.
If there are $\funf{f}\colon \algf{A}_1\to\algf{A}_2$ and $\funf{g}\colon \algf{B}_1\to\algf{B}_2$ then the homomorphism $\funf{f}\prodfr\funf{g}\colon \algf{A}_1\prodfr\algf{A}_2 \to \algf{B}_1\prodfr \algf{B}_2$ is given by
$$\funf{f}\prodfr\funf{g}\colon (\elf{x},\elf{y})\longmapsto (\funf{f}(\elf{x}),\funf{g}(\elf{y})).$$

Suppose $\typef{A}$ is a normal type and $\algf{A}$, $\algf{B}$ are $\typef{A}$-algebras. We define the linear product $\algf{A}\linprodfr \algf{B}$ of algebras $\algf{A}$ and $\algf{B}$. Suppose $\diamos[\objf{m}]$ is the minimal operator of $\typef{A}$. Suppose $\algf{C}$ is the product of $\{\diamos[\objf{m}]\}$-puration of $\algf{A}$ and  $\{\diamos[\objf{m}]\}$-puration of $\algf{B}$. $\algf{A}\linprodfr \algf{B}$ is an $\typef{A}$-algebra. $\algf{A}\linprodfr \algf{B}$ is a strong extension of $\algf{C}$ with the following interpretation of $\diamos[\objf{m}]$: 
\begin{enumerate}
\item $\diamoi[\objf{m}]{\algf{A}\linprodfr \algf{B}}{(\elf{x},\elf{y})}=(\diamoi[\objf{m}]{\algf{A}}{\elf{x}},\bau^{\algf{B}})$, for $\elf{x}\ne\baz^{\algf{A}}$;
\item $\diamoi[\objf{m}]{\algf{A}\linprodfr \algf{B}}{(\elf{x},\elf{y})}=(\baz^{\algf{A}},\diamoi[\objf{m}]{\algf{B}}{\elf{y}})$, otherwise.
\end{enumerate}
Let us check that this interpretation of $\diamos[\objf{m}]$ gives us an $\typef{A}$-algebras. Obviously, the only axioms we need to check are those where $\diamos[\objf{m}]$ occur. The axioms \ref{GLPA_Ax1}, \ref{GLPA_Ax2}, \ref{GLPA_Ax4}, and \ref{GLPA_Ax5} can be straightforward  check by considering cases from definition of interpretation of $\diamos[\objf{m}]$ for every variable occur in axiom. The fact that \ref{GLPA_Ax3} holds can be proved by considering following cases for $\elf{x}=(\elf{y},\elf{z})$:
\begin{enumerate}
\item \label{gen_case_1734_1}$\diamoi[\objf{m}]{\algf{A}}{\elf{y}}\ne \baz^{\algf{A}}$,
\item \label{gen_case_1734_2}$\diamoi[\objf{m}]{\algf{A}}{\elf{y}}= \baz^{\algf{A}}$ and $\elf{y}\ne \baz^{\algf{A}}$,
\item \label{gen_case_1734_3}$\elf{y}= \baz^{\algf{A}}$.
\end{enumerate}

Suppose $\typef{A}$ is a normal type, $\diamos[\objf{m}]$ is the minimal operator of $\typef{A}$, $\typef{B}$ is the $\{\diamos[\objf{m}]\}$-puration of $\typef{A}$, and $\tembf{l}\colon \typef{B}\to\typef{A}$ is the trivial type embedding. For embeddings of $\typef{A}$-algebras $\funf{f}\colon\algf{A}_1\to\algf{A}_2$ and $\funf{g}\colon\algf{B}_1\to\algf{B}_2$ the embedding $\funf{f}\linprodfr \funf{g}\colon \algf{A}_1\linprodfr \algf{B}_1\to\algf{A}_2\linprodfr\algf{B}_2$ is the only homomorphism $\funf{h}\colon \algf{A}_1\linprodfr \algf{B}_1\to\algf{A}_2\linprodfr \algf{B}_2$ such that $\purfr^{\tembf{l}}(\funf{h})=\purfr^{\tembf{l}}(\funf{f})\prodfr \purfr^{\tembf{l}}(\funf{g})$. Trivial check shows that the late definition is correct. 

\section{Linear Algebras}
In this section first we introduce the notion of a linear $\thrf{GLP}$-algebra. Then in Lemma \ref{fext_linear_preserve} we show that linearity of algebras is  preserved for free extensions that add no new constants and add new operators only below existed operators.

Suppose $\typef{A}=(\ordf{a},\csetf{A})$ is a type and $\algf{A}$ is an $\typef{A}$-algebra.  For every $\objf{i}\in\ordf{a}$ we define two binary relations on $\algf{A}$  
$$x\algl_{\objf{i}}\elf{y}\defiff \elf{y}\le \diamo[\objf{m}]{\elf{x}},$$
$$x\algs_{\objf{i}}\elf{y}\defiff \diamo[\objf{m}]{\elf{x}}\baor\elf{x}= \diamo[\objf{m}]{\elf{y}}\baor\elf{y}.$$
It is clear that for $\objf{i}<_{\ordf{a}}\objf{j}$ we have $\algl_{\objf{j}}\subset \algl_{\objf{i}}$ and $\algs_{\objf{i}}\subset \algs_{\objf{j}}$. We call the algebra $\algf{A}$ {\it linear} if $\bigcup\limits_{\diamos[\objf{i}]\in\typef{A}}(\algl_{\objf{i}}\cup \algs_{\objf{i}})$ is a linear preorder on $\algf{A}$ and $$\bigcup\limits_{\diamos[\objf{i}]\in\typef{A}}(\algl_{\objf{i}})\cap \bigcup\limits_{\diamos[\objf{i}]\in\typef{A}}(\algs_{\objf{i}})=\{(\baz,\baz)\}.$$

Let us consider the case of normal type $\typef{A}$ with minimal operator index $\objf{m}$. An  $\typef{A}$-algebra $\algf{A}$ is  linear if $\algl_{\objf{m}}\cup \algs_{\objf{m}}$ is a linear preorder on $\algf{A}$ and $$\algl_{\objf{m}}\cap \algs_{\objf{m}}=\{(\baz,\baz)\}.$$ 

The proofs of the four following lemmas are trivial and we omit them:

\begin{lemma} \label{algl_subset}Suppose $\typef{A}=(\ordf{a},\csetf{A})$ is a type, $\algf{A}$ is an $\typef{A}$-algebra, and $\objf{i},\objf{j}\in\ordf{a}$, $\objf{i}<_{\ordf{a}}\objf{j}$. Then for $\algf{A}$ we have $\algl_{\objf{j}}\subset\algl_{\objf{i}}$ and $\algs_{\objf{j}}\subset \algs_{\objf{i}}$.
\end{lemma}

\begin{lemma} \label{algl_properties} Suppose $\typef{A}$ is a type, $\diamos[\objf{i}]$ is an operator symbol from $\typef{A}$, and $\algf{A}$ is an $\typef{A}$-algebra. Then for $\algf{A}$
\begin{enumerate}
\item \label{algl_transitivity}$\algl_{\objf{i}}$ is a transitive relation;
\item \label{algl_comparability} for $\elf{x},\elf{y},\elf{z}\in \algf{A}$ such that $\elf{x},\elf{y},\elf{z}\in \algf{A}$ and $\elf{x}\algs_{\objf{i}}\elf{y}$ we have
$$\elf{x}\algl_{\objf{i}}\elf{z}\iff \elf{y}\algl_{\objf{i}}\elf{z},$$ $$\elf{z}\algl_{\objf{i}}\elf{x}\iff\elf{z}\algl_{\objf{i}}\elf{y}$$
\item \label{algl_alternatives} for $\elf{x},\elf{y}\in\algf{A}$ such that $\elf{x},\elf{y}\ne\baz$ at most one of the following three propositions holds:
\begin{enumerate}
\item  $\elf{x}\algl_{\objf{i}}\elf{y}$;
\item  $\elf{y}\algl_{\objf{i}}\elf{x}$;
\item $\elf{x}\algs_{\objf{i}}\elf{y}$.
\end{enumerate}
\end{enumerate}
\end{lemma}

\begin{lemma} \label{embedding_order_preservation} Suppose $\typef{A}$ is a type, $\diamos[\objf{i}]$ lies in $\typef{A}$, and $\funf{f}\colon \algf{A}\to \algf{B}$ is an embedding of $\typef{A}$-algebras. Then $\funf{f}$ preserve $\algl_{\objf{i}}$ and $\algs_{\objf{i}}$, i.e. $\forall \elf{x},\elf{y}\in\algf{A}$:
$$\elf{x}\algl^{\algf{A}}_{\objf{i}}\elf{y}\iff \funf{f}(\elf{x})\algl^{\algf{B}}_{\objf{i}}\funf{f}(\elf{y}),$$
$$\elf{x}\algs^{\algf{A}}_{\objf{i}}\elf{y}\iff \funf{f}(\elf{x})\algs^{\algf{B}}_{\objf{i}}\funf{f}(\elf{y}).$$
\end{lemma}

\begin{lemma} Suppose a $\tembf{l}\colon \typef{A}\to\typef{B}$ is a trivial type embedding, $\diamos[\objf{i}]$ is an operator symbol from $\typef{A}$, and $\algf{A}$ is $\typef{B}$-algebra. Then  $\algl^{\algf{A}}_{\objf{i}}$ and $\algs^{\algf{A}}_{\objf{i}}$ are equal to $\algl^{\purfr^{\tembf{l}}(\algf{A})}_{\objf{i}}$ and $\algs^{\purfr^{\tembf{l}}(\algf{A})}_{\objf{i}}$, respectively.
\end{lemma}

\begin{lemma} \label{algl_lin_properties} Suppose $\typef{A}$ is a normal type, $\diamos[\objf{m}]$ is a minimal operator symbol for $\typef{A}$, and $\algf{A}$ is a linear $\typef{A}$-algebra. Then 
for $\elf{x},\elf{y}\in \algf{A}$ we have 
$$\diamo[\objf{m}]{\elf{x}}=\diamo[\objf{m}]{\elf{y}} \iff \elf{x}\algs_{\objf{m}}\elf{y}.$$
\end{lemma}
\begin{proof} $\Rightarrow:$ Suppose $\diamo[\objf{m}]{\elf{x}}=\diamo[\objf{m}]{\elf{y}}\ne\baz$. We claim that $\elf{x}\algs_{\objf{m}}\elf{y}$. Obviously, it's enough to show that  $\lnot \elf{x}\algl_{\objf{m}}\elf{y}$ and $\lnot \elf{y}\algl_{\objf{m}}\elf{x}$. Assume that $\elf{x}\algl_{\objf{m}}\elf{y}$. Then $\diamo[\objf{m}]{\elf{y}}=\diamo[\objf{m}]{\elf{x}}\ge\elf{y}$. From Axiom \ref{GLPA_Ax3} of $\thrf{GLP}$-algebras it follows that  $\diamo[\objf{m}]{\elf{y}}=\baz$, contradiction. For the same reason, the assumption $\elf{y}\algl_{\objf{m}}\elf{x}$ leads to contradiction too. 

$\Leftarrow$: Now suppose that  $\elf{x}\algs_{\objf{m}}\elf{y}$. We claim that $\diamo[\objf{m}]{\elf{x}}=\diamo[\objf{m}]{\elf{y}}$. We have $$\elf{x}\baor\diamo[\objf{m}]{\elf{x}}=\elf{y}\baor\diamo[\objf{m}]{\elf{y}}.$$ Hence we have $$\diamo[\objf{m}]{\elf{x}}\baor\diamo[\objf{m}]{\diamo[\objf{m}]{\elf{x}}}=\diamo[\objf{m}]{\elf{y}}\baor\diamo[\objf{m}]{\diamo[\objf{m}]{\elf{y}}}.$$ And finally we conclude $$\diamo[\objf{m}]{\elf{x}}=\diamo[\objf{m}]{\elf{y}}.$$
\end{proof}

Suppose $\tembf{l}\colon \typef{A}\to\typef{B}$ is a type embedding, where $\tembf{l}=(\typef{A},\funf{f},\funf{g},\typef{B})$, $\typef{A}=(\ordf{a},\csetf{A})$, and $\typef{B}=(\ordf{b},\csetf{B})$. We call $\tembf{l}$ a {\it final type embedding}, if
\begin{enumerate}
\item $\typef{A}$ is a normal type;
\item $\funf{g}$ is bijection;
\item $\funf{f}$ is an embedding of $\ordf{a}$ into $\ordf{b}$ as final interval.
\end{enumerate}
 We call $\tembf{l}$ a {\it simple final type embedding}, if it is a final type embedding and there is only one element of $\ordf{b}$ that is not in the range of $\funf{f}$. We call $\tembf{l}$ a {\it normal type embedding}, if  $\tembf{l}$  is a simple final type embedding and a trivial embedding, the minimal operator index of $\typef{A}$ is $\ione$, and the minimal operator index of $\typef{B}$ is $\itwo$.

Obviously, for every simple final type embedding $\tembf{l}\colon \typef{A}\to\typef{B}$ we can find a normal embedding $\tembf{l}'\colon\typef{A}'\to\typef{B}'$ and bijective type embeddings $\tembf{r}_1\colon \typef{A}\to\typef{A}'$, $\tembf{r}_2\colon \typef{B}\to\typef{B}'$ such that $\tembf{l}=\tembf{r}_1\comp\tembf{l}'\comp\tembf{r}_2^{-1}$. 
\[
\begin{diagram}
\node{\typef{B}'}
\node{\typef{A}'}
  \arrow{w,t}{\tembf{l}'}\\
\node{\typef{B}}
  \arrow{n,l}{\tembf{r}_2}
\node{\typef{A}}
  \arrow{n,r}{\tembf{r}_1}
  \arrow{w,b}{\tembf{l}}
\end{diagram}
\]
In most cases, without lose of generality, we consider only normal embeddings instead of simple final type embeddings.

Further in this section we develop a generalization of the theory of $\thrf{GLP}$-words \cite{Bek04} for linear $\thrf{GLP}$-algebras.

Suppose $\typef{A}$ is a type with the minimal operator symbol $\diamos[\ione]$ and $\termf{t}$ is a closed  $\thrf{GLPA}_{\typef{A}}$-terms of the form
 $$\constf{c}_{\natf{n}} \baand \diamo[\ione]{\constf{c}_{\natf{n}-1} \baand \diamo[\ione]{ \constf{c}_{\natf{n}-2} \baand\diamo[\ione]{\ldots \diamo[\ione]{ \constf{c}_{2} \baand \diamo[\ione]{\constf{c}_{1}}} \ldots}}},$$ where $n\ge 1$. Then we call $\termf{t}$ a {\it quasi-words of the type $\typef{A}$}. 

Suppose $\tembf{l}\colon\typef{A}\to \typef{B}$ is a normal type embedding and $\algf{A}$ is a $\typef{A}$-algebra. Then we call a quasi-word
 $$\constf{c}_{\natf{n}} \baand \diamo[\ione]{\constf{c}_{\natf{n}-1} \baand \diamo[\ione]{ \constf{c}_{\natf{n}-2} \baand\diamo[\ione]{\ldots \diamo[\ione]{ \constf{c}_{2} \baand \diamo[\ione]{\constf{c}_{1}}} \ldots}}}$$
 of the type $\typef{B}$  {\it $\algf{A}$-normalized} if $\algf{A}\not \models \constf{c}_{\natf{i}}\algl_{\itwo}\constf{c}_{\natf{i+1}}$, for all $\natf{i}$ from $1$ to $\natf{n}-1$.

\begin{lemma} \label{quasi-word_normalization} Suppose $\tembf{l}\colon\typef{A}\to \typef{B}$ is a normal type embedding, $\algf{A}$ is a $\typef{A}$-algebra, and $\termf{t}$ is a quasi-word of the type $\typef{B}$. Then there exists an $\algf{A}$-normalized word $\termf{t}'$ such that $\fextfr{\algf{A}}\models \termf{t}=\termf{t}'$.
\end{lemma}
\begin{proof} Suppose $\termf{t}$ have the form 
 $$\constf{c}_{\natf{n}} \baand \diamo[\ione]{\constf{c}_{\natf{n}-1} \baand \diamo[\ione]{ \constf{c}_{\natf{n}-2} \baand\diamo[\ione]{\ldots \diamo[\ione]{ \constf{c}_{2} \baand \diamo[\ione]{\constf{c}_{1}}} \ldots}}}.$$
We prove the lemma by induction on $\natf{n}$. In the case of $\natf{n}=1$ the quasi-word $\termf{t}$ is just a constant symbol. Hence the induction basis  holds.

Now we prove the induction step. Either $\termf{t}$ is an $\algf{A}$-normalized quasi-word and we are done or there is a number $\natf{i}<\natf{n}$ such that $\algf{A}\models \constf{c}_{\natf{i}}\algl_{\itwo}\constf{c}_{\natf{i}+1}$. We denote the word
$$\constf{c}_{\natf{n}} \baand \diamo[\ione]{\constf{c}_{\natf{n}-1} \baand \diamo[\ione]{ \ldots \constf{c}_{\natf{i+1}} \baand\diamo[\ione]{\constf{c}_{\natf{i}-1}\baand\diamo[\ione]{\ldots \diamo[\ione]{ \constf{c}_{2} \baand \diamo[\ione]{\constf{c}_{1}}} \ldots}}\ldots}}$$
by  $\termf{w}$.
Clearly,
$$\fextfr^{\tembf{l}}(\algf{A})\models\forall \fovarf{x}(\constf{c}_{\natf{i}+1}\baand\diamo[\ione]{\fovarf{x}}= \constf{c}_{\natf{i}+1}\baand \diamo[\itwo]{\constf{c}_{\natf{i}}}\baand \diamo[\ione]{\fovarf{x}}= \constf{c}_{\natf{i}+1}\baand \diamo[\itwo]{\constf{c}_{\natf{i}}\baand \diamo[\ione]{\fovarf{x}}}\le \constf{c}_{\natf{i}+1}\baand \diamo[\ione]{\constf{c}_{\natf{i}}\baand \diamo[\ione]{\fovarf{x}}}).$$
From the other side, 
$$\fextfr^{\tembf{l}}(\algf{A})\models\forall \fovarf{x}(\constf{c}_{\natf{i}+1}\baand\diamo[\ione]{\fovarf{x}}\ge
\constf{c}_{\natf{i}+1}\baand\diamo[\ione]{\diamo[\ione]{\fovarf{x}}}\ge\constf{c}_{\natf{i}+1}\baand\diamo[\ione]{\constf{c}_{\natf{i}}\baand\diamo[\ione]{\fovarf{x}}}).$$
Hence
 $$\fextfr^{\tembf{l}}(\algf{A})\models\forall \fovarf{x}(\constf{c}_{\natf{i}+1}\baand\diamo[\ione]{\fovarf{x}}=\constf{c}_{\natf{i}+1}\baand\diamo[\ione]{\constf{c}_{\natf{i}}\baand\diamo[\ione]{\fovarf{x}}}).$$
Therefore
 $$\fextfr^{\tembf{l}}(\algf{A})\models \termf{t}=\termf{w}.$$
We use the inductive hypothesis for $\termf{w}$ and obtain an $\algf{A}$-normalized quasi-word $\termf{w}'$ such that $$\fextfr^{\tembf{l}}(\algf{A})\models \termf{w}'=\termf{w}=\termf{t}.$$ It finishes the proof of the induction step and the lemma.\end{proof}

\begin{lemma} \label{quasi-word_comparability}Suppose $\tembf{l}\colon\typef{A}\to\typef{B}$ is a normal type embedding and $\algf{A}$ is a constant complete linear $\typef{A}$-algebra. Then for every two quasi-words $\termf{t},\termf{w}$ of the type $\typef{B}$ we have either $\fextfr^{\tembf{l}}(\algf{A})\models \termf{t}  \algl_{\ione}\termf{w}$, or $\fextfr^{\tembf{l}}(\algf{A})\models \termf{t}\algs_{\ione}\termf{w}$, or $\fextfr^{\tembf{l}}(\algf{A})\models \termf{w}\algl_{\ione}\termf{t}$.
\end{lemma}
\begin{proof} We consider  $\algf{A}$-normalized quasi-words $\termf{t}$ and $\termf{w}$  of the type $\typef{B}$. Suppose $\termf{t}$ is
 $$\constf{c}_{\natf{n}} \baand \diamo[\ione]{\constf{c}_{\natf{n}-1} \baand \diamo[\ione]{ \constf{c}_{\natf{n}-2} \baand\diamo[\ione]{\ldots \diamo[\ione]{ \constf{c}_{2} \baand \diamo[\ione]{\constf{c}_{1}}} \ldots}}}$$
and $\termf{w}$ is
 $$\constf{q}_{\natf{m}} \baand \diamo[\ione]{\constf{q}_{\natf{m}-1} \baand \diamo[\ione]{ \constf{q}_{\natf{m}-2} \baand\diamo[\ione]{\ldots \diamo[\ione]{ \constf{q}_{2} \baand \diamo[\ione]{\constf{q}_{1}}} \ldots}}}.$$
We claim that either $\fextfr^{\tembf{l}}(\algf{A})\models \termf{t}  \algl_{\ione}\termf{w}$ or $\fextfr^{\tembf{l}}(\algf{A})\models \termf{t}\algs_{\ione}\termf{w}$ or $\fextfr^{\tembf{l}}(\algf{A})\models \termf{w}\algl_{\ione}\termf{t}$. If for some $1<i\le\natf{n}$ we have $\algf{A}\models \constf{c}_{\natf{i}}=\bau$ or for some $1<i\le\natf{m}$ we have $\algf{A}\models \constf{q}_{\natf{i}}=\bau$ then the claim holds trivially and we are done. Now we assume that for all $1<i\le\natf{n}$ we have $\algf{A}\not\models \constf{c}_{\natf{i}}=\bau$  and  for all $1<i\le\natf{m}$ we have $\algf{A}\not\models \constf{q}_{\natf{i}}=\bau$.  We find the minimal $\natf{i}$ from $1$ to $\min(\natf{n}-1,\natf{m}-1)$ such that $\algf{A}\not\models \constf{q}_{\natf{i}}\algs_{\itwo}\constf{c}_{\natf{i}}$. If there are no such a $\natf{i}$ then we show that
\begin{enumerate}
\item \label{quasi-word_comparability_l1_c1} $\algf{A}\models \termf{t}\algl_{\ione} \termf{w}$, if $\natf{n}<\natf{m}$,
\item \label{quasi-word_comparability_l1_c2}$\algf{A}\models \termf{t}\algs_{\ione} \termf{w}$, if $\natf{n}=\natf{m}$,
\item  \label{quasi-word_comparability_l1_c3}$\algf{A}\models \termf{w}\algl_{\ione} \termf{t}$, if $\natf{m}<\natf{n}$.
\end{enumerate}
We have $$\thrf{GLPA}_{\typef{B}}\vdash \forall \fovarf{x},\fovarf{y},\fovarf{z}( \fovarf{x}\algs_{\itwo}\fovarf{y} \;\Rightarrow\; \fovarf{x}\baand \diamo[\ione]{ \fovarf{z}} \algs_{\itwo}\fovarf{y}\baand \diamo[\ione]{\fovarf{z}})$$
and $$\thrf{GLPA}_{\typef{B}}\vdash \forall \fovarf{x},\fovarf{y}( \fovarf{x}\algs_{\itwo}\fovarf{y} \;\Rightarrow\; \diamo[\ione]{ \fovarf{x}} = \diamo[\ione]{\fovarf{y}}).$$
Using this two facts we prove \ref{quasi-word_comparability_l1_c2}. Second, using  these facts and 
$$\thrf{GLPA}_{\typef{B}}\vdash \forall \fovarf{x},\fovarf{y}( \fovarf{x}\algl_{\ione}\fovarf{y} \baand \diamo[\ione]{ \fovarf{x}})$$ we prove \ref{quasi-word_comparability_l1_c1} and \ref{quasi-word_comparability_l1_c3}.
Now we assume that we have $\natf{i}$ such that $\algf{A}\not\models \constf{q}_{\natf{i}}\algs_{\itwo}\constf{c}_{\natf{i}}$ and for all $1\le \natf{j}<\natf{i}$ we have  $\algf{A}\models \constf{q}_{\natf{j}}\algs_{\itwo}\constf{c}_{\natf{j}}$. Because $\algf{A}$ is linear we have either $\algf{A}\models \constf{q}_{\natf{i}}\algl_{\itwo}\constf{c}_{\natf{i}}$ or $\algf{A}\models \constf{c}_{\natf{i}}\algl_{\itwo}\constf{q}_{\natf{i}}$. Without lose of generality we assume that $\algf{A}\models \constf{q}_{\natf{i}}\algl_{\itwo}\constf{c}_{\natf{i}}$. Because $\termf{w}$ is  $\algf{A}$-normalized and $\algf{A}$ is linear, we have  $\algf{A}\models \constf{q}_{\natf{j}}\algl_{\itwo}\constf{c}_{\natf{i}}$ for all $\natf{j}$ from $\natf{i}$ to $\natf{m}$. Using the same method as above in the proof we show that
$$\fextfr^{\tembf{l}}(\algf{A})\models\forall \fovarf{x}(\constf{c}_{\natf{i}}\baand\diamo[\ione]{\fovarf{x}}=\constf{c}_{\natf{i}}\baand\diamo[\ione]{\constf{q}_{\natf{j}}\baand\diamo[\ione]{\fovarf{x}}}),$$
for all $\natf{j}$ from $\natf{i}$ to $\natf{m}$. Using the last we show that 
$$\fextfr^{\tembf{l}}(\algf{A})\models \termf{w}\algl_{\ione}\constf{c}_{\natf{i}} \baand \diamo[\ione]{\constf{c}_{\natf{i}-1} \baand \diamo[\ione]{ \constf{c}_{\natf{n}-2} \baand\diamo[\ione]{\ldots \diamo[\ione]{ \constf{c}_{2} \baand \diamo[\ione]{\constf{c}_{1}}} \ldots}}}.$$
And finally we conclude that $$\fextfr^{\tembf{l}}(\algf{A})\models \termf{w}\algl_{\ione}\termf{t}.$$
\end{proof}

\begin{corollary} \label{quasi-word_comparability_2}Suppose $\tembf{l}\colon\typef{A}\to\typef{B}$ is a normal type embedding and $\algf{A}$ is a constant complete linear $\typef{A}$-algebra. Then for every two quasi-words $\termf{t},\termf{w}$ of the type $\typef{B}$, terms $\diamo[\ione]{\termf{t}}$ and $\diamo[\ione]{\termf{w}}$ are $\le$-comparable in $\fextfr^{\tembf{l}}(\algf{A})$.
\end{corollary}

\begin{proposition}\label{quasi-word_combination} Suppose $\tembf{l}\colon\typef{A}\to\typef{B}$ is a normal type embedding and $\algf{A}$ is a constant complete linear $\typef{A}$-algebra. Then every element of $\fextfr^{\tembf{l}}(\algf{A})$ is equal to the value of some Boolean combination of  quasi-words of the type $\typef{A}$.
\end{proposition}
\begin{proof} We call a term $\termf{t}$ a {\it quasi-word closures of the type $\typef{B}$} if $\termf{t}$ is $\boxo[\ione]{\termf{w}}$ for some quasi-word $\termf{w}$ of the type $\typef{B}$ .

Note that the conclusion of the lemma is equivalent to the following proposition: every element of $\fextfr^{\tembf{l}}(\algf{A})$ is equal to the value of some Boolean combination of constants and quasi-word closures of the type $\typef{B}$. We will actually prove this equivalent form.

Every element of  $\fextfr^{\tembf{l}}(\algf{A})$  is the value of some closed $\thrf{GLPA}_{\typef{B}}$-term $\termf{t}$. We prove the lemma by induction on the length of the representing term. For constant symbols as representing term the lemma holds trivially. Obviously, the induction step holds for all cases for the top operation in representing term but some $\boxos[\objf{x}]$.  

Suppose $\termf{t}$ is $\boxo[\objf{x}]{\termf{w}}$. From induction hypothesis we know that $\termf{w}^{\algf{B}}$ is the value of some Boolean combination  $\termf{w}'$ of constants and quasi-word closures of the type $\typef{B}$. We use Corollary \ref{quasi-word_comparability_2} and constant completeness of $\algf{A}$ to transform CNF of $\termf{w}'$ in the way we describe subsequently. We obtain number $\natf{k}$ and for every $\natf{i}$ from $1$ to $\natf{k}$ we obtain constants $\constf{c}_{\natf{i}}\in\typef{A}$ and  quasi-word closures $\termf{u}_{\natf{i}}$, $\termf{v}_{\natf{i}}$ of the type $\typef{B}$ such that 
$$\fextfr^{\tembf{l}}(\algf{A})\models \termf{w}=\sum\limits_{1\le \natf{i}\le\natf{k}}(\constf{c}_{\natf{i}}\baand \termf{u}_{\natf{i}}\baand \bacmp{\termf{v}_{\natf{i}}}).$$
First, we assume that $\diamos[\objf{x}]\in\typef{A}$. For every $\natf{i}$ from $1$ to $\natf{k}$ we find  $\constf{q}_{\natf{i}}\in \typef{A}$ such that
$$\algf{A}\models \constf{q}_{\natf{i}}=\diamo[\objf{x}]{\constf{c}_{\natf{i}}}.$$ 
Then using items \ref{GLPA_Eq1_2} and \ref{GLPA_Eq1_3} of Lemma \ref{GLPA_Eq1} we conclude that
$$\fextfr^{\tembf{l}}(\algf{A})\models \termf{t}=\sum\limits_{1\le \natf{i}\le\natf{k}}(\constf{q}_{\natf{i}}\baand \termf{u}_{\natf{i}}\baand \bacmp{\termf{v}_{\natf{i}}})$$
and we are done.

Now we assume that $\objf{x}=\ione$.  Clearly, it's enough to show that for every $\natf{i}$ from $1$ to $\natf{k}$ we have a quasi-word closure $\diamo[\ione]{\termf{q}_{\natf{i}}}$ of the type $\typef{B}$ such that  $$\fextfr^{\tembf{l}}(\algf{A})\models \diamo[\ione]{\constf{c}_{\natf{i}}\baand \termf{u}_{\natf{i}}\baand \bacmp{\termf{v}_{\natf{i}}}}=\diamo[\ione]{\termf{q}_{\natf{i}}}.$$  If $$\fextfr^{\tembf{l}}(\algf{A})\models\constf{c}_{\natf{i}}\baand \termf{u}_{\natf{i}}\algl_{\ione}\termf{v}_{\natf{i}}$$ then we put $\termf{q}_{\natf{i}}=\constf{q}_{\natf{i}}\baor \termf{u}_{\natf{i}}$; simple check shows that $\termf{q}_{\natf{i}}$ satisfies the requirements. Note that here we also have $$\fextfr^{\tembf{l}}(\algf{A})\models\termf{q}_{\natf{i}}\algs_{\ione}\constf{c}_{\natf{i}}\baand \termf{u}_{\natf{i}}\baand \bacmp{\termf{v}_{\natf{i}}}.$$  Now we assume that $$\fextfr^{\tembf{l}}(\algf{A})\not\models\constf{c}_{\natf{i}}\baand \termf{u}_{\natf{i}}\algl_{\ione}\termf{v}_{\natf{i}}$$ From Lemma \ref{quasi-word_comparability} it follows that we have $$\fextfr^{\tembf{l}}(\algf{A})\models(\termf{v}_{\natf{i}}\algl_{\ione}\constf{c}_{\natf{i}}\baand \termf{u}_{\natf{i}})\foor(\constf{c}_{\natf{i}}\baand \termf{u}_{\natf{i}}\algs_{\ione}\termf{v}_{\natf{i}}).$$ Because $\termf{v}_{\natf{i}}$ starts with $\diamos[\ione]$ we have
$$\fextfr^{\tembf{l}}(\algf{A})\models \termf{v}_{\natf{i}}=\termf{v}_{\natf{i}}\baor \diamo[\ione]{\termf{v}_{\natf{i}}}\ge(\constf{c}_{\natf{i}}\baand \termf{u}_{\natf{i}})\baor \diamo[\ione]{\constf{c}_{\natf{i}}\baand \termf{u}_{\natf{i}}}\ge\constf{c}_{\natf{i}}\baand \termf{u}_{\natf{i}} .$$ 
Hence we can take $\bau$ as $\termf{q}_{\natf{i}}$. Note that here we also have $$\fextfr^{\tembf{l}}(\algf{A})\models\termf{q}_{\natf{i}}\algs_{\ione}\constf{c}_{\natf{i}}\baand \termf{u}_{\natf{i}}\baand \bacmp{\termf{v}_{\natf{i}}}.$$ 
\end{proof}

\begin{lemma} \label{quasi-word_equiv}Suppose $\tembf{l}\colon\typef{A}\to\typef{B}$ is a normal type embedding and $\algf{A}$ is a constant complete linear $\typef{A}$-algebra. Then every element of $\fextfr^{\tembf{l}}(\algf{A})$ is $\algs_{\ione}$-equivalent to the value of some  quasi-word of the type $\typef{B}$.
\end{lemma}
\begin{proof} We denote $\fextfr^{\tembf{l}}(\algf{A})$ by $\algf{B}$.  In the end part of  Proposition \ref{quasi-word_combination} we have actually shown that every $\elf{x}\in\algf{B}$ is equal to $\sum\limits_{1\le\natf{i}\le\natf{k}}\elf{y}_i$ such that all $\elf{y}_{\natf{i}}$ are  $\algs_{\ione}$-equivalent to the value of some  quasi-word of the type $\typef{B}$.

We claim that if $\elf{x},\elf{y}\in\algf{B}$ are $\algs_{\ione}$-equivalent to the values of some  quasi-words $\termf{t}_1,\termf{t}_2$ of the type $\typef{B}$, respectively, then $\elf{x}\baand^{\algf{B}}\elf{y}$ is the value of some quasi-word of the type $\typef{B}$.  There are three cases: $\algf{B}\models \termf{t}_1\algl_{\ione}\termf{t}_2$, $\algf{B}\models \termf{t}_1\algs_{\ione}\termf{t}_2$, $\algf{B}\models \termf{t}_2\algl_{\ione}\termf{t}_1$. Without lose of generality we consider only first two cases because third case is equivalent to the first one. If $\algf{B}\models\termf{t}_1\algl_{\ione}\termf{t}_2$ then  $\algf{B}\models\elf{x}\baor\diamo[\ione]{\elf{x}}\le\elf{y}$ and hence
$$\algf{B}\models\termf{t}_1 \baor \diamo[\ione]{\termf{t}_1}=\elf{x}\baor\diamo[\ione]{\elf{x}}=(\elf{x}\baor \elf{y})\baor\diamo[\ione]{\elf{x}\baor \elf{y}}.$$
Therefore in the first case we can take $\termf{t}_1$ as required quasi-word of the type $\typef{B}$.
In the second case we have $$\algf{B}\models\termf{t}_1 \baor \diamo[\ione]{\termf{t}_1}=\termf{t}_2 \baor \diamo[\ione]{\termf{t}_2}=\elf{x}\baor\diamo[\ione]{\elf{x}}=\elf{y}\baor\diamo[\ione]{\elf{y}}=(\elf{x}\baor \elf{y})\baor\diamo[\ione]{\elf{x}\baor \elf{y}},$$
and hence we can take $\termf{t}_1$ as required  quasi-word of the type $\typef{B}$.

Obviously, the lemma follows from the claim\end{proof}

\begin{lemma} \label{fext_linear_preserve}Suppose $\tembf{l}\colon \typef{A}\to \typef{B}$ is a final type embedding. Then for a linear $\typef{A}$-algebra $\algf{A}$ the algebra $\fextfr^{\tembf{l}}(\algf{A})$ is linear.
\end{lemma}
\begin{proof}We consider three cases:
\begin{enumerate}
 \item \label{fext_linear_preserve_c1}$\tembf{l}$ is a simple final type embedding;
\item \label{fext_linear_preserve_c2} there are only finitely many operator symbols in $\typef{B}$ that are not $\tembf{l}$ image of some operator symbol from $\typef{A}$;
\item \label{fext_linear_preserve_c3}$\tembf{l}$ is an arbitrary final type embedding.
\end{enumerate}

Case \ref{fext_linear_preserve_c1}: Follows from  Lemmas \ref{quasi-word_comparability} and \ref{quasi-word_equiv}.

Case \ref{fext_linear_preserve_c2}: Obviously, we can decompose $\tembf{l}$ as a composition of simple final type embeddings:
$$\tembf{l}=\tembf{r}_1\comp\tembf{r}_2\comp\ldots \comp \tembf{r}_{\natf{n}}.$$
From Lemma \ref{fextfr_comp} it follows that for an $\typef{A}$-algebra $\algf{A}$ algebras $\fextfr^{\tembf{l}}(\algf{A})$ and $$\fextfr^{\tembf{r}_{\natf{n}}}(\ldots(\fextfr^{\tembf{r}_{1}}(\algf{A})))$$
are isomorphic. Hence for a linear $\typef{A}$-algebra $\algf{A}$ the algebra $\fextfr^{\tembf{l}}(\algf{A})$ is linear.

Case \ref{fext_linear_preserve_c3}: Due to Lemma \ref{fextfr_const_shift} we can consider only the case of constant complete algebra $\algf{A}$. From the definition of linear algebra, it follows that the linearity of $\fextfr^{\tembf{l}}(\algf{A})$ follows from the subsequent claim. We claim that for every three elements $\elf{x},\elf{y},\elf{z}\in\fextfr^{\tembf{l}}(\algf{A})$ and operator $\diamos[\objf{i}]\in\typef{B}$ there exists a type $\typef{C}$, the trivial type embedding $\tembf{r}\colon\typef{C}\to\typef{A}$, a linear $\typef{C}$-algebra $\algf{B}$, and an embedding $\funf{f}\colon \algf{B}\to \fextfr^{\tembf{l}}(\algf{A})$ such that there exist $\funf{f}^{-1}(\elf{x})$, $\funf{f}^{-1}(\elf{y})$, and $\funf{f}^{-1}(\elf{z})$. From constant completeness of $\algf{A}$ it follows that we can find closed $\thrf{GLPA}_{\typef{B}}$-terms $\termf{t}_1,\termf{t}_2,\termf{t}_3$ such that $\elf{x}$, $\elf{y}$, and $\elf{z}$ are equal to the values of $\termf{t}_1$, $\termf{t}_2$, and $\termf{t}_3$ in $\fextfr^{\tembf{l}}(\algf{A})$, respectively. We find final type embeddings $\tembf{r}'\colon \typef{A}\to\typef{C}$ and trivial type embedding $\tembf{r}\colon\typef{C}\to\typef{B}$ such that $\tembf{r}'\comp\tembf{r}=\tembf{l}$, the set of operators from $\typef{C}$ that are not in $\ran(\tembf{r}_1)$ is finite, $\diamos[\objf{i}]\in\typef{C}$, and $\termf{t}_1,\termf{t}_,\termf{t}_3$ are $\thrf{GLPA}_{\typef{C}}$-terms. We denote by $\algf{B}$ the algebra $\fextfr^{\tembf{r}'}(\algf{A})$. From the case \ref{fext_linear_preserve_c2} we know that $\algf{B}$ is linear. Without lose of generality we can assume that $\fextfr^{\tembf{l}}(\algf{A})=\fextfr^{\tembf{r}}(\algf{B})$. Clearly, $\elf{x},\elf{y},\elf{z}$ are in the $\fextemb^{\tembf{r}}_{\algf{B}}$-image of $\algf{B}$. That finishes the proof of the claim.\end{proof}

From the Lemma \ref{fext_linear_preserve} and Lemma \ref{algl_subset} it follows that
\begin{corollary} \label{algl_fextemb} Suppose $\tembf{l}\colon\typef{A}\to\typef{B}$ is a normal type embedding and $\algf{A}$ is a linear $\typef{A}$-algebra. Then for every $\elf{x},\elf{y}\in\algf{A}$:
$$\elf{x}\algl_{\itwo}^{\algf{A}}\elf{y} \iff \fextemb_{\algf{A}}^{\tembf{l}}(\elf{x})\algl_{\ione}^{\fextfr^{\tembf{l}}(\algf{A})} \fextemb_{\algf{A}}^{\tembf{l}}(\elf{y})\mbox{ and}$$
$$\elf{x}\algs_{\itwo}^{\algf{A}}\elf{y} \iff \fextemb_{\algf{A}}^{\tembf{l}}(\elf{x})\algs_{\ione}^{\fextfr^{\tembf{l}}(\algf{A})} \fextemb_{\algf{A}}^{\tembf{l}}(\elf{y}).$$
\end{corollary}

\section{Some Factor Algebras}
Suppose $\constf{q}$ is a constant symbol, $\typef{A}=(\ordf{a},\csetf{A})$ is a normal type, $\diamos[\objf{m}]$ is the minimal operator  of $\typef{A}$ and $\constf{q}$ doesn't lie in $\typef{A}$. We denote $\typef{A}\typeplus\constf{q}$ by $\typef{B}$ and $\{\diamos[\objf{m}]\}$-puration of $\typef{A}$ by $\typef{C}$.

Suppose $\algf{A}$ is a $\typef{B}$-algebra. Then we denote by $\fquotalg^{\typef{A},\constf{q}}(\algf{A})$ the $\{\constf{q}\}$-puration of the factor algebra  $\algf{A}/{\sim}$, where $\sim$ is
$$\elf{x}\sim\elf{y}\defiff \elf{x}\baor \diamo[\objf{m}]{\constf{q}^{\algf{A}}}=\elf{x}\baor \diamo[\objf{m}]{\constf{q}}.$$
Obviously, $\sim$ is an equivalence relation. Let us check that $\sim$ is compatible with all operations of $\algf{A}$. Obviously, Boolean operations of $\algf{A}$ are compatible with $\sim$.  Now we prove compatibility for operators $\diamos[\objf{i}]$, where $\objf{i}\ne\objf{m}$. Suppose we have $\objf{i}>_{\ordf{a}}\objf{m}$ and $\elf{x},\elf{y}\in \algf{A}$ such that
$$\elf{x}\baor \diamo[\objf{m}]{\constf{q}}=\elf{y}\baor \diamo[\objf{m}]{\constf{q}}.$$
We need to prove that
$$\diamo[\objf{i}]{\elf{x}}\baor \diamo[\objf{m}]{\constf{q}}=\diamo[\objf{i}]{\elf{y}}\baor \diamo[\objf{m}]{\constf{q}}.$$
Using item \ref{GLPA_Eq1_4} of Lemma \ref{GLPA_Eq1} we obtain 
$$\diamo[\objf{i}]{\elf{x}}\baor \diamo[\objf{m}]{\constf{q}}=\diamo[\objf{i}]{\elf{x}\baor \diamo[\objf{m}]{\constf{q}}}\baor \diamo[\objf{m}]{\constf{q}}=\diamo[\objf{i}]{\elf{y}\baor \diamo[\objf{m}]{\constf{q}}}\baor \diamo[\objf{m}]{\constf{q}}=\diamo[\objf{i}]{\elf{y}}\baor \diamo[\objf{m}]{\constf{q}}.$$
Now we prove compatibility for operator $\diamos[\objf{m}]$. Suppose we have $\elf{x},\elf{y}\in\algf{A}$ such that
$$\elf{x}\baor \diamo[\objf{m}]{\constf{q}}=\elf{y}\baor \diamo[\objf{m}]{\constf{q}}.$$
We need to prove that
$$\diamo[\objf{m}]{\elf{x}}\baor \diamo[\objf{m}]{\constf{q}}=\diamo[\objf{m}]{\elf{y}}\baor \diamo[\objf{m}]{\constf{q}}.$$
Using item \ref{GLPA_Eq1_1} of Lemma \ref{GLPA_Eq1} we obtain
$$\diamo[\objf{m}]{\elf{x}}\baor \diamo[\objf{m}]{\constf{q}}=\diamo[\objf{m}]{\elf{x}}\baor \diamo[\objf{m}]{\diamo[\objf{m}]{\constf{q}}} \baor \diamo[\objf{m}]{\constf{q}}=\diamo[\objf{m}]{\elf{x}\baor \diamo[\objf{m}]{\constf{q}}} \baor \diamo[\objf{m}]{\constf{q}}.$$
We have the same  for $\elf{y}$
$$\diamo[\objf{m}]{\elf{y}}\baor \diamo[\objf{m}]{\constf{q}}=\diamo[\objf{m}]{\elf{y}\baor \diamo[\objf{m}]{\constf{q}}} \baor \diamo[\objf{m}]{\constf{q}}.$$
Hence 
$$\diamo[\objf{m}]{\elf{x}}\baor \diamo[\objf{m}]{\constf{q}}=\diamo[\objf{m}]{\elf{y}}\baor \diamo[\objf{m}]{\constf{q}}.$$
We have proved that $\fquotalg^{\typef{A},\constf{q}}(\algf{A})$ is well-defined. We denote the homomorphism from the algebra $\purfr^{\typef{A},\typef{B}}(\algf{A})$ to $\fquotalg^{\typef{A},\constf{q}}(\algf{A})$ that maps a given element $\elf{x}$ to the equivalent class $[\elf{x}]$ by $\fquotalgm^{\typef{A},\constf{q}}_{\algf{A}}$.

Suppose $\algf{A}$ is a $\typef{B}$-algebra. We define an $\typef{A}$-algebra $\squotalg^{\typef{A},\constf{q}}(\algf{A})$. Here we denote $\squotalg^{\typef{A},\constf{q}}(\algf{A})$ by $\algf{C}$. $\algf{B}$ is a factor algebra of the $\{\constf{q},\diamos[\objf{m}]\}$-puration of the algebra $\algf{A}$; the corresponding quotient relation is
$$\elf{x}\sim\elf{y}\defiff \elf{x}\baand^{\algf{A}} \diamoi[\objf{m}]{\algf{A}}{\constf{q}^{\algf{A}}}=\elf{y}\baand^{\algf{A}} \diamoi[\objf{m}]{\algf{A}}{\constf{q}^{\algf{A}}}.$$ Clearly, $\sim$ is an equivalence relation. Boolean operations obviously compatible with $\sim$. The fact that $\diamos[\objf{i}]$ is compatible with $\sim$ for $\objf{i}>_{\ordf{a}}\objf{m}$ can be proved with the use of item \ref{GLPA_Eq1_1} of Lemma \ref{GLPA_Eq1}. Hence $\algf{B}$ is well-defined. The algebra $\squotalg^{\typef{B},\constf{q}}(\algf{A})$ is an extension of the algebra $\algf{B}$. In order to complete the definition of $\algf{C}$ we need to give the interpretation of $\diamos[\objf{m}]$. We put 
$$\diamoi[\objf{m}]{\algf{C}}{[\elf{x}]}=[\diamoi[\objf{m}]{\algf{A}}{\elf{x}\baand^{\algf{A}}\diamoi[\objf{m}]{\algf{A}}{\constf{q}^{\algf{A}}}}].$$
Obviously, this definition of $\diamosi[\objf{m}]{\algf{C}}$ doesn't depend of the choice of $\elf{x}$ from a quotient class. Let us check that $\algf{C}$  is $\typef{A}$-algebra. For this check it sufficient to show that all axioms of $\typef{A}$-algebras with $\diamos[\objf{m}]$ holds in $\algf{C}$. It can be done straightforward for  axioms \ref{GLPA_Ax1}, \ref{GLPA_Ax2}, and \ref{GLPA_Ax3} of $\thrf{GLP}$-algebras. Now we prove that  axiom \ref{GLPA_Ax4} of $\thrf{GLP}$-algebras holds in $\algf{C}$. Let us work in $\algf{A}$. We need to show that
 equation
$$\diamo[\objf{m}]{\elf{x}\baand\diamo[\objf{m}]{\constf{q}}}\baand\diamo[\objf{m}]{\constf{q}}= (\diamo[\objf{i}]{\elf{x}}\baor \diamo[\objf{m}]{\elf{x}\baand\diamo[\objf{m}]{\constf{q}}})\baand\diamo[\objf{m}]{\constf{q}}$$
holds in $\algf{A}$ for all $\elf{x}\in \algf{A}$ and $\objf{i}>_{\ordf{a}}\objf{m}$.
From item \ref{GLPA_Eq1_2} of Lemma \ref{GLPA_Eq1} it follows that
$$\begin{aligned}\diamo[\objf{m}]{\elf{x}\baand\diamo[\objf{m}]{\constf{q}}}&= \diamo[\objf{i}]{\elf{x}\baand\diamo[\objf{m}]{\constf{q}}}\baor \diamo[\objf{m}]{\elf{x}\baand\diamo[\objf{m}]{\constf{q}}}=(\diamo[\objf{i}]{\elf{x}}\baand\diamo[\objf{m}]{\constf{q}})\baor \diamo[\objf{m}]{\elf{x}\baand\diamo[\objf{m}]{\constf{q}}}\\ &=(\diamo[\objf{i}]{\elf{x}}\baor \diamo[\objf{m}]{\elf{x}\baand\diamo[\objf{m}]{\constf{q}}})\baand \diamo[\objf{m}]{\constf{q}}.\end{aligned}$$
Hence the required equation holds in $\algf{A}$. Axiom \ref{GLPA_Ax5} can be checked in the same way as the axiom \ref{GLPA_Ax4} (with the use of item \ref{GLPA_Eq1_3} of Lemma \ref{GLPA_Eq1} instead of item \ref{GLPA_Eq1_2})   and we omit this check. Hence $\algf{C}$ is an $\typef{A}$-algebra. 

Suppose $\algf{A}$ is a $\typef{B}$-algebra.  There is a homomorphism
$$\fqprsqiso_{\algf{A}}^{\typef{A},\constf{q}}\colon \purfr^{\typef{C},\typef{A}}(\fquotalg^{\typef{A},\constf{q}}(\algf{A}))\prodfr\purfr^{\typef{C},\typef{A}}( \squotalg^{\typef{A},\constf{q}}(\algf{A}))\to \purfr^{\typef{C},\typef{B}}(\algf{A}),$$
$$\fqprsqiso_{\algf{A}}^{\typef{A},\constf{q}}\colon ([\elf{x}],[\elf{y}])\longmapsto (\elf{x}\baor^{\algf{A}}\diamoi[\objf{m}]{\algf{A}}{\constf{q}^{\algf{A}}})\baand^{\algf{A}}(\elf{y}\baor^{\algf{A}}\bacmpi{\algf{A}}{\diamoi[\objf{m}]{\algf{A}}{\constf{q}^{\algf{A}}}}).$$
Straightforward check shows that $\fqprsqiso_{\algf{A}}$ is a well-defined function, homomorphism, and isomorphism.

We are interested in the case when $\fqprsqiso_{\algf{A}}^{\typef{A},\constf{q}}$  is actually a homorphism of $\fquotalg^{\typef{A},\constf{q}}(\algf{A})\linprodfr \squotalg^{\typef{A},\constf{q}}(\algf{A})$ to $\purfr^{\tembf{r}}(\algf{A})$.

Suppose $\algf{A}$ is a linear $\typef{B}$-algebra. Then we define $$\fqlprsqiso_{\algf{A}}^{\typef{A},\constf{q}}\colon \fquotalg^{\typef{A},\constf{q}}(\algf{A})\linprodfr \squotalg^{\typef{A},\constf{q}}(\algf{A})\to\purfr^{\typef{A},\typef{B}}(\algf{A})$$ is the only $\funf{f}\colon\fquotalg^{\typef{A},\constf{q}}(\algf{A})\linprodfr \squotalg^{\typef{A},\constf{q}}(\algf{A})\to\purfr^{\typef{A},\typef{B}}(\algf{A})$ such that $\purfr^{\typef{C},\typef{A}}(\funf{f})=\fqprsqiso_{\algf{A}}^{\typef{A},\constf{q}}$. In order to check correctness of the definition of $\fqlprsqiso_{\algf{A}}^{\typef{A},\constf{q}}$ we prove
\begin{enumerate}
\item \label{fqlprsqiso_corr_c1}$\algf{A}\models \forall \fovarf{x}(\fonot \diamo[\objf{m}]{\constf{q}}\ge\fovarf{x}\foimp \diamo[\objf{m}]{\fovarf{x}}=\diamo[\objf{m}]{\fovarf{x}\baor \diamo[\objf{m}]{\constf{q}}} \baor\diamo[\objf{m}]{\constf{q}})$;
\item \label{fqlprsqiso_corr_c2} $\algf{A}\models \forall \fovarf{x}(\diamo[\objf{m}]{\constf{q}}\ge\fovarf{x}\foimp \diamo[\objf{m}]{\fovarf{x}}=\diamo[\objf{m}]{\constf{q}}\baand \diamo[\objf{m}]{\diamo[\objf{m}]{\constf{q}}\baand \fovarf{x}})$.
\end{enumerate}
Items \ref{fqlprsqiso_corr_c1} and \ref{fqlprsqiso_corr_c2} correspond to the different cases in the definition of the interpretation of $\diamos[\objf{m}]$ in linear product. Item \ref{fqlprsqiso_corr_c2} obviously holds. Now we prove item \ref{fqlprsqiso_corr_c1}. Suppose $\elf{x}\in\algf{A}$ such that $\diamoi[\objf{m}]{\algf{A}}{\constf{q}^{\algf{A}}}\not\ge^{\algf{A}}\elf{x}$. Then  $\constf{q}^{\algf{A}}\not\algl_{\objf{m}}^{\algf{A}}\elf{x}$. Hence either $\constf{q}^{\algf{A}}\algs_{\objf{m}}^{\algf{A}}\elf{x}$ or $\elf{x}\algl_{\objf{m}}^{\algf{A}} \constf{q}^{\algf{A}}$. Therefore $\diamoi[\objf{m}]{\algf{A}}{\elf{x}}\ge^{\algf{A}}\diamoi[\objf{m}]{\algf{A}}{\constf{q}^{\algf{A}}}$. Thus $\diamo[\objf{m}]{\elf{x}}=\diamo[\objf{m}]{\elf{x}\baor \diamo[\objf{m}]{\constf{q}}} \baor\diamo[\objf{m}]{\constf{q}}$. This finishes the proof of correctness of the definition of $\fqlprsqiso_{\algf{A}}^{\typef{A},\constf{q}}$.

\section{Free Extensions}   We call a tuple $\extsqtf{E}=(\tembf{l},\algf{A},\constf{q},\csetf{C})$ an {\it extension sequence type} if $\tembf{l}\colon\typef{A}\to\typef{B}$ is a normal type embedding, $\algf{A}$ is a linear $\typef{A}$-algebra, $\csetf{C}$ is a set of constant symbols, $\typef{B}\typeplus \csetf{C}$ is well-defined, and the constant symbol $\constf{q}\not \in \typef{B}\typeplus \csetf{C}$.

Suppose $\tembf{l}\colon \typef{A}\to\typef{B}$ is a normal type embedding and $\constf{q}\not\in\typef{B}$. For an $(\typef{A}\typeplus \constf{q})$-algebra $\algf{A}$ we denote by $\fextfquotfr^{\tembf{l},\constf{q}}(\algf{A})$ the $\typef{B}$-algebra $\fquotalg^{\typef{B},\constf{q}}(\fextfr^{\tembf{l}\typeplus\constf{q}}(\algf{A}))$ and we denote by $\fextfquotm^{\tembf{l},\constf{q}}_{\algf{A}}\colon \algf{A}\to\fextfquotfr^{\tembf{l},\constf{q}}(\algf{A})$ the homomorphism $\fextemb^{\typef{A}\typeplus\constf{q},\typef{B}\typeplus\constf{q}}_{\algf{A}}\comp \fquotalgm_{\fextfr^{\typef{A}\typeplus\constf{q},\typef{B}\typeplus\constf{q}}(\algf{A})}^{\typef{B},\constf{q}}$.

Suppose $\tembf{l}\colon \typef{A}\to\typef{B}$ is a normal type embedding and $\extsqtf{E}=(\tembf{l},\algf{A},\constf{q},\csetf{C})$  is a  extension sequence type. We call a non-empty sequence $\overline{\algf{H}}=(\algf{H}_1,\ldots,\algf{H}_{\natf{n}})$ an {\it extension sequence} of the type $\extsqtf{E}$ if $\algf{H}_{\natf{n}}$ is $(\typef{A}\typeplus \csetf{C})$-algebra and  $\algf{H}_{\natf{i}}$ is a $(\typef{A}\typeplus \csetf{C}\typeplus \constf{q})$-algebra, for $\natf{i}<\natf{n}$. Now we will define an $(\typef{A}\typeplus \csetf{C})$-algebra $\sqqlp_{\overline{\algf{H}}}$. If $\natf{n}=1$ then $$\sqqlp_{\overline{\algf{H}}}=\fextfr^{\tembf{l}\typeplus \csetf{C}}(\algf{H}_1).$$ Otherwise, $$\sqqlp_{\overline{\algf{H}}}=\fextfquotfr^{\tembf{l}\typeplus\csetf{C},\constf{q}}(\algf{H}_{\natf{1}}) \linprodfr \sqqlp_{(\algf{H}_{2},\ldots,\algf{H}_{\natf{n}})}.$$ 
Note that $\purfr^{\typef{A},\typef{B}\typeplus\csetf{C}}(\sqqlp_{\overline{\algf{H}}})$ is just the product
$$\begin{aligned}\purfr^{\typef{A},\typef{B}\typeplus\csetf{C}}(\fextfquotfr^{\tembf{l}\typeplus\csetf{C},\constf{q}}(\algf{H}_{1}))&\prodfr(\purfr^{\typef{A},\typef{B}\typeplus\csetf{C}}(\fextfquotfr^{\tembf{l}\typeplus\csetf{C},\constf{q}}(\algf{H}_{2}))\prodfr (\ldots \prodfr (\purfr^{\typef{A},\typef{B}\typeplus\csetf{C}}(\fextfquotfr^{\tembf{l}\typeplus\csetf{C},\constf{q}}(\algf{H}_{\natf{n}-1}))\prodfr \\ & \purfr^{\typef{A},\typef{B}\typeplus\csetf{C}}(\fextfr^{\tembf{l}\typeplus \csetf{C}}(\algf{H}_{\natf{n}})))\ldots)),\end{aligned}$$
in the natural way we encode it's elements by $\natf{n}$-tuples.   We give $\sqqlpam_{\overline{\algf{H}}}^{\extsqtf{E}}\colon \algf{A} \to \purfr^{\typef{A},\typef{B}\typeplus\csetf{C}}(\sqqlp_{\overline{\algf{H}}})$ by $$\elf{x}\longmapsto (\purfr^{\csetf{C}}(\fextfquotm_{\algf{H}_1}^{\tembf{l}\typeplus\csetf{C},\constf{q}})(\elf{x}),\ldots,\purfr^{\csetf{C}}(\fextfquotm_{\algf{H}_{\natf{n}-1}}^{\tembf{l}\typeplus\csetf{C},\constf{q}})(\elf{x}),\purfr^{\csetf{C}}(\fextemb^{\tembf{l}\typeplus\csetf{C}})(\elf{x})).$$
We define $\sqqlpm_{\overline{\algf{H}}}^{\extsqtf{E}}\colon \fextfr^{\tembf{l}}(\algf{A})\to \purfr^{\csetf{C}}(\sqqlp_{\overline{\algf{H}}})$ as the unique morphism such that $\fextemb^{\tembf{l}}_{\algf{A}}\comp \purfr^{\tembf{l}}(\sqqlpm_{\overline{\algf{H}}}^{\extsqtf{E}})=\sqqlpam_{\overline{\algf{H}}}^{\extsqtf{E}}$.
\[
\begin{diagram}
\node[3]{\purfr^{\typef{A},\typef{B}\typeplus\csetf{C}}(\sqqlp_{\overline{\algf{H}}})}
\node[3]{\purfr^{\csetf{C}}(\sqqlp_{\overline{\algf{H}}})}\\
\node{\algf{A}}
  \arrow{ene,t}{\sqqlpam_{\overline{\algf{H}}}^{\extsqtf{E}}}
  \arrow[2]{e,b}{\fextemb^{\tembf{l}}_{\algf{A}}}
\node[2]{\purfr^{\tembf{l}}(\fextfr^{\tembf{l}}(\algf{A}))}
  \arrow{n,r,..}{\purfr^{\tembf{l}}(\sqqlpm_{\overline{\algf{H}}}^{\extsqtf{E}})} 
\node[3]{\fextfr^{\tembf{l}}(\algf{A})}
  \arrow{n,r,..}{\sqqlpm_{\overline{\algf{H}}}^{\extsqtf{E}}}\\
\node[2]{\mbox{\it $\typef{A}$-algebras}}
\node[4]{\mbox{\it $\typef{B}$-algebras}}
\end{diagram}
\]

Obviously, the following two lemmas holds
\begin{lemma} \label{sqqlp_shift1} Suppose $\tembf{l}\colon \typef{A}\to\typef{B}$ is a normal type embedding, $\csetf{C},\csetf{D}$ are pairwise non-intersecting sets of constants, $\algf{A}$ is a linear $(\typef{A}\typeplus \csetf{D})$-algebra, $\algf{A}'=\purfr^{\csetf{D}}(\algf{A})$,  $\overline{\algf{H}}=(\algf{H}_1,\ldots,\algf{H}_{\natf{n}-1},\algf{H}_{\natf{n}})$ is an extension sequence of the type $\extsqtf{E}=(\tembf{l}\typeplus \csetf{D},\algf{A},\constf{q},\csetf{C})$. Then  $$\overline{\algf{H}'}=(\purfr^{\csetf{D}}(\algf{H}_1),\ldots,\purfr^{\csetf{D}}(\algf{H}_{\natf{n}-1}),\purfr^{\csetf{D}}(\algf{H}_{\natf{n}}))$$ is an extension sequence of the type  $\extsqtf{E}'=(\tembf{l},\algf{A}',\constf{q},\csetf{C})$ and there exist isomorphisms $\funf{f}\colon \fextfr^{\tembf{l}}(\algf{A}')\to\purfr^{\csetf{D}}(\fextfr^{\tembf{l}\typeplus\csetf{D}}(\algf{A}))$ and $\funf{g}\colon\sqqlp_{\overline{H'}}\to\purfr^{\csetf{D}}(\sqqlp_{\overline{H}})$ such that the following diagrams commute:
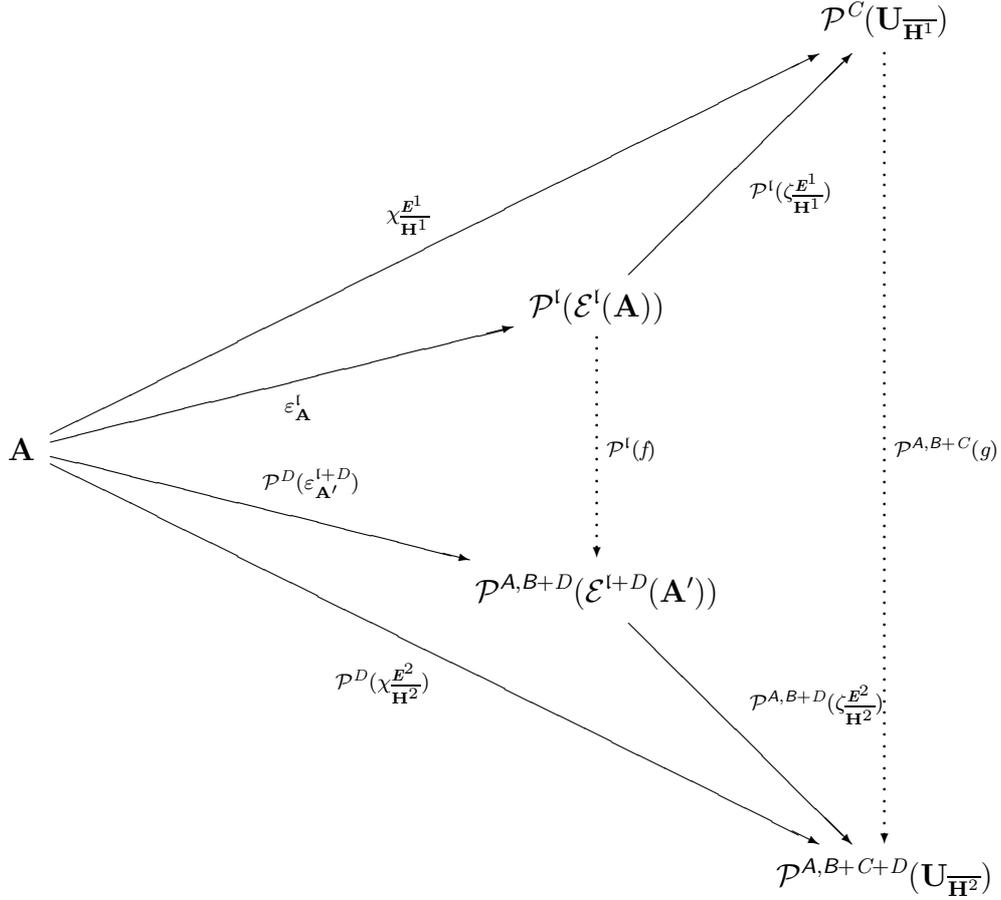
\[
\begin{diagram}
\node{\fextfr^{\tembf{l}}(\algf{A})}
\arrow{s,r,..}{\funf{f}}
\arrow[3]{e,t}{\fextemb^{\tembf{l}}_{\algf{A}}}
\node[3]{\purfr^{\csetf{C}}(\sqqlp_{\overline{\algf{H}^{1}}})}
\arrow{s,r,..}{\purfr^{\csetf{C}}(\funf{g})}\\
\node{\purfr^{\csetf{D}}(\fextfr^{\tembf{l}}(\algf{A}'))}
\arrow[3]{e,b}{\purfr^{\csetf{D}}(\fextemb^{\tembf{l}\typeplus\csetf{D}}_{\algf{A}'})}
\node[3]{\purfr^{\csetf{C}}(\sqqlp_{\overline{\algf{H}^{1}}})}\\
\end{diagram}
\]
\[
\begin{diagram}
\node{}\node{}\node{}
\node{\purfr^{\csetf{C}}(\sqqlp_{\overline{\algf{H}^{1}}})}
\arrow[6]{s,r,..}{\purfr^{\typef{A},\typef{B}\typeplus\csetf{C}}(\funf{g})}
\\
\\
\node{}\node{}
\node{\purfr^{\tembf{l}}(\fextfr^{\tembf{l}}(\algf{A}))}
\arrow{nne,b}{\purfr^{\tembf{l}}(\sqqlpm_{\overline{\algf{H}^1}}^{\extsqtf{E}^1})}
\arrow[2]{s,r,..}{\purfr^{\tembf{l}}(\funf{f})}
\\
\node{\algf{A}}
\arrow{ene,b}{\fextemb^{\tembf{l}}_{\algf{A}}}
\arrow{ese,t}{\purfr^{\csetf{D}}(\fextemb^{\tembf{l}\typeplus\csetf{D}}_{\algf{A}'})}
\arrow[3]{ne,t}{\sqqlpam_{\overline{\algf{H}^{1}}}^{\extsqtf{E}^1}}
\arrow[3]{se,b}{\purfr^{\csetf{D}}(\sqqlpam_{\overline{\algf{H}^{2}}}^{\extsqtf{E}^2})}
\\
\node{}\node{}
\node{\purfr^{\typef{A},\typef{B}\typeplus\csetf{D}}(\fextfr^{\tembf{l}\typeplus\csetf{D}}(\algf{A}'))}
\arrow{sse,t}{\purfr^{\typef{A},\typef{B}\typeplus\csetf{D}}(\sqqlpm_{\overline{\algf{H}^2}}^{\extsqtf{E}^2})}
\\
\\
\node{}\node{}\node{}
\node{\purfr^{\typef{A},\typef{B}\typeplus\csetf{C}\typeplus \csetf{D}}(\sqqlp_{\overline{\algf{H}^{2}}})}
\end{diagram}
\]
\end{lemma}

\begin{lemma} \label{sqqlp_shift2} Suppose $\tembf{l}\colon \typef{A}\to\typef{B}$ is a normal type embedding, $\csetf{C},\csetf{D}$ are pairwise non-intersecting sets of constants, $\algf{A}$ is a linear $\typef{A}$-algebra, $\extsqtf{E}_1=(\tembf{l},\algf{A},\constf{q},\csetf{C})$ and $\extsqtf{E}_2=(\tembf{l},\algf{A},\constf{q},\csetf{C}\sqcup\csetf{D})$ are extension sequence types, $\overline{\algf{H}^2}=(\algf{H}_1,\ldots,\algf{H}_{\natf{n}-1},\algf{H}_{\natf{n}})$ is an extension sequence of the type $\extsqtf{E}_2$. Then $$\overline{\algf{H}^1}=(\purfr^{\csetf{D}}(\algf{H}_1),\ldots,\purfr^{\csetf{D}}(\algf{H}_{\natf{n}-1}),\purfr^{\csetf{D}}(\algf{H}_{\natf{n}}))$$ is an extension sequence of the type $\extsqtf{E}_1$ such that $\purfr^{\csetf{D}}(\sqqlp_{\overline{\algf{H}^2}})=\sqqlp_{\overline{\algf{H}^1}}$, $\sqqlpam_{\overline{\algf{H}^{2}}}^{\extsqtf{E}^{2}}=\sqqlpam_{\overline{\algf{H}^{1}}}^{\extsqtf{E}^{1}}$, and  $\sqqlpm_{\overline{\algf{H}^{2}}}^{\extsqtf{E}^{2}}=\sqqlpm_{\overline{\algf{H}^{1}}}^{\extsqtf{E}^{1}}$.\end{lemma}

\begin{lemma} \label{sqqlpm_isomorphism}Suppose $\tembf{l}\colon \typef{A}\to\typef{B}$ is a normal type embedding, $\extsqtf{E}=(\tembf{l},\algf{A},\constf{q},\csetf{C})$ is an extension sequence type, and $\overline{\algf{H}}$ is an extension sequence of the type $\extsqtf{E}$. Then $\sqqlpm_{\overline{\algf{H}}}^{\extsqtf{E}}\colon \fextfr^{\tembf{l}}(\algf{A})\to \sqqlp_{\overline{\algf{H}}}$ is isomorphism.
\end{lemma}
We will prove Lemma \ref{sqqlpm_isomorphism} later in the section.

Suppose  $\tembf{l}\colon\typef{A}\to\typef{B}$ is a normal type embedding and $\extsqtf{E}=(\tembf{l},\algf{A},\constf{q},\csetf{C})$  is an  extension sequence type. For an extension sequence $\overline{\algf{H}}$ of the type $\extsqtf{E}$ we denote by $\sqext_{\overline{\algf{H}}}$ the only strong constant extension by the set $\csetf{C}$ of $\fextfr^{\tembf{l}}(\algf{A})$ that is isomorphic to $\sqqlp_{\overline{\algf{H}}}$ under an isomorphism $\funf{f}\colon \sqext_{\overline{\algf{H}}}\to \sqqlp_{\overline{\algf{H}}}$ such that $\purfr^{\csetf{C}}(\funf{f})=\sqqlpm^{\extsqtf{E}}_{\overline{\algf{H}}}$.

\begin{lemma} \label{ext_sqext_prop}Suppose $\tembf{l}\colon\typef{A}\to \typef{B}$ is a normal type embedding and $\algf{A}$ is a linear $\typef{A}$-algebra. Then for every strong constant extension $\algf{B}$ of $\fextfr^{\tembf{l}}(\algf{A})$ by a finite set of constants $\csetf{C}$ and $\constf{q}\not \in \typef{B}\typeplus \csetf{C}$ there exists an extension sequence $\overline{\algf{H}}$ of the type $\extsqtf{E}=(\tembf{l},\algf{A},\constf{q},\csetf{C})$ such that $\sqext_{\overline{\algf{H}}}=\algf{B}$.
\end{lemma}
We will prove Lemma \ref{ext_sqext_prop} later in the section.

Suppose $\constf{q}$ is a constant symbol and $\typef{A}=(\ordf{a},\csetf{A})$ is a type such that $\constf{q}\in \typef{A}$ and $\boxos[\ione]\in \typef{A}$.  We define the mapping $\sqtrs{\constf{q}}{\typef{A}}$ of $\thrf{GLPA}_{\typef{A}}$-terms
\begin{itemize}
\item $\sqtr{\constf{q}}{\typef{A}}{\baz}=\boxo[\ione]{\constf{q}}\baor\baz$;
\item $\sqtr{\constf{q}}{\typef{A}}{\bau}=\boxo[\ione]{\constf{q}}\baor\bau$;
\item $\sqtr{\constf{q}}{\typef{A}}{\constf{c}}=\boxo[\ione]{\constf{q}}\baor\constf{c}$, where $\constf{c}$ is constant symbol;
\item $\sqtr{\constf{q}}{\typef{A}}{\fovarf{x}}=\boxo[\ione]{\constf{q}}\baor\fovarf{x}$, where $\fovarf{x}$ is first-order variable;
\item $\sqtr{\constf{q}}{\typef{A}}{\termf{t}_1\baand \termf{t}_2}=\boxo[\ione]{\constf{q}}\baor(\sqtr{\constf{q}}{\typef{A}}{\termf{t}_1}\baand \sqtr{\constf{q}}{\typef{A}}{\termf{t}_2})$;
\item $\sqtr{\constf{q}}{\typef{A}}{\termf{t}_1\baor \termf{t}_2}=\boxo[\ione]{\constf{q}}\baor(\sqtr{\constf{q}}{\typef{A}}{\termf{t}_1}\baor \sqtr{\constf{q}}{\typef{A}}{\termf{t}_2})$;
\item $\sqtr{\constf{q}}{\typef{A}}{\bacmp{\termf{t}}}=\boxo[\ione]{\constf{q}}\baor\bacmp{\sqtr{\constf{q}}{\typef{A}}{\termf{t}}}$;
\item $\sqtr{\constf{q}}{\typef{A}}{\diamo[\objf{x}]{\termf{t}}}=\boxo[\ione]{\constf{q}}\baor\diamo[\objf{x}]{\sqtr{\constf{q}}{\typef{A}}{\termf{t}}}$.
\end{itemize}
Similarly, for a propositional variable $\prvarf{x}$ we define the mapping $\sqftrs{\prvarf{x}}{\ordf{a}}$ of $\thrf{GLP}_{\ordf{a}}$-formulas
\begin{itemize}
\item $\sqftr{\prvarf{x}}{\ordf{a}}{\top}= \boxm[\ione]\prvarf{x}\pror\top$;
\item $\sqftr{\prvarf{x}}{\ordf{a}}{\bot}= \boxm[\ione]\prvarf{x}\pror\bot$;
\item $\sqftr{\prvarf{x}}{\ordf{a}}{\prvarf{y}}= \boxm[\ione]\prvarf{x}\pror\prvarf{y}$, for a propositional variable $\prvarf{y}$;
\item $\sqftr{\prvarf{x}}{\ordf{a}}{\prflf{f}\pror\prflf{p}}= \boxm[\ione]\prvarf{x}\pror (\sqftr{\prvarf{x}}{\ordf{a}}{\prflf{f}}\pror\sqftr{\prvarf{x}}{\ordf{a}}{\prflf{p}})$, for $\thrf{GLP}_{\ordf{a}}$-formulas $\prflf{f},\prflf{p}$;
\item $\sqftr{\prvarf{x}}{\ordf{a}}{\prflf{f}\prand\prflf{p}}= \boxm[\ione]\prvarf{x}\pror (\sqftr{\prvarf{x}}{\ordf{a}}{\prflf{f}}\prand\sqftr{\prvarf{x}}{\ordf{a}}{\prflf{p}})$, for $\thrf{GLP}_{\ordf{a}}$-formulas $\prflf{f},\prflf{p}$;
\item $\sqftr{\prvarf{x}}{\ordf{a}}{\prflf{f}\primp\prflf{p}}= \boxm[\ione]\prvarf{x}\pror (\sqftr{\prvarf{x}}{\ordf{a}}{\prflf{f}}\primp\sqftr{\prvarf{x}}{\ordf{a}}{\prflf{p}})$, for $\thrf{GLP}_{\ordf{a}}$-formulas $\prflf{f},\prflf{p}$;
\item $\sqftr{\prvarf{x}}{\ordf{a}}{\prnot \prflf{f}}= \boxm[\ione]\prvarf{x}\pror (\prnot \sqftr{\prvarf{x}}{\ordf{a}}{\prflf{f}})$, for a $\thrf{GLP}_{\ordf{a}}$-formula $\prflf{f}$;
\item $\sqftr{\prvarf{x}}{\ordf{a}}{\boxm[\objf{x}]{ \prflf{f}}}= \boxm[\ione]\prvarf{x}\pror (\boxm[\objf{x}] \sqftr{\prvarf{x}}{\ordf{a}}{\prflf{f}})$, for a $\thrf{GLP}_{\ordf{a}}$-formula $\prflf{f}$ and $\boxos[\objf{x}]\in\typef{A}$.
\end{itemize}
Obviously,  for a $\thrf{GLPA}_{\typef{A}}$-term  $\termf{t}$ the formula $\sqftr{\sppvar{\constf{q}}}{\ordf{a}}{\trtfl{\termf{t}}}$ is $\thrf{GLPA}_{\typef{A}}$-equivalent to $\trtfl{(\sqftr{\constf{q}}{\typef{A}}{\termf{t}})}$

\begin{lemma} \label{sqftr} Suppose $\ordf{a}$ is a linear ordered set, $\ione$ is the minimal element of $\ordf{a}$, $\itwo$ is the minimal element of $\ordf{a}\setminus\{\ione\}$, $\prflf{f}$ and $\prflf{p}$ are formulas from $\lang{\thrf{GLP}_{\ordf{a}}}$, and $\prvarf{x}$ is a propositional variable such that $\prvarf{x}$ doesn't occur in $\prflf{f}$, $\boxm[\ione]$ doesn't occur in $\prflf{p}$, and $\thrf{GLP}_{\ordf{a}}\nvdash (\prflf{p}\prand \boxm[\itwo]\prflf{p})\primp \prvarf{x}$. Then 
$$\thrf{GLP}_{\ordf{a}}\vdash (\prflf{p}\prand \boxm[\ione]\prflf{p})\primp\prflf{f}\iff \thrf{GLP}_{\ordf{a}}\vdash (\prflf{p}\prand \boxm[\ione]\prflf{p})\primp\sqftr{\prvarf{x}}{\ordf{a}}{\prflf{f}}.$$
\end{lemma}

The proof of Lemma \ref{sqftr} uses technique that is significantly different from the technique of the other parts of the paper. We prove Lemma \ref{sqftr} in Section \ref{syntactical_facts}.
\begin{lemma} \label{sqtr} Suppose $\tembf{l}\colon\typef{A}\to\typef{B}$ is a normal embedding, $\constf{q}\not\in\typef{A}$ and $\algf{A}$ is an $(\typef{A}\typeplus \constf{q})$-algebra such that $\purfr^{ \constf{q}}(\algf{A})$ is a constant complete $\typef{A}$-algebra and $\constf{q}^{\algf{A}}\ne\bau^{\algf{A}}$. Then for a closed $\thrf{GLPA}_{\typef{B}}$-terms $\termf{t}_1$ and $\termf{t}_2$ we have
$$\fextfr^{\tembf{l}}(\algf{A})\models \termf{t}_1=\termf{t}_2\iff \fextfr^{\tembf{l}}(\algf{A})\models \sqtr{\constf{q}}{\typef{B}}{\termf{t}_1}=\sqtr{\constf{q}}{\typef{B}}{\termf{t}_2}$$
\end{lemma}
\begin{proof}Suppose  $\ordf{a}$ is the operator index set of $\typef{A}$  and $\ordf{b}$ is the operator index set for $\typef{B}$.

 Clearly, $\fextfr^{\tembf{l}}(\algf{A})\models \termf{t}_1=\termf{t}_2$ iff there exist closed $\thrf{GLPA}_{\typef{A}}$-terms  $\termf{w}_1,\termf{u}_1,\ldots,\termf{w}_{\natf{n}},\termf{u}_{\natf{n}}$ such that
$$\algf{A}\models \termf{w}_1=\termf{u}_1 \foand\ldots\foand\termf{w}_{\natf{n}}=\termf{u}_{\natf{n}}$$
and
\begin{equation}\label{sqtr_eq1}\thrf{GLPA}_{\typef{B}}\vdash (\termf{w}_1=\termf{u}_1 \foand\ldots\foand\termf{w}_{\natf{n}}=\termf{u}_{\natf{n}} )\foimp \termf{t}_1=\termf{t}_2.\end{equation}
(\ref{sqtr_eq1}) is equivalent to
$$\thrf{GLP}_{\ordf{b}}+(\trtfl{\termf{w}_1}\prequiv  \trtfl{\termf{u}_1}) \prand\ldots \prand (\trtfl{\termf{w}_{\natf{n}}}\prequiv  \trtfl{\termf{u}_{\natf{n}}})\vdash \trtfl{\termf{t}_1}\prequiv\trtfl{\termf{t}_2}.$$
Because $\constf{q}^{\algf{A}}\ne\bau^{\algf{A}}$, for every $\termf{w}_1,\termf{u}_1,\ldots,\termf{w}_{\natf{n}},\termf{u}_{\natf{n}}$ such that
$$\algf{A}\models \termf{w}_1=\termf{u}_1 \foand\ldots\foand\termf{w}_{\natf{n}}=\termf{u}_{\natf{n}}$$
we have
$$\thrf{GLP}_{\ordf{a}}+(\trtfl{\termf{w}_1}\prequiv  \trtfl{\termf{u}_1}) \prand\ldots \prand (\trtfl{\termf{w}_{\natf{n}}}\prequiv  \trtfl{\termf{u}_{\natf{n}}})\nvdash \trtfl{\termf{t}_1}\prequiv\trtfl{\termf{t}_2}$$
Hence from Lemma \ref{sqftr} it follows that (\ref{sqtr_eq1}) is equivalent to  
$$\thrf{GLP}_{\ordf{b}}+(\trtfl{\termf{w}_1}\prequiv  \trtfl{\termf{u}_1}) \prand\ldots \prand (\trtfl{\termf{w}_{\natf{n}}}\prequiv  \trtfl{\termf{u}_{\natf{n}}})\vdash  \sqftr{\ordf{b}}{\sppvar{\constf{q}}}{\trtfl{\termf{t}_1}\prequiv\trtfl{\termf{t}_2}}.$$
Clearly,  $$\thrf{GLP}_{\ordf{b}}\vdash \sqftr{\ordf{b}}{\sppvar{\constf{q}}}{\trtfl{\termf{t}_1}\prequiv\trtfl{\termf{t}_2}}\prequiv(\sqftr{\ordf{b}}{\sppvar{\constf{q}}}{\trtfl{\termf{t}_1}}\prequiv\sqftr{\ordf{b}}{\sppvar{\constf{q}}}{\trtfl{\termf{t}_2}}).$$
Therefore (\ref{sqtr_eq1}) is equivalent to
$$\thrf{GLPA}_{\typef{B}}\vdash \termf{w}_1=\termf{u}_1 \foand\ldots\foand\termf{w}_{\natf{n}}=\termf{u}_{\natf{n}} \foimp \sqtr{\typef{B}}{\constf{q}}{\termf{t}_1}=\sqtr{\typef{B}}{\constf{q}}{\termf{t}_2}.$$
Also, $\fextfr^{\tembf{l}}(\algf{A})\models \sqtr{\typef{B}}{\constf{q}}{\termf{t}_1}=\sqtr{\typef{B}}{\constf{q}}{\termf{t}_2}$ iff there exists closed $\thrf{GLPA}_{\typef{A}}$-terms  $\termf{w}_1,$ $\termf{u}_1,\ldots,\termf{w}_{\natf{n}},\termf{u}_{\natf{n}}$ such that
$$\algf{A}\models \termf{w}_1=\termf{u}_1 \foand\ldots\foand\termf{w}_{\natf{n}}=\termf{u}_{\natf{n}}$$
and
\begin{equation}\thrf{GLPA}_{\typef{B}}\vdash (\termf{w}_1=\termf{u}_1 \foand\ldots\foand\termf{w}_{\natf{n}}=\termf{u}_{\natf{n}} )\foimp \sqtr{\typef{B}}{\constf{q}}{\termf{t}_1}=\sqtr{\typef{B}}{\constf{q}}{\termf{t}_2}.\end{equation}
Henceforth the lemma holds.
\end{proof}

From Lemma \ref{sqtr} we conclude
\begin{corollary} \label{sqtr_cor} Suppose $\tembf{l}\colon\typef{A}\to\typef{B}$ is a normal type embedding, $\constf{q}\not\in\typef{A}$ and $\algf{A}$ is an $(\typef{A}\typeplus \constf{q})$-algebra such that $\purfr^{ \constf{q}}(\algf{A})$ is a constant complete $\typef{A}$-algebra and $\constf{q}^{\algf{A}}\ne\bau^{\algf{A}}$. Then $\squotalg^{\typef{B},\constf{q}}(\fextfr^{\tembf{l}\typeplus\constf{q}}(\algf{A}))$ is isomorphic to $\fextfr^{\tembf{l}}(\purfr^{\constf{q}}(\algf{A}))$.
\end{corollary}

Now we will prove Lemma \ref{sqqlpm_isomorphism}.  
\begin{proof}
We prove the lemma by induction on the length of $\overline{\algf{H}}$. From Lemma \ref{fextfr_const_shift} it follows that in the case of one element  $\overline{\algf{H}}$ the lemma holds. Suppose $\overline{\algf{H}}=(\algf{H}_1,\ldots,\algf{H}_{\natf{n}})$, where $\natf{n}\ge 2$. We denote by $\overline{\algf{G}}$ the sequence $(\algf{H}_2,\algf{H}_3,\ldots,\algf{H}_{\natf{n}})$ From Lemma \ref{sqqlp_shift1} it follows that we can consider only the case of constant complete algebra $\algf{A}$. We consider the algebra $$\purfr^{\csetf{C}}(\fextfquotfr^{\tembf{l}\typeplus \csetf{C},\constf{q}}(\algf{H}_1))\linprodfr \fextfr^{\tembf{l}}(\algf{A}).$$
 We denote the homomorphism $\idf_{\purfr^{\csetf{C}}(\fextfquotfr^{\tembf{l}\typeplus \csetf{C},\constf{q}}(\algf{H}_1))}\linprodfr \sqqlpm_{\overline{\algf{G}}}^{\extsqtf{E}}$ by $\funf{f}$, $$\funf{f}\colon \purfr^{\csetf{C}}(\fextfquotfr^{\tembf{l}\typeplus \csetf{C},\constf{q}}(\algf{H}_1))\linprodfr \fextfr^{\tembf{l}}(\algf{A})\to \purfr^{\csetf{C}}(\sqqlp^{\extsqtf{E}}_{\overline{\algf{H}}}).$$ From inductive hypothesis we know that $\sqqlpm_{\overline{\algf{G}}}^{\extsqtf{E}}$ is an isomorphism. Hence $\funf{f}$ is an isomorphism. Because $\algf{A}$ is constant complete, there is at most one homorphism from $\fextfr^{\tembf{l}}(\algf{A})$ to any given algebra. Hence in order to prove the inductive hypothesis  we only need to show that $\purfr^{\csetf{C}}(\fextfquotfr^{\tembf{l}\typeplus \csetf{C},\constf{q}}(\algf{H}_1))\linprodfr \fextfr^{\tembf{l}}(\algf{A})$ and $\fextfr^{\tembf{l}}(\algf{A})$ are isomorphic.

 We denote the $(\typef{A}\typeplus\constf{q})$-algebra $\purfr^{\csetf{C}}(\algf{H}_1)$ by $\algf{B}$. Clearly, $\purfr^{\constf{q}}(\algf{B})=\algf{A}$. Because $\fextfr^{\tembf{l}\typeplus \constf{q}}(\algf{B})$ is linear, the algebra $\fextfr^{\tembf{l}}(\algf{A})$  is isomorphic to
$$\fquotalg^{\typef{B},\constf{q}}(\fextfr^{\tembf{l}\typeplus \constf{q}}(\algf{B}))\linprodfr \squotalg^{\typef{B},\constf{q}}(\fextfr^{\tembf{l}\typeplus \constf{q}}(\algf{B})).$$
Obviously, $\fquotalg^{\typef{B},\constf{q}}(\fextfr^{\tembf{l}\typeplus \constf{q}}(\algf{B}))$ is isomorphic to $\purfr^{\csetf{C}}(\fextfquotfr^{\tembf{l}\typeplus \csetf{C},\constf{q}}(\algf{H}_1))$. From Corollary \ref{sqtr_cor} it follows that $\fextfr^{\tembf{l}}(\purfr^{\constf{q}}(\algf{B}))$ is isomorphic to $\squotalg^{\typef{B},\constf{q}}(\fextfr^{\tembf{l}\typeplus \constf{q}}(\algf{B}))$. Hence $\squotalg^{\typef{B},\constf{q}}(\fextfr^{\tembf{l}\typeplus \constf{q}}(\algf{B}))$ is isomorphic to $\fextfr^{\tembf{l}}(\algf{A})$. Therefore $\purfr^{\csetf{C}}(\fextfquotfr^{\tembf{l}\typeplus \csetf{C},\constf{q}}(\algf{H}_1))\linprodfr \fextfr^{\tembf{l}}(\algf{A})$ and $\fextfr^{\tembf{l}}(\algf{A})$ are isomorphic.\end{proof}

Now we will prove Lemma \ref{ext_sqext_prop}.
\begin{proof} From Lemma \ref{sqqlp_shift1} it follows that the general case of the lemma follows from the case of constant complete algebra $\algf{A}$. Further, we will assume that $\algf{A}$ is constant complete. We choose a finite sequence of closed $\thrf{GLPA}_{\typef{B}}$-terms $\termf{t}_1,\ldots,\termf{t}_{\natf{n}}$ such that for every $\constf{c}\in\csetf{C}$ we have $\algf{B}\models \constf{c}=\termf{t}_{\natf{i}}$ for some $\natf{i}$ from $1$ to $\natf{n}$ and every proper subterm of every $\termf{t}_{\natf{i}}$ is graphically equal to $\termf{t}_{\natf{j}}$ for some $\natf{j}$. Now we choose some set of fresh constants $\csetf{E}=\{\constf{e}_1,\ldots,\constf{e}_{\natf{n}}\}$.  We consider the strong constant extension $\algf{B}'$ of $\fextfr^{\tembf{l}}(\algf{A})$ by the set of constants $\csetf{E}$ with interpretations $\constf{e}_{\natf{i}}^{\algf{B}'}=\termf{t}_{\natf{i}}^{\fextfr^{\tembf{l}}(\algf{A})}$. We are going to find  an extension sequence $\overline{\algf{S}}$ of the type $\extsqtf{F}=(\tembf{l},\algf{A},\constf{q},\csetf{E})$ such that $\sqext_{\overline{\algf{S}}}=\algf{B}'$. We denote by $\tembf{r}\colon \typef{A}\typeplus \csetf{C}\to\typef{A}\typeplus \csetf{E}$ such that it maps symbols from $\typef{A}$ to themselves and $\constf{c}\in\csetf{C}$ to $\constf{e}_{\natf{i}}$, where $\natf{i}$  is a number from $1$ to $\natf{n}$ such that  $\algf{B}\models \termf{t}_{\natf{i}}=\constf{c}$. Obviously from such a $\overline{\algf{S}}$ we can construct the required $\overline{\algf{H}}$ by applying $\purfr^{\tembf{r}}$ to elements of $\overline{\algf{S}}$. 

For $\natf{i}$ from $0$ to $\natf{n}$ we denote by $\csetf{E}_{\natf{i}}$ the set $\{\constf{e}_1,\ldots,\constf{e}_{\natf{i}}\}$. By induction on $\natf{i}$ from $0$ to $\natf{n}$ we prove that there exists an extension sequence $\overline{\algf{S}^{\natf{i}}}=(\algf{S}^{\natf{i}}_1,\ldots,\algf{S}^{\natf{i}}_{\natf{k}_{\natf{i}}})$ of the type $\extsqtf{F}_{\natf{i}}=(\tembf{l},\algf{A},\constf{q},\csetf{E}_{\natf{i}})$ such that 
\begin{itemize}
\item $\sqext_{\overline{\algf{S}^{\natf{i}}}}=\purfr^{\csetf{E}\setminus \csetf{E}_{\natf{i}}}(\algf{B}')$,
\item for every $\natf{j}<\natf{i}$ we have $\algf{S}^{\natf{i}}_{\natf{l}}\models \constf{e}_{\natf{j}}=\bau \foor  \constf{e}_{\natf{j}}=\baz$, for all $\natf{l}$ from $1$ to $\natf{k}_{\natf{i}}$  if $\termf{t}_{\natf{j}}$ is $\diamo[\ione]{\termf{w}}$ for some $\termf{w}$ .
\end{itemize} 
From the inductive hypothesis for $\natf{i}=\natf{n}$ it follows that required $\overline{\algf{S}}$ exists and further it follows that the required $\overline{\algf{H}}$ exists. The case of $\natf{i}=0$ is trivial. Now we prove the inductive hypothesis for $\natf{i}+1$ using the inductive hypothesis for $\natf{i}$. Suppose $\termf{t}_{\natf{i}+1}$ doesn't starts with $\diamos[\ione]$. Then we can find a closed $\thrf{GLPA}_{\typef{B}\typeplus \csetf{E}_{\natf{i}}}$-term $\termf{w}$ such that there are at most one non-constant functional symbol in $\termf{w}$, there are no $\diamos[\ione]$ in $\termf{w}$, and $\algf{B}'\models \termf{w}=\constf{e}_{\natf{i}+1}$.  We give $\overline{\algf{S}^{\natf{i}+1}}$ as following:
\begin{itemize}
\item  $\natf{k}_{\natf{i}+1}=\natf{k}_{\natf{i}}$,
\item for all $\natf{j}$ from $1$ to $\natf{k}_{\natf{i}}$, the algebra $\algf{S}^{\natf{i}+1}_{\natf{j}}$ is the strong constant extension of $\algf{S}^{\natf{i}}_{\natf{j}}$ by $\constf{e}_{\natf{i}+1}$ with the interpretation $\constf{e}_{\natf{i}+1}^{\algf{S}^{\natf{i}+1}_{\natf{j}}}=\termf{w}^{\algf{S}^{\natf{i}}_{\natf{j}}}$.
\end{itemize}
Simple check shows that for $\overline{\algf{S}^{\natf{i}+1}}$ the induction hypothesis holds.

Further we assume that $\termf{t}_{\natf{i}+1}$ starts with $\diamos[\ione]$. Obviously, we have $\purfr^{\csetf{E}\setminus\csetf{E}_{\natf{i}}}(\algf{B}')\models \diamos[\ione]{\constf{c}}=\constf{e}_{\natf{i}+1}$ for some constant symbol $\constf{c}\in \typef{A}\typeplus\csetf{E}_{\natf{i}}$. We consider  the minimal $\natf{u}$ from $1$ to $\natf{k}_{\natf{i}}$ such that $\fextfquotfr^{\tembf{l}\typeplus\csetf{E}_{\natf{i}},\constf{q}}(\algf{S}^{\natf{i}}_{\natf{u}})\not\models \constf{c}=\bau$; if there are no such a number $\natf{u}$ then we give $\overline{\algf{S}^{\natf{i}+1}}$ as following:
\begin{itemize}
\item  $\natf{k}_{\natf{i}+1}=\natf{k}_{\natf{i}}$,
\item for all $\natf{j}$ from $1$ to $\natf{k}_{\natf{i}}$ algebra $\algf{S}^{\natf{i}+1}_{\natf{j}}$ is the strong constant extension of $\algf{S}^{\natf{i}}_{\natf{j}}$ by $\constf{e}_{\natf{i}+1}$ with interpretation $\constf{e}_{\natf{i}+1}^{\algf{S}^{\natf{i}+1}_{\natf{j}}}=\bau^{\algf{S}^{\natf{i}}_{\natf{j}}}$.
\end{itemize}
Simple check shows that for $\overline{\algf{S}^{\natf{i}+1}}$ the induction hypothesis holds.

Further we assume that we have found such a number $\natf{u}$. If  $\fextfquotfr^{\tembf{l}\typeplus\csetf{E}_{\natf{i}},\constf{q}}(\algf{S}^{\natf{i}}_{\natf{u}})\models \diamo[\ione]{\constf{c}}=\baz$ the we give $\overline{\algf{S}^{\natf{i}+1}}$ as following:
\begin{itemize}
\item  $\natf{k}_{\natf{i}+1}=\natf{k}_{\natf{i}}$,
\item for all $\natf{j}$ from $1$ to $\natf{u}$, the algebra $\algf{S}^{\natf{i}+1}_{\natf{j}}$ is the strong constant extension of $\algf{S}^{\natf{i}}_{\natf{j}}$ by $\constf{e}_{\natf{i}+1}$ with interpretation $\constf{e}_{\natf{i}+1}^{\algf{S}^{\natf{i}+1}_{\natf{j}}}=\baz^{\algf{S}^{\natf{i}}_{\natf{j}}}$;
\item for all $\natf{j}$ from $\natf{u}+1$ to $\natf{k}_{\natf{i}}$, the algebra $\algf{S}^{\natf{i}+1}_{\natf{j}}$ is the strong constant extension of $\algf{S}^{\natf{i}}_{\natf{j}}$ by $\constf{e}_{\natf{i}+1}$ with interpretation $\constf{e}_{\natf{i}+1}^{\algf{S}^{\natf{i}+1}_{\natf{j}}}=\bau^{\algf{S}^{\natf{i}}_{\natf{j}}}$.
\end{itemize}
Simple check shows that for $\overline{\algf{S}^{\natf{i}+1}}$ the induction hypothesis holds.

Further we assume that $\fextfquotfr^{\tembf{l}\typeplus\csetf{E}_{\natf{i}},\constf{q}}(\algf{S}^{\natf{i}}_{\natf{u}})\models \diamo[\ione]{\constf{c}}\ne\baz$. We give $\overline{\algf{S}^{\natf{i}+1}}$ as following:
\begin{itemize}
\item  $\natf{k}_{\natf{i}+1}=\natf{k}_{\natf{i}}+1$,
\item  for all $\natf{j}$ from $1$ to $\natf{u}-1$, the algebra $\algf{S}^{\natf{i}+1}_{\natf{j}}$ is the strong constant extension of $\algf{S}^{\natf{i}}_{\natf{j}}$ by $\constf{e}_{\natf{i}+1}$ with interpretation $\constf{e}_{\natf{i}+1}^{\algf{S}^{\natf{i}+1}_{\natf{j}}}=\baz^{\algf{S}^{\natf{i}}_{\natf{j}}}$,
\item $\algf{S}^{\natf{i}+1}_{\natf{u}}$ is the strong constant extension of $\purfr^{\constf{q}}(\algf{S}^{\natf{i}}_{\natf{j}-1})$ by $\{\constf{q},\constf{e}_{\natf{i}+1}\}$ with interpretations $\constf{q}^{\algf{S}^{\natf{i}+1}_{\natf{u}}}=\constf{c}^{\algf{S}^{\natf{i}}_{\natf{u}}}$ and $\constf{e}_{\natf{i}+1}^{\algf{S}^{\natf{i}+1}_{\natf{u}}}=\baz^{\algf{S}^{\natf{i}}_{\natf{u}}}$,
\item for all $\natf{j}$ from $\natf{u}+1$ to $\natf{k}_{\natf{i}}+1$ the algebra $\algf{S}^{\natf{i}+1}_{\natf{j}}$ is the strong constant extension of $\algf{S}^{\natf{i}}_{\natf{j}-1}$ by $\constf{e}_{\natf{i}+1}$ with interpretation $\constf{e}_{\natf{i}+1}^{\algf{S}^{\natf{i}+1}_{\natf{j}}}=\bau^{\algf{S}^{\natf{i}}_{\natf{j}-1}}$.
\end{itemize}

Clearly, we have $\sqqlp_{\overline{\algf{S}^{\natf{i}+1}}}\models \constf{e}_{\natf{i}+1}=\termf{t}$. By induction on $\natf{j}$ we check that for all $\natf{j}$ from $1$ to $\natf{i}$ we have $\sqqlp_{\overline{\algf{S}^{\natf{i}+1}}}\models \constf{e}_{\natf{j}}=\termf{t}_{\natf{j}}$; from this and the previous sentence it will follows that $\sqext_{\overline{\algf{S}^{\natf{i}+1}}}=\purfr^{\csetf{E}\setminus\csetf{E}_{\natf{i}+1}}(\algf{B}')$.  The case of all $\termf{t}_{\natf{j}}$ but $\termf{t}_{\natf{j}}$ that starts with $\diamos[\ione]$ trivially holds. Further we will assume that $\termf{t}_{\natf{j}}$  starts with $\diamos[\ione]$. From the inductive hypothesis of the second induction it follows that  $\sqqlp_{\overline{\algf{S}^{\natf{i}+1}}}\models \diamo[\ione]{\constf{b}}=\termf{t}_{\natf{j}}$ for some constant symbol $\constf{b}$ that is not $\constf{e}_{\natf{i}+1}$. We find the minimal $\natf{o}$ such that $\fextfquotfr^{\tembf{l}\typeplus\csetf{E}_{\natf{i}},\constf{q}}(\algf{S}^{\natf{i}}_{\natf{o}})\not\models \constf{b}=\baz$. We consider two cases: 1. $\natf{o}<\natf{u}$ and 2. $\natf{o}\ge\natf{u}$ or $\natf{o}$ is undefined. Suppose $\natf{o}<\natf{u}$. From the induction hypothesis of the first induction  we have $\algf{S}^{\natf{i}}_{\natf{l}}\models \constf{e}_{\natf{j}}=\baz$ for all $\natf{l}$ from $1$ to $\natf{o}$ and  $\algf{S}^{\natf{i}}_{\natf{l}}\models \constf{e}_{\natf{j}}=\bau$ for all $\natf{l}$ from $\natf{o}$ to $\natf{k}_{\natf{i}}$. Hence the inductive hypothesis for this $\natf{j}$ holds. Now suppose  $\natf{o}\ge\natf{u}$ or $\natf{o}$ is undefined. Here we will assume that $\natf{u}\ne\natf{k}_{\natf{i}}$; the case of $\natf{u}=\natf{k}_{\natf{i}}$ is almost the same. We have $\fextfquotfr^{\tembf{l}\typeplus \csetf{E}_{\natf{i}},\constf{q}}(\algf{S}^{\natf{i}}_{\natf{u}})\models \diamo[\ione]{\constf{b}}=\baz$. Hence $\fextfr^{\tembf{l}\typeplus \csetf{E}_{\natf{i}}\typeplus\constf{q}}(\algf{S}^{\natf{i}}_{\natf{u}})\models \diamo[\ione]{\constf{q}}\ge \diamo[\ione]{\constf{b}}$. We also know that $\fextfr^{\tembf{l}\typeplus \csetf{E}_{\natf{i}}\typeplus\constf{q}}(\algf{S}^{\natf{i}}_{\natf{u}})\not\models \diamo[\ione]{\constf{q}} \ge \diamo[\ione]{\constf{c}}$. From the linearity of $\fextfr^{\tembf{l}\typeplus \csetf{E}_{\natf{i}}\typeplus\constf{q}}(\algf{S}^{\natf{i}}_{\natf{u}})$ it follows that  $$\fextfr^{\tembf{l}\typeplus \csetf{E}_{\natf{i}}\typeplus\constf{q}}(\algf{S}^{\natf{i}}_{\natf{u}})\models \diamo[\ione]{\constf{c}} \algl_{\ione} \diamo[\ione]{\constf{q}}$$ and  $$\fextfr^{\tembf{l}\typeplus \csetf{E}_{\natf{i}}\typeplus\constf{q}}(\algf{S}^{\natf{i}}_{\natf{u}})\models \diamo[\ione]{\constf{q}} \algl_{\ione} \diamo[\ione]{\constf{b}} \foor \diamo[\ione]{\constf{q}} \algs_{\ione} \diamo[\ione]{\constf{b}}.$$  Thus $$\fextfr^{\tembf{l}\typeplus \csetf{E}_{\natf{i}}\typeplus\constf{q}}(\algf{S}^{\natf{i}}_{\natf{u}})\models \constf{c} \algl_{\ione}\constf{b}.$$ Hence $$\fextfr^{\tembf{l}\typeplus \csetf{E}_{\natf{i}}\typeplus\constf{q}}(\algf{S}^{\natf{i}}_{\natf{u}})\models \diamo[\ione]{\constf{c}} \ge\constf{b}.$$  Therefore $\fextfquotfr^{\tembf{l}\typeplus \csetf{E}_{\natf{i}+1},\constf{q}}(\algf{S}^{\natf{i}+1}_{\natf{u}})\models\constf{b}=\baz$. From the last we conclude the inductive hypothesis. This finishes the proof of our second  inductive claim. It also finishes the proof of the first inductive claim and the lemma. 
\end{proof}

\section{Elementary Theories of $\thrf{GLP}$-Algebras}
We will assume that all types $\typef{A}=(\csetf{A},\ordf{a})$ we consider are effective in the following sense:
\begin{itemize}
\item sets $\csetf{A}$ and $|\ordf{a}|$ are enumerable,
\item $<_{\ordf{a}}$ is decidable relation.
\end{itemize}

Suppose $\typef{A}$ is a type with a minimal operator symbol $\diamos[\objf{m}]$.

For a term $\termf{t}$ we denote by  $\termf{t}^0$ the term $\bacmp{\termf{t}}$ and by $\termf{t}^1$ the term $\termf{t}$. We denote by $\slangat{\typef{A}}$ the class of all formulas in the language of $\typef{A}$-algebras of the form $\termf{t}_1^{\natf{p}_1}\baand(\termf{t}_2^{\natf{p}_2}\baand\ldots(\termf{t}_{\natf{n-1}}^{\natf{p}_{\natf{n}-1}}\baand\termf{t}_{\natf{n}}^{\natf{p}_{\natf{n}}})\ldots)=\baz$, where 
\begin{enumerate}
\item for every $\natf{i}$ from $1$ to $\natf{n}$ the number $\natf{p}_{\natf{i}}\in\{0,1\}$;
\item  for every $\natf{i}$ from $1$ to $\natf{n}$ the term $\termf{t}_{\natf{i}}$ is either $\termf{w}_{\natf{i}}$ or $\boxo[\objf{x}]{\termf{w}_{\natf{i}}}$, where $\boxos[\objf{x}]\in \typef{A}$ and $\termf{w}_{\natf{i}}$ is either a constant symbol from $\typef{A}$ or a first-order variable;
\item for $0<\natf{i}<\natf{j}\le \natf{n}$ terms $\termf{t}_{\natf{i}}$ and $\termf{t}_{\natf{j}}$  are graphically nonidentical.
\end{enumerate}

We call a type $\typef{A}$ {\it finite} if there are only finitely many symbols in $\typef{A}$.

We denote by $\slangatc{\typef{A}}$ the set of all closed formulas from $\slangat{\typef{A}}$. Note that for a finite type $\typef{A}$ there are only finitely many formulas in $\slangatc{\typef{A}}$.

We denote by $\slangate{\typef{A}}$ the class of all formulas in the language of $\typef{A}$-algebras of the form $\termf{t}_1=\termf{t}_2$ such that for all  $\diamos[\objf{a}]\in \typef{A}$ every occurrence of $\boxos[\objf{a}]$ in $\termf{t}_1$ and $\termf{t}_2$ is of the form $\diamo[\objf{a}]{\termf{w}}$, where $\termf{w}$ is either a constant from $\typef{A}$ or a first-order variable. 

We denote by $\slangprn{\typef{A}}$ the class of all formulas of the form $\prflf{f}[\prvarf{x}_1,\ldots,\prvarf{x}_{\natf{n}}/\foflf{p}_1,\ldots,\foflf{p}_{\natf{n}}]$, where $\prflf{f}$ is a propositional formula in disjunctive normal form, $\{\prvarf{x}_1,\ldots,\prvarf{x}_{\natf{n}}\}$ is the set of all propositional variable that lies in $\prflf{f}$, formulas  $\foflf{p}_1,\ldots,\foflf{p}_{\natf{n}}\in\slangat{\typef{A}}$, and for all $1\le \natf{i}<\natf{j}\le \natf{n}$ formulas $\foflf{p}_{\natf{i}}$ and $\foflf{p}_{\natf{j}}$ are graphically nonidentical. We denote by $\slangprnc{\typef{A}}$ the set of all closed formulas from $\slangprn{\typef{A}}$. Obviously, for a finite type $\typef{A}$ there are only finitely many formulas in $\slangprnc{\typef{A}}$.

We call a propositional formula $\prflf{f}$ a {\it positive propositional formula} if the only connectives used in $\prflf{f}$ are $\foand$ and $\foor$.  Note that we consider $\bot$ as a positive formula.

A {\it quantifier prefix} $\prefixf{P}$ is a string of the form $\quant_1 \fovarf{x}_1\ldots \quant_{\natf{n}} \fovarf{x}_{\natf{n}}$, where every $\quant_{\natf{i}}$ is either $\forall$ or $\exists$ and $\natf{n}\ge 0$. For a quantifier prefix $\quant \fovarf{x} \prefixf{P}$ we denote by $\slangn{\typef{A}}{\quant \fovarf{x} \prefixf{P}}$ the class of all formulas of the form $\prflf{f}[\prvarf{x}_1,\ldots,\prvarf{x}_{\natf{n}}/\quant \fovarf{x}\foflf{p}_1,\ldots,\quant \fovarf{x}\foflf{p}_{\natf{n}}]$, where $\prflf{f}$ is a positive propositional formula in disjunctive normal form, $\{\prvarf{x}_1,\ldots,\prvarf{x}_{\natf{n}}\}$ is the set of all propositional variable that lies in $\prflf{f}$, and formulas $\foflf{p}_1,\ldots,\foflf{p}_{\natf{n}}\in\slangn{\typef{A}}{\prefixf{P}}$ are pairwise graphically nonidentical. We denote by $\slangnc{\typef{A}}{\prefixf{P}}$ the set of all closed formulas from $\slangn{\typef{A}}{\prefixf{P}}$. Obviously, for a finite type $\typef{A}$ and a quantifier prefix $\prefixf{P}$ there are only finitely many  formulas in $\slangn{\typef{A}}{\prefixf{P}}$.

Obviously, the following three lemmas holds:

\begin{lemma} \label{formula_normalization1}Suppose $\typef{A}$ is a type. Then for a quantifier-less $\foflf{f}$ from $\lang{\thrf{GLPA}_{\typef{A}}}$  such that every atomic subformula of $\foflf{f}$ is from $\slangate{\typef{A}}$ we can effectively find a $\foflf{f}'\in\slangprn{\typef{A}}$ such that $\foflf{f}'$ is  $\thrf{GLPA}_{\typef{A}}$-equivalent to $\foflf{f}$ and $\freevariables{\foflf{f}'}\subset\freevariables{\foflf{f}}$.
\end{lemma}

\begin{lemma} \label{slangn_closedness} Suppose $\typef{A}$ is a type. Then for every quantifier prefix $\prefixf{P}$, every positive propositional formula $\prflf{f}(\prvarf{x}_1,\ldots,\prvarf{x}_{\natf{n}})$, and formulas $\foflf{p}_1,\ldots,\foflf{p}_{\natf{n}}\in\slangn{\typef{A}}{\prefixf{P}}$ we can effectively find $\foflf{p}\in\slangn{\typef{A}}{\prefixf{P}}$ that is $\thrf{GLPA}_{\typef{A}}$-equivalent to $\prflf{f}[\prvarf{x}_1,\ldots,\prvarf{x}_{\natf{n}}/\foflf{p}_1,\ldots,\foflf{p}_{\natf{n}}]$ and $\freevariables{\foflf{p}}\subset\freevariables{\foflf{p}_1}\cup\ldots\cup \freevariables{\foflf{p}_{\natf{n}}}$.
\end{lemma}

\begin{lemma} \label{formula_normalization2} Suppose $\typef{A}$ is a type. Then for every closed formula from $\lang{\thrf{GLPA}_{\typef{A}}}$ we can effectively find quantifier prefix $\prefixf{P}$ and formula $\foflf{f}\in\slangnc{\typef{A}}{\prefixf{P}}$ such that $\foflf{f}$ will be $\thrf{GLPA}_{\typef{A}}$-equivalent to the given formula.
\end{lemma}

\begin{lemma} Suppose $\tembf{l}\colon \typef{A}\to \typef{B}$ is a normal type embedding. Then for every  constant $\constf{c}\in\typef{A}$ and formula $\foflf{f}\in\slangatc{\typef{A}}$   we can effectively find  $\cmpf{SimpQuotTr}^{\tembf{l}}(\foflf{f},\constf{c})\in\slangprnc{\typef{A}}$ such that for every $\typef{A}$-algebra $\algf{A}$ and it's extension $\algf{A}'$ by some $\constf{q}$, $\algf{A}'\models \constf{c}=\constf{q}$ we have 
$$\algf{A}\models \cmpf{SimpQuotTr}^{\tembf{l}}(\foflf{f},\constf{c})\iff \fextfquotfr^{\tembf{l},\constf{q}}(\algf{A}')\models \foflf{f}$$
\end{lemma}
\begin{proof}  Suppose $\foflf{f}$ is from $\slangatc{\typef{A}}$ is of the form $\termf{t}=\baz$. We put $$\foflf{f}'\formeq\constf{c}\algl_{\itwo}\termf{t}.$$
Clearly, $\foflf{f}'$ lies in $\slangate{\typef{A}}$.
Using Lemma \ref{formula_normalization1} obtain $\foflf{f}''\in\slangprnc{\typef{A}}$ that is $\thrf{GLPA}_{\typef{A}}$-provable equivalent of $\foflf{f}'$. The formula $\foflf{f}''$ is the value of  $\cmpf{SimpQuotTr}^{\tembf{l}}(\foflf{f},\constf{c})$.

Let us chechk that $$\algf{A}\models \constf{c}\algl_{\itwo}\termf{t}\iff \fextfquotfr^{\tembf{l},\constf{q}}(\algf{A}')\models \foflf{f}.$$Suppose $\algf{A}$ is an $\typef{A}$-algebra and $\algf{A}'$ is an extension of $\algf{A}$ by some $\constf{q}$ such that $\algf{A}'\models \constf{c}=\constf{q}$. Clearly, we have 
$$ \begin{aligned}\algf{A}\models  \constf{c}\algl_{\itwo}\termf{t} &\iff \fextfr^{\tembf{l}\typeplus \constf{q}}(\algf{A}')\models \constf{c}\algl_{\itwo}\termf{t}  \\ & \iff \fextfr^{\tembf{l}\typeplus \constf{q}}(\algf{A}')\models \constf{c}\algl_{\ione}\termf{t} \\ &\iff \fquotalg^{\typef{B},\constf{q}}(\fextfr^{\tembf{l}\typeplus \constf{q}}(\algf{A}'))\models  \termf{t}=\baz.\end{aligned}$$\end{proof}

\begin{lemma} \label{diag_frext_translation} Suppose $\tembf{l}\colon \typef{A}\to \typef{B}$ is a normal type embedding. Then for every $\foflf{f}\in\slangprnc{\typef{B}}$ we can effectively find a closed $\cmpf{DiagFrExtTr}^{\tembf{l}}(\foflf{f})\in\slangprnc{\typef{A}}$ such that for every linear $\typef{A}$-algebra $\algf{A}$
$$\fextfr^{\tembf{l}}(\algf{A})\models \foflf{f}\iff \algf{A}\models \cmpf{DiagFrExtTr}^{\tembf{l}}(\foflf{f}).$$
\end{lemma}
\begin{proof} Using Lemma \ref{formula_normalization1} we conclude that the lemma follows from the modification of lemma with weaker restriction on $\cmpf{DiagFrExtTr}^{\tembf{l}}(\foflf{f})$: $\cmpf{DiagFrExtTr}^{\tembf{l}}(\foflf{f})$ is a quantifier-less formula with all atoms from $\slangatc{\typef{A}}$. Further we prove the modified lemma. Clearly, the general case of the modified lemma follows from the case of $\foflf{f}\in\slangatc{\typef{A}}$.

We consider  $\foflf{f}\in\slangatc{\typef{A}}$. Suppose $\foflf{f}$ is $\termf{t}_1^{\natf{p}_1}\baand(\termf{t}_2^{\natf{p}_2}\baand\ldots(\termf{t}_{\natf{n-1}}^{\natf{p}_{\natf{n}-1}}\baand\termf{t}_{\natf{n}}^{\natf{p}_{\natf{n}}})\ldots)=\baz$. We denote by $\csetf{D}$ the set of all constant symbols that are in some $\termf{t}_{\natf{i}}$.

We consider all sequences $\overline{\csetf{K}}=(\csetf{K}_{1},\ldots,\csetf{K}_{\natf{p}})$ such that $\csetf{K}_{\natf{i}}\ne\emptyset$ and $\csetf{D}\cup\{\bau\}=\bigsqcup\limits_{1\le \natf{i}\le\natf{p}}\csetf{K}_{\natf{i}}$. We have sequences $\overline{\csetf{K}^{1}},\ldots,\overline{\csetf{K}^{\natf{s}}}$. For $1\le\natf{i}\le\natf{s}$
$$\overline{\csetf{K}^{\natf{i}}}=(\csetf{K}_{1}^{\natf{i}},\ldots,\csetf{K}_{\natf{p}_{\natf{i}}}^{\natf{i}})$$
We construct closed formulas $\foflf{p}_1,\ldots,\foflf{p}_{\natf{s}}$ with all atoms from $\slangate{\typef{A}}$ such that for a $\typef{A}$-algebra $\algf{A}$
$$\begin{aligned}\algf{A}\models & \foflf{p}_{\natf{i}}\iff \forall \constf{c}_1,\constf{c}_2\in\csetf{D}\cup\{\bau\} (\\ &((\algf{A}\models \constf{c}_1\algs_{\itwo} \constf{c}_2) \iff \exists 1\le\natf{j}\le \natf{p}_{\natf{i}}(\constf{c}_1,\constf{c}_2\in\csetf{K}_{\natf{j}}^{\natf{i}})) \foand \\ &((\algf{A}\models \constf{c}_1\algl_{\itwo}\constf{c}_2)\iff \exists 1\le \natf{j}_1<\natf{j}_2\le \natf{p}_{\natf{i}}(\constf{c}_1\in\csetf{K}^{\natf{i}}_{\natf{j}_1} \foand \constf{c}_2\in\csetf{K}^{\natf{i}}_{\natf{j}_2}))).\end{aligned}$$
Clearly, for a given linear $\typef{A}$-algebra $\algf{A}$ there exists exactly one $\natf{i}$ from $1$ to $\natf{s}$ such that $\algf{A}\models \foflf{p}_{\natf{i}}$.

 For $0<\natf{k}\le\natf{s}$ and $0<\natf{i}\le \natf{p}_{\natf{k}}$ we construct the formula $\foflf{t}_{\natf{k},\natf{i}}$ by replacing every occurrence of the form $\diamo[\ione]{\constf{e}}$ in $\foflf{f}$. We replace an occurrence of the considered form  with $\bau$ if $\constf{e}\in\csetf{K}_{\natf{j}}^{\natf{k}}$ for some $\natf{j}< \natf{i}$. And we replace an occcurence of the considered form  with $\baz$ if $\constf{e}\in\csetf{K}_{\natf{j}}^{\natf{k}}$ for some $\natf{j}\ge \natf{i}$.

 For $0<\natf{k}\le\natf{s}$ and $0<\natf{i}\le \natf{p}_{\natf{k}}$ we choose some fixed $\constf{e}_{\natf{k},\natf{i}}\in\csetf{K}_{\natf{k},\natf{i}}$. We put $$\cmpf{DiagFrExtTr}^{\tembf{l}}(\foflf{f})\formeq      \bigwedge\limits_{0<\natf{k}<\natf{p}}(\foflf{p}_{\natf{k}}\to (\bigwedge\limits_{0<\natf{i}<\natf{p}_{\natf{k}}}\cmpf{SimpQuotTr}^{\tembf{l}}(\foflf{t}_{\natf{k},\natf{i}},\constf{e}_{\natf{k},\natf{i}}))\foand \foflf{t}_{\natf{k},\natf{p}_{\natf{k}}}).$$ 

Suppose $\algf{A}$ is a linear $\typef{A}$-algebra. Suppose $\foflf{p}_{\natf{k}}$ holds on $\algf{A}$. We consider some extensions $\algf{S}_1,\ldots,\algf{S}_{\natf{p}_{\natf{k}}-1}$ of $\algf{A}$ by a fresh constant $\constf{q}$ such that $\algf{S}_{\natf{i}}\models \constf{q}=\constf{c}$ for some $\constf{c}$ from $\csetf{K}^{\natf{i}}$. We put $\algf{S}_{\natf{p}_{\natf{k}}}=\algf{A}$. We have just formed  the $(\tembf{l},\algf{A},\constf{q},\emptyset)$ extension sequence $\overline{S}=(\algf{S}_1,\ldots,\algf{S}_{\natf{p}_{\natf{k}}-1},\algf{S}_{\natf{p}_{\natf{k}}})$. Then the following propositions are equivalent:
\begin{enumerate}
\item $\algf{A}\models \cmpf{DiagFrExtTr}^{\tembf{l}}(\foflf{f})$;
\item $\algf{A}\models \cmpf{SimpQuotTr}^{\tembf{l}}(\foflf{p}_{\natf{k},\natf{i}})$ for every $0<\natf{i}<\natf{p}_{\natf{k}}$ and $\algf{A}\models \foflf{p}_{\natf{k},\natf{p}_{\natf{k}}}$;
\item $\fquotalg^{\typef{B},\constf{q}}(\fextfr^{\tembf{l}\typeplus\constf{q}}(\algf{S}_{\natf{i}}))\models \foflf{p}_{\natf{k},\natf{i}}$  for every $0<\natf{i}<\natf{p}_{\natf{k}}$ and $\fextfr^{\tembf{l}}(\algf{S}_{\natf{p}_{\natf{k}}})\models \foflf{p}_{\natf{k},\natf{p}_{\natf{k}}}$;
\item $\sqqlp_{\overline{\algf{S}}}\models \foflf{f}$;
\item $\fextfr^{\tembf{l}}(\algf{A})\models \foflf{f}$.
\end{enumerate}
Therefore 
$$\algf{A}\models \cmpf{DiagFrExtTr}^{\tembf{l}}(\foflf{f})\iff\fextfr^{\tembf{l}}(\algf{A})\models \foflf{f}.$$
\end{proof}

\begin{lemma}\label{lin_prod_tr}Suppose $\typef{A}$ is a type with a minimal element, $\constf{q}\not\in\typef{A}$, and $\typef{B}$ is the extension of $\typef{A}$ by $\constf{q}$. Then every quantifier prefix $\prefixf{P}$ and $\foflf{f}\in\slangnc{\typef{A}}{\prefixf{P}}$ we can effectively find a $\cmpf{QuotTr}^{\tembf{l},\constf{q}}(\foflf{f})\in\slangnc{\typef{B}}{\prefixf{P}}$ such that for every linear $\typef{B}$-algebra $\algf{A}$ 
$$\fquotalg^{\typef{A},\constf{q}}(\algf{A})\models \foflf{f}\iff \algf{A}\models \cmpf{QuotTr}^{\tembf{l},\constf{q}}(\foflf{f}).$$
\end{lemma}
\begin{proof} In order to obtain $\cmpf{QuotTr}^{\tembf{l},\constf{q}}(\foflf{f})$ we first replace every $\termf{t}=\baz$ in $\foflf{f}$ with $\boxo[\ione]{\constf{q}}\baor\termf{t}=\boxo[\ione]{\constf{q}}$ and then find a $\thrf{GLPA}_{\typef{A}\typeplus\constf{q}}$-equivalent formula from $\slangnc{\typef{A}}{\prefixf{P}}$. 
\end{proof}

We fix a countable family of propositional variables $\prvarf{p}_1,\ldots,\prvarf{p}_{\natf{n}},\ldots$.

For a finite type $\typef{A}$ and quantifier prefix $\prefixf{P}$ we denote the number of formulas in $\slangnc{\typef{A}}{\prefixf{P}}$ by $\natf{u}_{\typef{A},\prefixf{P}}$. We choose an enumerations of formulas of $\slangnc{\typef{A}}{\prefixf{P}}$:
$$\slangnc{\typef{A}}{\prefixf{P}}=\{\foflf{e}^{\typef{A},\prefixf{P}}_{1},\ldots,\foflf{e}^{\typef{A},\prefixf{P}}_{\natf{u}_{\typef{A},\prefixf{P}}}\}.$$

\begin{lemma}Suppose $\typef{A}$ is a normal type. Then for a quantifier prefix $\prefixf{P}$ and a formula  $\foflf{f}\in \slangnc{\typef{A}}{\prefixf{P}}$  we can effectively find a positive propositional formula $\cmpf{LinProdTr}^{\typef{A}}(\foflf{f})$ such that
\begin{itemize}
\item $\cmpf{LinProdTr}^{\typef{A}}(\foflf{f})$  is positive;
\item any variable in $\cmpf{LinProdTr}^{\typef{A}}(\foflf{f})$ is $\prvarf{p}_{\natf{i}}$ for some $\natf{i}\le 2\natf{u}_{\typef{A},\prefixf{P}}$;
\item  for every pair of linear $\typef{A}$-algebras $(\algf{A},\algf{B})$ we have $\algf{A}\linprodfr \algf{B}\models \foflf{f}$ iff the result of the application to $\cmpf{LinProdTr}^{\typef{A}}(\foflf{f})$ of the following substitution  is a true judgment:
\begin{itemize}
\item $\prvarf{p}_1\;\leftarrow\;\algf{A}\models \foflf{e}_1^{\typef{A},\prefixf{P}}$,\\
 $\ldots$
\item $\prvarf{p}_{\natf{u}_{\typef{A},\prefixf{P}}}\;\leftarrow\;\algf{A}\models \foflf{e}_{\natf{u}_{\typef{A},\prefixf{P}}}^{\typef{A},\prefixf{P}}$,
\item $\prvarf{p}_{\natf{u}_{\typef{A},\prefixf{P}}+1}\;\leftarrow\;\algf{B}\models \foflf{e}_{1}^{\typef{A},\prefixf{P}}$,\\
$\ldots$
\item $\prvarf{p}_{2\natf{u}_{\typef{A},\prefixf{P}}}\;\leftarrow\;\algf{B}\models \foflf{e}_{\natf{u}_{\typef{A},\prefixf{P}}}^{\typef{A},\prefixf{P}}$.
\end{itemize}
\end{itemize}

\end{lemma}
\begin{proof} Suppose $\diamos[\objf{m}]$ is the minimal operator symbol of $\typef{A}$. 

We give the construction of $\cmpf{LinProdTr}^{\typef{A}}(\foflf{f})$ by induction on the length of $\prefixf{P}$; the effectiveness is a trivial consequence of our proof.  

First we prove the basis of the induction, i.e. the case of empty $\prefixf{P}$. We will construct $\cmpf{LinProdTr}^{\typef{A}}(\foflf{f})$ such that may be it will not be positive but all other conditions will holds for it. 

If we construct such a $\cmpf{LinProdTr}^{\typef{A}}(\foflf{f})$ for all formulas $\foflf{f}\in\slangatc{\typef{A}}$ then obviously, we can construct it for all formulas $\foflf{f}\in\slangprnc{\typef{A}}$. So further we assume that $\foflf{f}$ is a formula from $\slangatc{\typef{A}}$. Suppose all subterms of the form $\diamo[\objf{m}]{\termf{t}}$ for $\foflf{f}$ are terms $\diamo[\objf{m}]{\termf{w}_1},\ldots,\diamo[\objf{m}]{\termf{w}_{\natf{m}}}$. 
For a $\{0,1\}$-sequence $\seqf{p}=(\seqelf{p}_1,\ldots,\seqelf{p}_{\natf{m}})\in \mathrm{Sq}_2(\natf{m})$ of the length $\natf{m}$ we denote by $\foflf{t}_{\seqf{p}}$ the formula
$$(\bigwedge\limits_{0\le\natf{i}\le \natf{m}, \seqelf{p}_{\natf{i}}=0}\termf{w}_{\natf{i}}= \baz) \foand (\bigwedge\limits_{0\le\natf{i}\le \natf{m}, \seqelf{p}_{\natf{i}}=1}\termf{w}_{\natf{i}}\ne \baz).$$ For a binary sequence $\seqf{p}=(\seqelf{p}_1,\ldots,\seqelf{p}_{\natf{m}})\in \mathrm{Sq}_2(\natf{m})$ of the length $\natf{m}$ we denote by $\foflf{p}_{\seqf{p}}$ the result of replacement of all occurrences of $\diamo[\objf{m}]{\termf{w}_{\natf{i}}}$ with  $\baz$ for all $\natf{i}$ such that $\seqelf{p}=0$. For a binary sequence $\seqf{p}=(\seqelf{p}_1,\ldots,\seqelf{p}_{\natf{m}})\in \mathrm{Sq}_2(\natf{m})$ of the length $\natf{m}$ we denote by $\foflf{h}_{\seqf{p}}$ the result of replacement of all occurrences of $\diamo[\objf{m}]{\termf{w}_{\natf{i}}}$ with  $\bau$ for all $\natf{i}$ such that $\seqelf{p}_{\natf{i}}=1$. 

Obviously, for linear $\typef{A}$-algebras $\algf{A}$,$\algf{B}$, an element $(\elf{x},\elf{y})\in\algf{A}\linprodfr\algf{B}$, and $\natf{i}$ from $1$ to $\natf{m}$ we have 
$$\begin{aligned}\algf{A}\linprodfr\algf{B} & \models (\elf{x},\elf{y})=\diamo[\objf{m}]{\termf{w}_{\natf{i}}}\iff  (\algf{A}\models \termf{w}_{\natf{i}}=\baz  \mbox{, and }\algf{A}\models \elf{x}=\baz\mbox{, and } \\ & \algf{B}\models \elf{y}=\diamo[\objf{m}]{\termf{w}_{\natf{i}}} ) \mbox{ or} (\algf{A}\models \termf{w}_{\natf{i}}\ne\baz  \mbox{, and }\algf{A}\models \elf{x}=\diamo[\objf{m}]{\termf{w}_{\natf{i}}}\mbox{, and } \algf{B}\models \elf{y}=\baz ).\end{aligned}$$
Hence for linear $\typef{A}$-algebras $\algf{A}$,$\algf{B}$
$$\algf{A}\linprodfr\algf{B}\models \foflf{f} \iff \bigvee\limits_{\seqf{p}\in \mathrm{Sq}_2(\natf{m})}(\algf{A}\models\foflf{t}_{\seqf{p}})\foand (\algf{A}\models \foflf{p}_{\seqf{p}})\foand(\algf{B}\models \foflf{h}_{\seqf{p}}).$$
Using this equivalence we easily construct the required formula $\cmpf{LinProdTr}^{\typef{A}}(\foflf{f})$.

We claim that  we can transform  $\cmpf{LinProdTr}^{\typef{A}}(\foflf{f})$ that we constructed above  to a formula that satisfies all conditions of the lemma. For every formula $\foflf{p}\in\slangprnc{\typef{A}}$ we can find a formula from $\slangprnc{\typef{A}}$ that is $\thrf{GLPA}_{\typef{A}}$-equivalent  to $\fonot\foflf{p}$. Every propositional formula is equivalent to some propositional formula in Disjunctive Normal Form; all occurrences of $\prnot$ in formulas in {\bf DNF} are of the form $\prnot \prvarf{x}$, where $\prvarf{x}$ is a propositional variable. Obviously, our claim follows from the two previous sentences.  

Now we prove the step of induction. Suppose $\prefixf{P}=\exists\fovarf{x}\prefixf{P}'$ (the proof for $\prefixf{P}=\forall\fovarf{x}\prefixf{P}'$ can be carried out in a similar way). Suppose $\foflf{f}(\fovarf{x})$ is a formula from $\slangnc{\typef{A}}{\prefixf{P}'}$ and there are no free variables in $\foflf{f}(\fovarf{x})$ other than $\fovarf{x}$. We are going to construct a formula $\cmpf{LinProdTr}^{\typef{A}}(\exists \fovarf{x} \foflf{f})$ that satisfies all conditions of the lemma. 

We choose  a fresh constant symbol $\constf{c}\not \in \typef{A}$. We denote by $\prflf{p}$ the propositional formula $\cmpf{LinProdTr}^{\typef{A}\typeplus\constf{c}}(\foflf{f}(\constf{c}))$. We can transform $\prflf{p}$ to an equivalent  positive formula $\prflf{p}'$ in Disjunctive Normal Form:
$$\bigvee\limits_{0<\natf{i}\le\natf{k}}(\bigwedge\limits_{\natf{j}\in\setf{A}_{\natf{i}}} \prvarf{p}_{\natf{j}})\foand(\bigwedge\limits_{\natf{j}\in\setf{B}_{\natf{i}}} \prvarf{p}_{\natf{u}_{\typef{A}\typeplus\constf{c},\prefixf{P}'}+\natf{j}}),$$
where $\setf{A}_{\natf{i}},\setf{B}_{\natf{i}}\subset\{1,\ldots,\natf{u}_{\typef{A}\typeplus \constf{c},\prefixf{P}'}\}$. For every $0<\natf{i}\le\natf{k}$ using Lemma \ref{slangn_closedness} we find $\natf{s}_{\natf{i}},\natf{t}_{\natf{i}}\in\{1,\ldots,\natf{u}_{\typef{A},\prefixf{P}}\}$ such that $\foflf{e}^{\typef{A},\prefixf{P}}_{\natf{s}_{\natf{i}}}$ and  $\foflf{e}^{\typef{A},\prefixf{P}}_{\natf{t}_{\natf{i}}}$ are $\thrf{GLPA}_{\typef{A}}$-equivalent to the formulas $\exists \fovarf{x}(\bigwedge\limits_{\natf{j}\in\setf{A}_{\natf{i}}}\foflf{t}_{\natf{j}})$ and $\exists \fovarf{x}(\bigwedge\limits_{\natf{j}\in\setf{B}_{\natf{i}}}\foflf{t}_{\natf{j}})$, respectively, where for every $\natf{j}$ from $1$ to $\natf{u}_{\typef{A}\typeplus \constf{c},\prefixf{P}'}$ the formula $\foflf{t}_{\natf{j}}$ is $\foflf{e}^{\typef{A}\typeplus \constf{c},\prefixf{P}'}_{\natf{j}}$ with every occurrence of $\constf{c}$ replaced with $\fovarf{x}$.

 Suppose $\algf{A}$ and $\algf{B}$ are linear $\typef{A}$-algebras. Clearly, the following propositions are equivalent:
\begin{enumerate}
\item $\algf{A}\linprodfr \algf{B}\models\foflf{f}$;
\item there exist constant extensions $\algf{A}'$ and $\algf{B}'$ by constant $\constf{c}$ of algebras $\algf{A}$ and $\algf{B}$, respectively such that the result of the following substitution applied to $\prflf{p}$ is a true judgment:
\begin{itemize}
\item $\prvarf{p}_1\;\leftarrow\;\algf{A}'\models \foflf{e}_1^{\typef{A}\typeplus \constf{c},\prefixf{P}'}$,\\
 $\ldots$
\item $\prvarf{p}_{\natf{u}_{\typef{A}\typeplus \constf{c},\prefixf{P}'}}\;\leftarrow\;\algf{A}'\models \foflf{e}_{\natf{u}_{\typef{A}\typeplus \constf{c},\prefixf{P}'}}^{\typef{A}\typeplus \constf{c},\prefixf{P}'}$,
\item $\prvarf{p}_{\natf{u}_{\typef{A}\typeplus \constf{c},\prefixf{P}'}+1}\;\leftarrow\;\algf{B}'\models \foflf{e}_{1}^{\typef{A}\typeplus\constf{c},\prefixf{P}'}$,\\
$\ldots$
\item $\prvarf{p}_{2\natf{u}_{\typef{A}\typeplus \constf{c},\prefixf{P}'}}\;\leftarrow\;\algf{B}'\models \foflf{e}_{\natf{u}_{\typef{A}\typeplus\constf{c},\prefixf{P}'}}^{\typef{A}\typeplus\constf{c},\prefixf{P}'}$;
\end{itemize}
\item for some $\natf{i}$ from $1$ to $\natf{k}$ there exist $\constf{c}$ constant extensions $\algf{A}'$ and $\algf{B}'$ of algebras $\algf{A}$ and $\algf{B}$, respectively such that $\algf{A}'\models \foflf{e}^{\typef{A}\typeplus\constf{c},\prefixf{P}'}_{\natf{j}}$ for all $\natf{j}\in\setf{A}_{\natf{i}}$ and $\algf{B}'\models \foflf{e}^{\typef{A}\typeplus\constf{c},\prefixf{P}'}_{\natf{j}}$ for all $\natf{j}\in\setf{B}_{\natf{i}}$;
\item  the result of the substitution from the lemma formulation applied to $\bigvee\limits_{0<\natf{i}\le\natf{k}}\prvarf{p}_{\natf{s}_{\natf{i}}} \foand\prvarf{p}_{\natf{t}_{\natf{i}}+\natf{u}_{\typef{A},\prefixf{P}}}$ is a true judgment.
\end{enumerate}
We put $$\cmpf{LinProdTr}^{\typef{A}}(\foflf{f})\formeq  \bigvee\limits_{0<\natf{i}\le\natf{k}}\prvarf{p}_{\natf{s}_{\natf{i}}}\foand\prvarf{p}_{\natf{t}_{\natf{i}}+\natf{u}_{\typef{A},\prefixf{P}}}.$$

If $\prefixf{P}$ starts with $\forall$ then we can carry the proof in a dual way to $\exists$ case. We replace conjunctions with disjunctions, $\exists$ quantifiers with $\forall$ quantifiers, existential propositions with universal, etc.\end{proof}

For a quantifier prefix $\prefixf{P}=\quant_1\fovarf{x}_1\ldots\quant_{\natf{n}}\fovarf{x}_{\natf{n}}$ we denote by $\negprefix{\prefixf{P}}$ the prefix $\quant_1'\fovarf{x}_1\ldots\quant_{\natf{n}}'\fovarf{x}_{\natf{n}}$ such that $\quant_{\natf{i}}'\ne\quant_{\natf{i}}$ for all $\natf{i}$ from $1$ to $\natf{n}$.

Clearly, the following lemma holds:
\begin{lemma}\label{neg_th_tr} For a type $\typef{A}$, a quantifier prefix $\prefixf{P}$ and a set $\thrf{T}\subset\slangnc{\typef{A}}{\prefixf{P}}$ we can effectively find a subset $\cmpf{NegThTr}^{\typef{A}}(\prefixf{P},\thrf{T})\subset\slangnc{\typef{A}}{\prefixf{P}}$ such that for every $\typef{A}$-algebra $\algf{A}$ $$\thrf{T}=\elthr[\slangn{\typef{A}}{\prefixf{P}}]{\algf{A}}\Implication \cmpf{NegThTr}^{\typef{A}}(\prefixf{P},\thrf{T})=\elthr[\slangn{\typef{A}}{\negprefix{\prefixf{P}}}]{\algf{A}}.$$
Moreover, for a  a type $\typef{A}$, a quantifier prefix $\prefixf{P}$ and sets $\thrf{T}_1,\thrf{T}_2\subset\slangnc{\typef{A}}{\prefixf{P}}$
$$\thrf{T}_1\subset\thrf{T}_2 \Implication \cmpf{NegThTr}^{\typef{A}}(\prefixf{P},\thrf{T}_2)\subset \cmpf{NegThTr}^{\typef{A}}(\prefixf{P},\thrf{T}_1).$$\end{lemma}

For a class of formulas $\thrf{L}$ from the first-order language of $\typef{A}$-algebras and $\typef{A}$-algebra $\algf{A}$ we denote by $\elthr[\thrf{L}]{\algf{A}}$ the set of all closed formulas from $\thrf{L}$ that holds in $\algf{A}$. We denote by $\elthr{\algf{A}}$ the set of all well-built first-order closed formulas that holds in $\algf{A}$.  

For a quantifier prefix $\prefixf{P}=\quant_1 \fovarf{x}_1 \ldots \quant_{\natf{n}} \fovarf{x}_{\natf{n}}$ we denote by $\doubleprefix{\prefixf{P}}$ the quantifier prefix $$\quant_1 \fovarf{x}_1\quant_1 \fovarf{x}_1'\quant_2 \fovarf{x}_2\quant_2 \fovarf{x}_2'\ldots \quant_{\natf{n}} \fovarf{x}_{\natf{n}}\quant_{\natf{n}} \fovarf{x}_{\natf{n}}',$$ where  $\fovarf{x}_1',\ldots,\fovarf{x}_{\natf{n}}'$ are pairwise different fresh variables that are chosen in a some fixed way. 

\begin{lemma}\label{frext_translation} For a quantifier prefix $\prefixf{P}$, normal type embedding $\tembf{l}\colon \typef{A}\to \typef{B}$ with finite $\typef{A}$, and a subset $\thrf{T}$ of $\slangnc{\typef{A}}{\doubleprefix{\prefixf{P}}}$ we can effectively find a subset $\cmpf{FrExtThTr}^{\tembf{l}}(\prefixf{P},\thrf{T})$ of $\slangnc{\typef{B}}{\prefixf{P}}$ such that for every $\typef{A}$-algebra $\algf{A}$  $$\thrf{T}=\elthr[\slangn{\typef{A}}{\doubleprefix{\prefixf{P}}}]{\algf{A}}\Implication \cmpf{FrExtThTr}^{\tembf{l}}(\prefixf{P},\thrf{T})=\elthr[\slangn{\typef{A}}{\prefixf{P}}]{\fextfr^{\tembf{l}}(\algf{A})}.$$
\end{lemma}
\begin{proof} We prove  the lemma simultaneously with the following proposition by induction on the length of $\prefixf{P}$: for a quantifier prefix $\prefixf{P}$, normal type embedding $\tembf{l}\colon \typef{A}\to \typef{B}$ with finite $\typef{A}$, and  subsets $\thrf{T}_1,\thrf{T}_2\subset\slangnc{\typef{A}}{\doubleprefix{\prefixf{P}}}$ we have $$\thrf{T}_1\subset\thrf{T}_2\Implication\cmpf{FrExtThTr}^{\tembf{l}}(\prefixf{P},\thrf{T}_1)\subset \cmpf{FrExtThTr}^{\tembf{l}}(\prefixf{P},\thrf{T}_2).$$ For empty prefix $\prefixf{P}$ the lemma straightforward follows from Lemma \ref{diag_frext_translation}. 

Suppose $\prefixf{P}=\quant \fovarf{x} \prefixf{P}'$. We can only consider the case of $\quant=\exists$; if $\quant=\forall$ then we put $$\cmpf{FrExtThTr}^{\tembf{l}}(\prefixf{P},\thrf{T})=\cmpf{NegThTr}^{\typef{B}}(\negprefix{\prefixf{P}},\cmpf{FrExtThTr}^{\tembf{l}}(\negprefix{\prefixf{P}},\cmpf{NegThTr}^{\typef{A}}(\doubleprefix{\prefixf{P}},\thrf{T}))).$$ 

We choose fresh constant symbols $\constf{c}$ and $\constf{q}$. For every subset $\thrf{U}$ of the set of formulas from $\slangnc{\typef{A}\typeplus \constf{c}\typeplus\constf{q}}{\doubleprefix{(\prefixf{P}')}}$ we can construct formula $\foflf{p}_{\thrf{U}}\in \slangnc{\typef{A}\typeplus \constf{c} \typeplus\constf{q}}{\doubleprefix{(\prefixf{P}')}}$ which is $\thrf{GLPA}_{\typef{A}\typeplus\constf{c}\typeplus\constf{q}}$-equivalent to conjunction of all formulas from $\thrf{U}$.  Further, for every $\foflf{p}_{\thrf{U}}$ we construct $\foflf{p}_{\thrf{U}}'$ by replacing  every occurrence of $\constf{c}$ and $\constf{q}$ with  fresh variables $\fovarf{y}_1$ and $\fovarf{y}_2$, respectively. Then for $\foflf{p}_{\thrf{U}}'$ we denote by $\foflf{p}_{\thrf{U}}''$ the formula $\exists \fovarf{y}_1 \exists \fovarf{y}_2 \foflf{p}_{\thrf{U}}''$. And finally, by renaming some bounded variables  for every $\foflf{p}_{\thrf{U}}''$ we construct an equivalent $\foflf{p}_{\thrf{U}}'''\in \slangnc{\typef{A}}{\doubleprefix{\prefixf{P}}}$. We denote by $\setf{A}$ the set of all $\thrf{U}$ such that  $\foflf{p}_{\thrf{U}}'''\in\thrf{T}$.  Obviously, for a linear $\typef{A}$-algebra $\algf{A}$ such that $\elthr[\slangn{\typef{A}}{\doubleprefix{\prefixf{P}}}]{\algf{A}}=\thrf{T}$ we have $$\setf{A}=\{\thrf{U}\subset\elthr[\slangn{\typef{A}\typeplus\constf{c}\typeplus\constf{q}}{\doubleprefix{(\prefixf{P}')}}]{\algf{A}'}\mid \mbox{ $\algf{A}'$ is a constant extensions by  $\{\constf{c},\constf{q}\}$ of $\algf{A}$}\}.$$  
We put $$\setf{B}=\{\thrf{U}\mid \exists \thrf{U}'\in\setf{A}' (\thrf{U}=\{\foflf{t}\in\thrf{U}'\mid \mbox{ there are no $\constf{q}$ in $\foflf{t}$}\})\}.$$
 Obviously, for a linear $\typef{A}$-algebra $\algf{A}$ such that $\elthr[\slangn{\typef{A}}{\doubleprefix{\prefixf{P}}}]{\algf{A}}=\thrf{T}$ we have $$\setf{B}=\{\thrf{U}\subset\elthr[\slangn{\typef{A}\typeplus\constf{c}}{\doubleprefix{(\prefixf{P}')}}]{\algf{A}'}\mid \mbox{ $\algf{A}'$ is a constant extensions by  $\constf{c}$ of $\algf{A}$}\}.$$ 

Using the inductive hypothesis we construct
$$\setf{B}'=\{\cmpf{FrExtThTr}^{\tembf{l}\typeplus\constf{c}}(\prefixf{P}',\setf{U})\mid \thrf{U}\in \setf{B}\}$$
and
$$\setf{A}'=\{\cmpf{FrExtThTr}^{\tembf{l}\typeplus\constf{c}}(\prefixf{P}',\setf{U})\mid \thrf{U}\in \setf{A}\}.$$
We consider downward closures $\setf{A}''$ and $\setf{B}''$ of $\setf{A}'$ and $\setf{B}'$ respectively. 

Clearly, for a linear $\typef{A}$-algebra $\algf{A}$ such that $\elthr[\slangn{\typef{A}}{\doubleprefix{\prefixf{P}}}]{\algf{A}}=\thrf{T}$ we have $$\setf{B}''=\{\elthr[\slangn{\typef{B}\typeplus\constf{c}}{\prefixf{P}'}]{\fextfr^{\tembf{l}\typeplus\constf{c}}(\algf{A}')}\mid \mbox{ $\algf{A}'$ is a constant extensions by  $\constf{c}$ of $\algf{A}$}\}$$
and 
$$\setf{A}''=\{\elthr[\slangn{\typef{B}\typeplus\constf{c}}{\prefixf{P}'}]{\fextfr^{\tembf{l}\typeplus\constf{c}\typeplus\constf{q}}(\algf{A}')}\mid \mbox{ $\algf{A}'$ is a constant extensions by  $\{\constf{q},\constf{c}\}$ of $\algf{A}$}\}.$$

We put
$$\setf{A}'''=\{\thrf{U}\mid \{\cmpf{QuotTr}^{\typef{B}\typeplus\constf{c},\constf{q}}(\foflf{t})\mid \foflf{t}\in\thrf{U}\}\in\setf{A}'\}.$$
Clearly, for a linear $\typef{A}$-algebra $\algf{A}$ such that $\elthr[\slangn{\typef{A}}{\doubleprefix{\prefixf{P}}}]{\algf{A}}=\thrf{T}$ we have
$$\setf{A}'''=\{\elthr[\slangn{\typef{B}\typeplus\constf{c}}{\prefixf{P}'}]{\fextfquotfr^{\tembf{l}\typeplus\constf{c},\constf{q}}(\algf{A}')}\mid \mbox{ $\algf{A}'$ is a constant extensions by  $\{\constf{q},\constf{c}\}$ of $\algf{A}$}\}.$$

Clearly, for a normal type $\typef{C}$, quantifier prefix $\prefixf{W}$ and two subsets $\thrf{U}_1,\thrf{U}_2\subset\slangnc{\typef{C}}{\prefixf{W}}$ we can effectively construct the set $\cmpf{LinProdThTr}^{\typef{C}}(\prefixf{W},\thrf{U}_1,\thrf{U}_2)\subset\slangnc{\typef{C}}{\prefixf{W}}$ such that for  linear $\typef{C}$-algebras $\algf{C}_1$, $\algf{C}_2$, $\elthr[\slangn{\typef{C}}{\prefixf{W}}]{\algf{C}_1}=\thrf{U}_1$ and  $\elthr[\slangn{\typef{C}}{\prefixf{W}}]{\algf{C}_2}=\thrf{U}_2$ we have  $\elthr[\slangn{\typef{C}}{\prefixf{W}}]{\algf{C}_1\linprodfr\algf{C}_2}=\cmpf{LinProdThTr}^{\typef{C}}(\prefixf{W},\thrf{U}_1,\thrf{U}_2)$ ; here we use Lemma \ref{lin_prod_tr}. Moreover, for a normal type $\typef{C}$, quantifier prefix $\prefixf{W}$ and subsets $\thrf{U}_1,\thrf{U}_2,\thrf{U}_3,\thrf{U}_4\subset\slangnc{\typef{C}}{\prefixf{W}}$ we have $$\thrf{U}_1\subset\thrf{U}_3\foand\thrf{U}_2\subset\thrf{U}_4\Implication \cmpf{LinProdThTr}^{\typef{C}}(\prefixf{W},\thrf{U}_1,\thrf{U}_2)\subset\cmpf{LinProdThTr}^{\typef{C}}(\prefixf{W},\thrf{U}_3,\thrf{U}_4).$$

We consider an infinite sequence:
\begin{itemize}
\item $\setf{C}_1=\setf{B}'$;
\item $\setf{C}_{\natf{i}+1}=\{\cmpf{LinProdThTr}^{\typef{B}\typeplus\constf{c}}(\prefixf{P}',\thrf{U}_1,\thrf{U}_2)\mid \thrf{U}_1\in\setf{A}''', \thrf{U}_2\in\setf{C}_{\natf{i}}\}$, for $\natf{i}\ge 1$.
\end{itemize} 
We denote by $\setf{D}$ the set $\bigcup\limits_{\natf{i}\ge 1} \setf{C}_{\natf{i}}$. Clearly, for a linear $\typef{A}$-algebra $\algf{A}$ such that $\elthr[\slangn{\typef{A}}{\doubleprefix{\prefixf{P}}}]{\algf{A}}=\thrf{T}$ we have $$\setf{C}_{\natf{i}}=\{\elthr[\slangn{\typef{B}\typeplus\constf{c}}{\prefixf{P}'}]{\sqext_{\overline{\algf{S}}}}\mid \mbox{ $\overline{\algf{S}}$ is $(\tembf{l},\algf{A},\constf{q},\{\constf{c}\})$-extension sequence of the length $\natf{i}$}\}$$
and $$\begin{aligned}\setf{D}&=\{\elthr[\slangn{\typef{B}\typeplus\constf{c}}{\prefixf{P}'}]{\sqext_{\overline{\algf{S}}}}\mid \mbox{ $\overline{\algf{S}}$ is $(\tembf{l},\algf{A},\constf{q},\{\constf{c}\})$-extension sequence}\}\\ &=\{\elthr[\slangn{\typef{B}\typeplus\constf{c}}{\prefixf{P}'}]{\algf{B}}\mid \mbox{ $\algf{B}$ is a constant extension of $\fextfr^{\tembf{l}}(\algf{A})$ by $\constf{c}$}\}\end{aligned}$$

Sets $\setf{C}_{\natf{i}}$ are subsets of  $\slangnc{\prefixf{P}'}{\typef{B}\typeplus\constf{c}}$. Obviously, if for some $\natf{i},\natf{j}$ we have $\setf{C}_{\natf{i}}=\setf{C}_{\natf{j}}$ then $\setf{C}_{\natf{i}+1}=\setf{C}_{\natf{j}+1}$. Hence  $$\setf{D}=\bigcup\limits_{\natf{i}\ge 1} \setf{C}_{\natf{i}}=\bigcup\limits_{1\le \natf{i}\le \natf{k}}\setf{C}_{\natf{i}},$$ where $\natf{k}=2^{|\slangnc{\prefixf{P}'}{\typef{B}\typeplus\constf{c}}|}$. Therefore we can calculate $\setf{D}$.

From the set $\setf{D}$ we construct the set $\setf{D}'$ of all $\exists \fovarf{x} \foflf{t}$ such that $\foflf{t}[\fovarf{x}/\constf{q}]$ lies in some element of $\setf{D}$. The resulting set $\cmpf{FrExtThTr}^{\tembf{l}}(\prefixf{P},\thrf{T})$ is the set of all $$\prflf{r}[\prvarf{x}_1,\ldots,\prvarf{x}_{\natf{n}}/\exists \fovarf{x} \foflf{t}_1,\ldots, \exists \fovarf{x} \foflf{t}_{\natf{n}}]$$ (where $\prflf{r}(\prvarf{x}_1,\ldots,\prvarf{x}_{\natf{n}})$ is positive propositional formula in disjunctive normal form and $\foflf{t}_{\natf{i}}(\fovarf{x})$ are pairwise different formulas from $\slangnc{\typef{A}}{\prefixf{P}'}$ such that $\prflf{r}$ is true under the substitution $$\prvarf{x}_{\natf{i}}\;\leftarrow\; (\exists \fovarf{x}\foflf{t}_{\natf{i}}) \in \setf{D}'.$$
Clearly, such a  $\cmpf{FrExtThTr}^{\tembf{l}}(\prefixf{P},\thrf{T})$ satisfies the condition of the lemma. Obviously, our additional induction assumption is satisfied too.  
\end{proof}

Using Lemma \ref{formula_normalization2} and Lemma \ref{frext_translation} we obtain
\begin{corollary} Suppose $\tembf{l}\colon \typef{A}\to \typef{B}$ is a normal type embedding, $\algf{A}$ is an $\typef{A}$-algebra, and the theory $\elthr{\algf{A}}$ is decidable. Then the theory $\elthr{\fextfr^{\tembf{l}}(\algf{A})}$ is decidable.
\end{corollary}

Because, of the correspondence between notions of normal type embedding and simple final type embedding we can conclude that
\begin{corollary} \label{fext_dec_tr} Suppose $\tembf{l}\colon \typef{A}\to \typef{B}$ is a simple final type embedding, $\algf{A}$ is an $\typef{A}$-algebra, and the theory $\elthr{\algf{A}}$ is decidable. Then the theory $\elthr{\fextfr^{\tembf{l}}(\algf{A})}$ is decidable.
\end{corollary}

We denote by $\sptype{\onf{a}}$ the type $((\onf{a},<),\emptyset)$, where $(\onf{a},<)$ is an ordinal $\onf{a}$ with standard ordering. We denote by $\spemb{\onf{a}}{\onf{b}}\colon \sptype{\onf{a}}\to \sptype{\onf{b}+\onf{a}}$ the type embedding that maps an operator symbol $\boxos[\onf{g}]\in\sptype{\onf{a}}$ to $\boxos[\onf{b}+\onf{g}]$. Clearly all $\spemb{\onf{a}}{\onf{b}}$ are simple final type embeddings. We denote by $\algf{F}_0$ the two element Boolean algebra. Note that $\algf{F}_0$ is $\sptype{0}$-algebra. We denote by $\algf{F}_{\onf{a}}$ the $\sptype{\onf{a}}$-algebra $\fextfr^{\spemb{0}{\onf{a}}}(\algf{F}_0)$. Clearly an algebra $\algf{F}_{1+\onf{a}}$ is isomorphic to $\fextfr^{\spemb{\onf{a}}{1}}(\algf{F}_{\onf{a}})$
In \cite{ArtBek93} S.N.~Artemov and L.D.~Beklemishev have proved that  $\elthr{\algf{F}_{1}}$ is decidable. Using Corollary \ref{fext_dec_tr}, Lemma \ref{fextfr_comp}, and mentioned theorem from \cite{ArtBek93} we conclude that

\begin{theorem} For every $\natf{n}$ the elementary theory $\elthr{\algf{F}_{\natf{n}}}$ is decidable.
\end{theorem}

\section{Some Syntactical Facts} \label{syntactical_facts} 
The aim of the section is to prove Lemmas \ref{box_switch_imp} and \ref{sqftr}. In the section we will assume that a reader is familiar with paper ``Kripke semantics for provability logic \thrf{GLP}'' by L.D.~Beklemishev \cite{Bek07}. Moreover, in the section we will use the terminology  of \cite{Bek07} rather than the terminology of the other parts of the present paper.

We briefly remind the main notions and results of \cite{Bek07}. The polymodal provability logic $\thrf{GLP}$ were considered as a polymodal logic with modalities indexed by natural numbers. In par with the logic $\thrf{GLP}$  there were considered a weaker logic $\thrf{J}$ in the same language (we don't give an axiomatization of $\thrf{J}$ here, we give a complete semantics for this logic below) .  Kripke models with  accessibility relations $R_{\natf{i}}$ for all $\natf{i}\ge \natf{m}$  were called $\natf{m}$-models. The rank $\mathrm{rk}_{\natf{m}}(\mathcal{A})$ of an $\natf{m}$-model $\mathcal{A}$ is the minimal $\natf{n}\ge 0$ such that for all $\natf{k}\ge \natf{m}+\natf{n}$ the relation $R_{\natf{k}}$ is empty  in $\mathcal{A}$; $\mathrm{rk}_{\natf{m}}$ is a partial function from $\natf{m}$-models to natural numbers. The notion of stratified  were important in \cite{Bek07}. The notion of {\it hereditarily rooted finite stratified $\natf{m}$-model} $\mathcal{A}$ can be given by induction on rank (for every such a $\natf{m}$-model $\mathcal{A}$ the rank $\mathrm{rk}_{\natf{m}}(\mathcal{A})$ is a finite number) as following: 
\begin{enumerate}
\item $\natf{m}$-model $\mathcal{A}$ with $\mathrm{rk}_{\natf{m}}(\mathcal{A})=0$ is hereditarily rooted finite stratified  if all $R_{\natf{i}}$ are empty and there is exactly one point in  $\mathcal{A}$;
\item  $\natf{m}$-model $\mathcal{A}$ with $\mathrm{rk}_{\natf{m}}(\mathcal{A})=\natf{n}+1$ is hereditarily rooted finite stratified  if
\begin{enumerate}
\item \label{plane_p}points of $\mathcal{A}$ can be separated on $(\natf{m}+1)$-submodels $\alpha_1,\ldots,\alpha_{\natf{n}}$ such that 
\begin{enumerate}
\item for every $\alpha_{\natf{i}}$ the restriction of $R_{\natf{m}}$ on points of $\alpha_{\natf{i}}$ is empty,
\item for every different $\alpha_{\natf{i}},\alpha_{\natf{j}}$, $\natf{k}>\natf{m}$, $x\in\alpha_{\natf{i}}$, and $y\in\alpha_{\natf{j}}$ the point $y$ isn't $R_{\natf{k}}$-accessible from $x$ in $\mathcal{A}$,
\item all $\alpha_{\natf{i}}$ are finite hereditarily rooted stratified models,
\item $\alpha_1,\ldots,\alpha_{\natf{n}}$ are called $(\natf{m}+1)$-planes,
\end{enumerate}
\item $R_{\natf{m}}$ in $\mathcal{A}$ is strict partial order on $(\natf{m}+1)$-planes,
\item in $\mathcal{A}$ there exist the lowest $(\natf{m}+1)$-plane.
\end{enumerate}
\end{enumerate}

For an $\natf{m}$-model $\mathcal{A}$ there is at most one separation on $(\natf{m}+1)$-models $\alpha_1,\ldots,\alpha_{\natf{n}}$ that  satisfies properties from \ref{plane_p}.

A point $a$ of a hereditarily rooted finite stratified $\natf{m}$-model $\mathcal{A}$ is the {\it hereditary root of $\mathcal{A}$} if either $a$ is the only point of $\mathcal{A}$ or $a$ is the hereditary root of the root plane $\alpha$ of $\mathcal{A}$.

The logic $\thrf{J}$ is complete with respect to the class of all hereditarily rooted finite stratified models.

 In \cite{Bek07} there was defined blowup operation $\mathcal{A}\longmapsto \mathcal{A}^{(\natf{n})}$ that maps a $(\natf{m}+1)$-model to $\natf{m}$-model. We give the definition by induction on the number of planes in a model $\mathcal{A}$. Suppose $\alpha$ is the root $(\natf{m}+2)$-plane of $\mathcal{A}$. Suppose $(\natf{m}+2)$-planes $\alpha_1,\ldots,\alpha_{\natf{k}}$ are all immediate successors of $\alpha$. For $1\le\natf{i}\le\natf{k}$ we denote by $\mathcal{A}_{\natf{i}}$ the cone from $\alpha_{\natf{i}}$ in $\mathcal{A}$. The model $$\mathcal{A}^{(\natf{n})}=(\bigsqcup\limits_{1\le\natf{i}\le \natf{k}}\underbrace{\mathcal{A}_{\natf{i}}^{(\natf{n})}+\ldots+\mathcal{A}_{\natf{i}}^{(\natf{n})}}\limits_{\mbox{$\natf{n}$-times}}) + \{\mathcal{A}\},$$ where $\{\mathcal{A}\}$ denotes $\mathcal{A}$ enriched by the empty $R_{\natf{m}}$ and for $\natf{m}$-models $\mathcal{C}$ and $\mathcal{B}$ the model $\mathcal{C}+\mathcal{B}$ is $\mathcal{C} \sqcup \mathcal{B}$ with $R_{\natf{m}}$ enriched by all $xR_{\natf{m}} y$, for $x\in\mathcal{B}$ and $y\in \mathcal{C}$. 

 Also in \cite{Bek07} there were defined the operation $\mathcal{A}\longmapsto\mathfrak{B}_{\natf{n}}(\mathcal{A})$ that maps finite hereditary rooted stratified $\natf{m}$-models to hereditarily rooted finite stratified $\natf{m}$-models. For a $\natf{m}$-model $\mathcal{A}$ we define the $\natf{m}$-model $\mathfrak{B}_{\natf{n}}(\mathcal{A})$ by induction on  the rank of $\mathcal{A}$. Suppose $\mathcal{A}$ is separated on $(\natf{m}+1)$-planes $\alpha_1,\ldots,\alpha_{\natf{k}}$. Then $\mathfrak{B}_{\natf{n}}(\mathcal{A})$ is the disjoint union $\bigsqcup\limits_{1\le\natf{i}\le\natf{k}}\mathfrak{B}_{\natf{n}}(\alpha_{\natf{i}})^{(\natf{n})}$ with $R_{\natf{m}}$ enriched by all $xR_{\natf{m}}y$ such that $x\in\alpha_{\natf{i}}$, $y\in\alpha_{\natf{j}}$, $1\le\natf{i},\natf{j}\le\natf{k}$, and $\alpha_{\natf{i}}R_{\natf{m}}\alpha_{\natf{j}}$ in $\mathcal{A}$.

$\mathit{dp}(\prflf{f})$  denotes the modal depth of a formula $\prflf{f}$.
The following straightforward corollary of  \cite[Lemma 7.6]{Bek07} holds
\begin{lemma} For $\natf{n}\le\natf{m}$ and a hereditarily rooted finite stratified model $\mathcal{A}$ the model $\mathfrak{B}_{\natf{n}}(\mathcal{A})$ and $\mathfrak{B}_{\natf{m}}(\mathcal{A})$ satisfies the same formulas $\prflf{f}$ with $\mathit{dp}(\prflf{f})\le\natf{n}$.
\end{lemma}
 From the lemma above and \cite[Theorem 4]{Bek07} we straightforward obtain the following completeness result for $\thrf{GLP}$ 
\begin{theorem} \label{glp_completeness} For a $\thrf{GLP}$-formula $\prflf{f}$ and number $\natf{m}\ge\mathit{dp}(\prflf{f})$ the following sentences are equivalent
\begin{enumerate}
\item $\thrf{GLP}\vdash\prflf{f}$; 
\item for every hereditarily rooted finite stratified $0$-model $\mathcal{A}$ we have $\mathfrak{B}_{\natf{m}}(\mathcal{A})\models \prflf{f}$.
\end{enumerate}\end{theorem}

\begin{lemma} \label{suitable_model} Suppose $\prflf{f}$ is a formula and $\thrf{GLP}\nvdash \prflf{f}$. Then  for every $\natf{n}$ there exists a  hereditarily rooted finite stratified  $0$-model $\mathcal{B}$ such that $\mathfrak{B}_{\natf{n}}(\mathcal{B}),a\nVdash \prflf{f}$, where $a$ is the hereditary root of $\mathfrak{B}_{\natf{n}}(\mathcal{B})$, but for every point $x\ne a$ from a non-root $1$-plane of $\mathcal{B}$ we have $\mathfrak{B}_{\natf{n}}(\mathcal{B}),x\Vdash \prflf{f}$.
\end{lemma}
\begin{proof}  $M(\prflf{f})$ is the conjunction of all  $\boxm[\natf{m}_{\natf{i}}]\prflf{p}\primp\boxm[\natf{m}_{\natf{i}}+1]\prflf{p}$, where $\boxm[\natf{m}_{\natf{i}}]\prflf{p}$ is a subformula of $\prflf{f}$. $M^{+}(\prflf{f})$ is the conjunction
$$M(\prflf{f})\prand \boxm[0]M(\prflf{f})\prand\ldots\prand \boxm[\natf{k}]M(\prflf{f}),$$
where $\natf{k}$ is the maximum of all $\natf{m}_{\natf{i}}+1$. 

 Because $\thrf{GLP}\vdash M^{+}(\prflf{f})$, we have $\thrf{J}\nvdash  M^{+}(\prflf{f})\primp\prflf{f}$. Therefore there exists hereditarily rooted finite stratified $0$-model $\mathcal{A}$ such that $\mathcal{A}\not\models M^{+}(\prflf{f})\primp\prflf{f}$. Clearly, we can find $x\in \mathcal{A}$ such that $\mathcal{A},x\nvdash  M^{+}(\prflf{f})\primp\prflf{f}$ but in any $y$ accessible from $x$ by any $R_i$ we have $\mathcal{A},y\Vdash M^{+}(\prflf{f})\primp\prflf{f}$. Now we consider the submodel $\mathcal{B}$ of $\mathcal{A}$ that consists of all points accessible from $x$ by some $R_i$. Clearly, $\mathcal{B}\models M(\prflf{f})$ 

\cite[Lemma 9.3]{Bek07} states that if a hereditarily rooted finite stratified $\natf{0}$-model $\mathcal{C}\models M(\prflf{p})$ then for any $x$ from $\mathfrak{B}_{\natf{n}}(\mathcal{C})$ and a subformula $\prflf{t}$ of $\prflf{p}$ we have $$\mathfrak{B}_{\natf{n}}(\mathcal{C}),x\Vdash\prflf{t} \iff \mathcal{C},\pi^{\star}(x)\Vdash \prflf{t},$$ where $\pi^{\star}\colon \mathfrak{B}_{\natf{n}}(\mathcal{C})\to \mathcal{C}$ is the natural projection (we don't give the definition of $\pi^{\star}$ here it is given in \cite{Bek07} just above \cite[Lemma 9.3]{Bek07}).

  Using \cite[Lemma 9.3]{Bek07} we conclude that in the hereditary root $a$ of $\mathfrak{B}_{\natf{n}}(\mathcal{B})$  we have $\mathfrak{B}_{\natf{n}}(\mathcal{B}),a\nVdash \prflf{f}$ but in any $x\ne a$ from $\mathfrak{B}_{\natf{n}}(\mathcal{B})$ we have $\mathfrak{B}_{\natf{n}}(\mathcal{B}),x \Vdash\prflf{f}$.  
\end{proof}

Using Lemma \ref{mod_add_cons} we reformulate Lemma \ref{box_switch_imp}.
{
\renewcommand{\thelemma}{\ref{box_switch_imp}}
\begin{lemma}Suppose $\prflf{f}$ and $\prflf{p}$ are formulas without $\boxm[0]$. Then 
$$\thrf{GLP}\vdash \boxm[1]\prflf{p}\primp \prflf{f}\iff \thrf{GLP}\vdash \boxm[0]\prflf{p}\primp \prflf{f}.$$
\end{lemma}
\addtocounter{theorem}{-1}
}
\begin{proof} $(\Rightarrow)$: Holds because $\thrf{GLP}\vdash \boxm[0]\prflf{p}\primp\boxm[1]\prflf{p}$.

$(\Leftarrow)$: We denote by $\natf{n}$ the modal depth $\mathit{dp}( \boxm[1]\prflf{p}\primp \prflf{f})=\mathit{dp}( \boxm[0]\prflf{p}\primp \prflf{f})$.  We'll prove that
$$\thrf{GLP}\nvdash \boxm[1]\prflf{p}\primp \prflf{f}\Implication \thrf{GLP}\nvdash \boxm[0]\prflf{p}\primp \prflf{f},$$
using Theorem \ref{glp_completeness}. Suppose $\thrf{GLP}\nvdash \boxm[1]\prflf{p}\primp \prflf{f}$.  We'll construct a $0$-model $\mathcal{B}$ such that $\mathfrak{B}_{\natf{n}}(\mathcal{B})\not\models\boxm[0]\prflf{p}\primp\prflf{f}$ .

 From Lemma \ref{suitable_model} it follows that we have a  hereditarily rooted finite stratified  $0$-model $\mathcal{C}$ such that for the hereditarily root $a$ of $\mathfrak{B}_{\natf{n}}(\mathcal{C})$ we have $\mathfrak{B}_{\natf{n}}(\mathcal{C}),a\nVdash \boxm[1]\prflf{p}\primp \prflf{f}$. We consider the root $1$-plane of $\mathcal{C}$. We denote this plane by $\mathcal{A}$. Clearly, $\mathfrak{B}_{\natf{n}}(\mathcal{A}),a\nVdash \boxm[1]\prflf{p}\primp \prflf{f}$. Because $\mathfrak{B}_{\natf{n}}(\mathcal{A}),a\Vdash \boxm[1]\prflf{p}$, we have $\mathfrak{B}_{\natf{n}}(\mathcal{A}),x\Vdash \prflf{p}$, for every $x$ from non-root $2$-plane of $\mathfrak{B}_{\natf{n}}(\mathcal{A})$. The root of $\mathfrak{B}_{\natf{n}}(\{\mathcal{A}\})$ is a copy of $\mathfrak{B}_{\natf{n}}(\mathcal{A})$. We consider $b\in \mathfrak{B}_{\natf{n}}(\{\mathcal{A}\})$ that corresponds to $a$ in that copy. Because there are no $\boxm[0]$ in $\prflf{f}$ and $\mathfrak{B}_{\natf{n}}(\mathcal{A}),a\nVdash\prflf{f}$, we have $\mathfrak{B}_{\natf{n}}(\{\mathcal{A}\}),b\nVdash\prflf{f}$. Clearly, every non-root $1$-plane of $\mathfrak{B}_{\natf{n}}(\{\mathcal{A}\})$ is a copy of a proper cone of $\mathfrak{B}_{\natf{n}}(\mathcal{A})$. Hence for any $x$ from a non-root $1$-plane of $\mathfrak{B}_{\natf{n}}(\{\mathcal{A}\})$ we  have $\mathfrak{B}_{\natf{n}}(\{\mathcal{A}\}),x\Vdash \prflf{p}$. We conclude that $\mathfrak{B}_{\natf{n}}(\{\mathcal{A}\}),b\nVdash\boxm[0]\prflf{p}\primp\prflf{f}$. We put $\mathcal{B}=\{\mathcal{A}\}$.\end{proof}

  For a propositional variable $\prvarf{x}$ we define $\sqfatrs{\prvarf{x}}\colon \lang{\thrf{GLP}}\to\lang{\thrf{GLP}}$:
\begin{itemize}
\item $\sqfatr{\prvarf{x}}{\top}= \boxo[0]\prvarf{x}\pror\top$;
\item $\sqfatr{\prvarf{x}}{\bot}= \boxo[0]\prvarf{x}\pror\bot$;
\item $\sqfatr{\prvarf{x}}{\prvarf{y}}= \boxo[0]\prvarf{x}\pror\prvarf{y}$, for a propositional variable $\prvarf{y}$;
\item $\sqfatr{\prvarf{x}}{\prflf{f}\pror\prflf{p}}= \boxo[0]\prvarf{x}\pror (\sqfatr{\prvarf{x}}{\prflf{f}}\pror\sqfatr{\prvarf{x}}{\prflf{p}})$, for $\thrf{GLP}$-formulas $\prflf{f},\prflf{p}$;
\item $\sqfatr{\prvarf{x}}{\prflf{f}\prand\prflf{p}}= \boxo[0]\prvarf{x}\pror (\sqfatr{\prvarf{x}}{\prflf{f}}\prand\sqfatr{\prvarf{x}}{\prflf{p}})$, for $\thrf{GLP}$-formulas $\prflf{f},\prflf{p}$;
\item $\sqfatr{\prvarf{x}}{\prflf{f}\primp\prflf{p}}= \boxo[0]\prvarf{x}\pror (\sqfatr{\prvarf{x}}{\prflf{f}}\primp\sqfatr{\prvarf{x}}{\prflf{p}})$, for $\thrf{GLP}$-formulas $\prflf{f},\prflf{p}$;
\item $\sqfatr{\prvarf{x}}{\prnot \prflf{f}}= \boxo[0]\prvarf{x}\pror (\prnot \sqfatr{\prvarf{x}}{\prflf{f}})$, for a $\thrf{GLP}$-formula $\prflf{f}$;
\item $\sqfatr{\prvarf{x}}{\boxm[\natf{n}]{ \prflf{f}}}= \boxo[0]\prvarf{x}\pror (\boxm[\natf{n}]\sqfatr{\prvarf{x}}{\prflf{f}})$, for a $\thrf{GLP}$-formula $\prflf{f}$ and a natural number $\natf{n}$.
\end{itemize}

We have an equivalent form of Lemma \ref{sqftr}.
{
\renewcommand{\thelemma}{\ref{sqftr}}
\begin{lemma} Suppose $\prflf{f}$,$\prflf{p}$ are formulas from $\lang{\thrf{GLP}}$, and $\prvarf{x}$ is a propositional variable such that $\prvarf{x}$ doesn't occur in $\prflf{f}$, $\boxm[0]$ doesn't occur in $\prflf{p}$, and $\thrf{GLP}\nvdash (\prflf{p}\prand \boxm[1]\prflf{p})\primp \prvarf{x}$. Then 
$$\thrf{GLP}\vdash (\prflf{p}\prand \boxm[0]\prflf{p})\primp\prflf{f}\iff \thrf{GLP}\vdash (\prflf{p}\prand \boxm[0]\prflf{p})\primp\sqfatr{\prvarf{x}}{\prflf{f}}.$$
\end{lemma}
\addtocounter{theorem}{-1}
}
\begin{proof} $\Rightarrow$: We claim that for  all $\prflf{x}$ if $\thrf{GLP}+\prflf{f}\vdash\sqfatr{\prvarf{x}}{\prflf{x}}$ then $\thrf{GLP}+\prflf{f}\vdash\prflf{x}$; obviously, $(\Rightarrow)$ follows from Lemma \ref{glp_deduction} and the claim. We prove the claim by induction on the length of proof of $\prflf{x}$. Simple check shows that induction hypothesis holds for the axioms. For induction step the induction hypothesis can be proved easily for both cases of the last inference rule. 

$\Leftarrow$:   We denote by $\natf{n}$ the modal depth $\mathit{dp}( \boxm[1]\prflf{p}\primp \prflf{f})=\mathit{dp}( \boxm[0]\prflf{p}\primp \prflf{f})$. We will prove
$$\thrf{GLP}\nvdash (\prflf{p}\prand \boxm[0]\prflf{p})\primp\prflf{f}\Implication \thrf{GLP}\nvdash (\prflf{p}\prand \boxm[0]\prflf{p})\primp\sqfatr{\prvarf{x}}{\prflf{f}}.$$
Suppose $\thrf{GLP}\nvdash (\prflf{p}\prand \boxm[0]\prflf{p})\primp\prflf{f}$.  From Lemma \ref{box_switch_imp} it follows that $\thrf{GLP}\nvdash (\prflf{p}\prand \boxm[0]\prflf{p})\primp \prvarf{x}$. By Lemma \ref{suitable_model} we obtain a  hereditarily rooted finite stratified  $0$-model $\mathcal{B}$ such that in the hereditary root $b$ is of $\mathfrak{B}_{\natf{n}}(\mathcal{B})$ we have $\mathfrak{B}_{\natf{n}}(\mathcal{B}),b\nvdash (\prflf{p}\prand \boxm[0]\prflf{p})\primp \prvarf{x}$, but for every point $w\ne a$ from non-root $1$-plane of $\mathcal{B}$ we have $\mathfrak{B}_{\natf{n}}(\mathcal{B}),w\Vdash (\prflf{p}\prand \boxm[0]\prflf{p})\primp \prvarf{x}$. Clearly, $\mathfrak{B}_{\natf{n}}(\mathcal{B}),b\Vdash \prflf{p}\prand \boxm[0]\prflf{p}$ and $\thrf{GLP}\vdash \boxm[0]\prflf{p}\primp \boxm[\natf{k}]\prflf{p}$ for every $\natf{k}$. Hence from Theorem \ref{glp_completeness} it follows that $\mathfrak{B}_{\natf{n}}(\mathcal{B})\models \prflf{p}$. Hence $\mathfrak{B}_{\natf{n}}(\mathcal{B}),b\nVdash \prvarf{x}$ and in all $w\ne b$ from  $\mathfrak{B}_{\natf{n}}(\mathcal{B})$ there is $\mathfrak{B}_{\natf{n}}(\mathcal{B}),w\Vdash \prvarf{x}$.

 By Theorem \ref{glp_completeness} we have a  hereditarily rooted finite stratified  $0$-model $\mathcal{A}$ such that $\mathfrak{B}_{\natf{n}}(\mathcal{A})\not \models (\prflf{p}\prand \boxm[0]\prflf{p})\primp\prflf{f}$. Obviously, we can choose such a  $\mathcal{A}$ that $\mathfrak{B}_{\natf{n}}(\mathcal{A}) \models \prnot\prvarf{x}$. We consider model  $\mathcal{C}=\mathcal{B}+\mathcal{A}$. We consider $\mathfrak{B}_{\natf{n}}(\mathcal{A})$  as a submodel of $\mathfrak{B}_{\natf{n}}(\mathcal{C})$ in a natural way.  Obviously, for a point $w$ from $\mathfrak{B}_{\natf{n}}(\mathcal{A})$ we have
$$\mathfrak{B}_{\natf{n}}(\mathcal{C}),w\Vdash \boxm[0]x \iff w\in\mathfrak{B}_{\natf{n}}(\mathcal{A}).$$
By trivial induction on subformulas of $\prflf{t}$ we show that for a point $w$ from $\mathfrak{B}_{\natf{n}}(\mathcal{A})$ and a formula $\prflf{t}$ we have
$$\mathfrak{B}_{\natf{n}}(\mathcal{A}),w\Vdash \prflf{t} \iff \mathfrak{B}_{\natf{n}}(\mathcal{C}),w\Vdash \sqfatr{\prvarf{x}}{\prflf{t}}.$$ 
Obviously,  for a point $w$ from $\mathfrak{B}_{\natf{n}}(\mathcal{A})$ we have
$$\mathfrak{B}_{\natf{n}}(\mathcal{A}),w\Vdash \prflf{p} \iff \mathfrak{B}_{\natf{n}}(\mathcal{C}),w\Vdash \prflf{p}$$
and because $\mathfrak{B}(\mathcal{B})\models \prflf{p}$
$$\mathfrak{B}_{\natf{n}}(\mathcal{A}),w\Vdash \boxm[0]\prflf{p} \iff \mathfrak{B}_{\natf{n}}(\mathcal{C}),w\Vdash \boxm[0]\prflf{p}.$$
Hence for a point $w$ from $\mathfrak{B}_{\natf{n}}(\mathcal{A})$ we have
$$\mathfrak{B}_{\natf{n}}(\mathcal{A}),w\Vdash (\prflf{p}\prand \boxm[0]\prflf{p})\primp\prflf{f} \iff \mathfrak{B}_{\natf{n}}(\mathcal{C}),w\Vdash (\prflf{p}\prand \boxm[0]\prflf{p})\primp\sqfatr{\prvarf{x}}{\prflf{f}}.$$
Therefore $\mathfrak{B}_{\natf{n}}(\mathcal{C})\not\models (\prflf{p}\prand \boxm[0]\prflf{p})\primp\sqfatr{\prvarf{x}}{\prflf{f}}$. Finally, we conclude that  $\thrf{GLP}\nvdash (\prflf{p}\prand \boxm[0]\prflf{p})\primp\sqfatr{\prvarf{x}}{\prflf{f}}$.
\end{proof}


\bibliographystyle{plain}
\bibliography{bibliography}
\end{document}